%% file: BK-algebras.tex
\documentclass[a4paper]{amsart}
\input{./defs.tex}

\hypersetup{
  colorlinks = false,
  urlcolor = blue,
  pdfauthor = {Michael Ehrig, Catharina Stroppel and Daniel Tubbenhauer},
  pdfkeywords = {},
  pdftitle = {The Blanchet-Khovanov algebras},
  pdfsubject = {},
  pdfpagemode = UseNone,
  bookmarksopen = true,
  bookmarksopenlevel = 3,
  pdfdisplaydoctitle = true,
  linktocpage=true,
}

\begin{document}
\vbadness=10001
\hbadness=10001
\title[The Blanchet-Khovanov algebras]{The Blanchet-Khovanov algebras}
\author{Michael Ehrig}
\address{M.E.: School of Mathematics \& Statistics, Carslaw Building, University of Sydney, NSW 2006, Australia}
\email{michael.ehrig@sydney.edu.au}

\author{Catharina Stroppel}
\address{C.S.: Mathematisches Institut, Universit\"at Bonn, Endenicher Allee 60, Room 4.007, 53115 Bonn, Germany}
\email{stroppel@math.uni-bonn.de}

\author{Daniel Tubbenhauer}
\address{D.T.: Mathematisches Institut, Universit\"at Bonn, Endenicher Allee 60, Room 1.003, 53115 Bonn, Germany}
\email{dtubben@math.uni-bonn.de}

\dedicatory{Dedicated to Christian Blanchet's sixtieth birthday}


\begin{abstract}
Blanchet introduced certain singular cobordisms 
to fix the functoriality of Khovanov homology.
In this paper we introduce graded  
algebras consisting of such singular cobordisms {\`a} la Blanchet. 
As the main result we explicitly describe these algebras in algebraic terms using the combinatorics of arc diagrams.
\end{abstract}

\maketitle

\vspace*{-1.1cm}

\tableofcontents

\vspace*{-1.1cm}
%
%
\section{Introduction}\label{sec:intro}
\input{res/1-intro.tex}
\noindent \textbf{Acknowledgements:} We like to 
thank David Rose and Nathalie Wahl 
for helpful conversations, 
and 
Paul Wedrich and the referee for helpful comments. 
M.E. and D.T. thank the whiteboard 
in their office for many helpful illustrations.
%
\section{\texorpdfstring{$\gltwo$}{gl2}-foams and \texorpdfstring{$\gltwo$}{gl2}-web algebras}\label{subsec:foamy}
\input{res/2-webs-and-foams.tex}
\section{Blanchet-Khovanov algebras}\label{subsec:genKh}
\input{res/3-Khovanov.tex}
\section{Equivalences}\label{sec:fromKhtofoams}
\input{res/4-fromKhtofoams.tex}

%


\bibliographystyle{plainurl}
\bibliography{BK-algebras}
\end{document}

%% file: defs.tex
\usepackage{calligra}
\usepackage{fontenc}

\usepackage{mathtools}
\mathtoolsset{mathic}
\usepackage[final]{microtype}

\usepackage{latexsym,exscale,enumitem,amsfonts,amssymb,mathtools}
\usepackage{amsmath,amsthm,amsfonts,amssymb,amscd,stmaryrd,textcomp,xcolor,bbm}
\usepackage[normalem]{ulem}
\usepackage{tabu}
\usepackage{hhline}
\usepackage{thmtools}
\usepackage{etex}

\usepackage{hyperref}

\usepackage[all]{xy}
\SelectTips{cm}{}

\usepackage{tikz}
\tikzset{anchorbase/.style={baseline={([yshift=-0.5ex]current bounding box.center)}}}
\usetikzlibrary{decorations.markings}
\usetikzlibrary{decorations.pathreplacing}
\usetikzlibrary{arrows,shapes,positioning}
\tikzstyle directed=[postaction={decorate,decoration={markings,
    mark=at position #1 with {\arrow{>}}}}]
\tikzstyle rdirected=[postaction={decorate,decoration={markings,
    mark=at position #1 with {\arrow{<}}}}]

\newcommand{\comm}[1]{}

\newcommand{\Z}{\mathbb{Z}}
\newcommand{\R}{\mathbb{R}}
\newcommand{\F}{\boldsymbol{\mathfrak{F}}}
\newcommand{\Ff}{p\boldsymbol{\mathrm{F}}}
\newcommand{\field}{\mathbb{K}}
\newcommand{\TQFTold}{\mathcal{Z}\!}
\newcommand{\TQFT}{\mathcal{T}\!}
\newcommand{\M}{\mathfrak{W}}
\newcommand{\Vect}{\field\text{-}\boldsymbol{\mathrm{Vect}}}

\newcommand{\osymbol}{\mathrm{o}}
\newcommand{\psymbol}{{\color{nonecolor}\mathrm{p}}}
\newcommand{\invo}[1]{\text{-}#1}

\newcommand{\sltwo}{\mathfrak{sl}_2}
\newcommand{\gltwo}{\mathfrak{gl}_2}

\newcommand{\Udot}{\dot{\mathrm{U}}_q(\mathfrak{gl}_{\infty})}
\newcommand{\Udott}{\dot{\mathrm{U}}_q(\mathfrak{gl}_{n})}

\newcommand{\onek}{\mathbf{1}_{\vec{k}}}
\newcommand{\onel}{\mathbf{1}_{\vec{l}}}
\newcommand{\oneinsert}[1]{\mathbf{1}_{#1}}

\newcommand{\foamcat}{\boldsymbol{\mathfrak{F}}}
\newcommand{\webalg}{\mathfrak{W}}

\newcommand{\webcat}{\mathfrak{W}\text{-}\boldsymbol{\mathrm{biMod}}^p_{\mathrm{gr}}}

\newcommand{\sltwowebcat}{\boldsymbol{\mathrm{Web}}}
\newcommand{\cupcat}{\boldsymbol{\mathrm{CM}}}

\newcommand{\web}{\boldsymbol{\mathrm{w}}}
\newcommand{\CUP}{\mathrm{Cup}}
\newcommand{\CAP}{\mathrm{Cap}}
\newcommand{\CUPRAY}{\mathrm{CupRay}}
\newcommand{\CAPRAY}{\mathrm{CapRay}}
\newcommand{\Modpgr}[1]{#1\text{-}\boldsymbol{\mathrm{biMod}}^p_{\mathrm{gr}}}

\newcommand{\Y}{\mathbbm{bl}}
\newcommand{\iY}{\overline{\mathbbm{bl}}}
\newcommand{\bY}{\mathbbm{bl}^{\diamond}}

\newcommand{\biMod}[1]{#1\text{-}\boldsymbol{\mathrm{biMod}}}

\newcommand{\Iso}[2]{\Phi_{#1}^{#2}}
\newcommand{\Isoo}[2]{\phi_{#1}^{#2}}

\newcommand{\Hom}{\mathrm{Hom}}
\newcommand{\twoHom}{2\mathrm{Hom}}

\newcommand{\twoEnd}{2\mathrm{End}}

\newcommand{\down}{{\scriptstyle\vee}}
\newcommand{\up}{{\scriptstyle\wedge}}
\newcommand{\downb}{\scriptstyle\boldsymbol{\vee}}
\newcommand{\upb}{\scriptstyle\boldsymbol{\wedge}}

\newcommand{\block}{\mathtt{bl}}
\newcommand{\bblock}{\mathtt{bl}^{\diamond}}

\newcommand{\length}{\mathrm{d}}

\newcommand{\arcalg}{\mathfrak{A}}
\newcommand{\Arcalg}{\mathfrak{A}^{\F}}

\newcommand{\op}{\operatorname}
\newcommand{\pos}{\mathrm{p}}

\newcommand{\dummy}{{\scriptstyle \bigstar}}

\theoremstyle{definition}
\newtheorem{theorem}{Theorem}[section]
\newtheorem{corollary}[theorem]{Corollary}
\newtheorem{lemma}[theorem]{Lemma}
\newtheorem{proposition}[theorem]{Proposition}
\newtheorem*{theoremnonumber}{Theorem}

\declaretheorem[style=definition,name=Example,qed=$\blacktriangle$,numberlike=theorem]{example}
\declaretheorem[style=definition,name=Definition,qed=$\blacktriangle$,numberlike=theorem]{definition}
\declaretheorem[style=definition,name=Remark,qed=$\blacktriangle$,numberlike=theorem]{remark}
\declaretheorem[style=definition,name=Convention,qed=$\blacktriangle$,numberlike=theorem]{convention}
\declaretheorem[style=definition,name=Definition,numberlike=theorem]{definitionn}
\declaretheorem[style=definition,name=Example,numberlike=theorem]{examplen}

\newcommand{\makeqed}{\hfill\ensuremath{\blacksquare}}
\newcommand{\makeqedtri}{\hfill\ensuremath{\blacktriangle}}
\newcommand{\qedmake}{\hfill\ensuremath{\square}}

\definecolor{mycolor}{rgb}{0.9,0,0}
\definecolor{colormy}{rgb}{0.75,0,0}
\definecolor{nonecolor}{rgb}{0.9,0,0.9}
\definecolor{cupgray}{gray}{0.4}
\definecolor{myblue}{rgb}{0,0,0.9}
\definecolor{mygreen}{rgb}{0,0.9,0}
\definecolor{myred}{rgb}{0.9,0,0}
\definecolor{myyellow}{rgb}{0.9,0.9,0}



%% file: res/1-intro.tex
For an arbitrary field $\field$ we consider the 
$\gltwo$\textit{-web algebra} $\webalg$, 
which we call \textit{web algebra} for short. 
(For the reason why we like to call it $\gltwo$-web algebra 
instead of $\sltwo$-web algebra 
see in the introduction of~\cite{EST2}.) 
This is a graded $\field$-algebra which 
naturally appears as an algebra of singular 
cobordisms. In particular, it is of topological origin. 
The underlying category of singular cobordisms was used in~\cite{Bla} 
by Blanchet to fix the functoriality of Khovanov homology. 
Its objects are certain trivalent graphs and its 
morphisms are singular cobordisms whose boundary are such trivalent graphs. 
We call these singular cobordisms $\gltwo$\textit{-foams} (or 
\textit{foams} for short). 
Note that Blanchet's category is a sign modified version of the 
original cobordism category which describes Khovanov homology 
and which was for instance used by 
Bar-Natan in his formulation of Khovanov homology, see~\cite{BN1}.
The fact that such a twist in the definition 
of Khovanov homology solves the functoriality leaves 
the question whether this could 
also be fixed \textit{algebraically} using the original 
construction of Khovanov involving his arc algebra, see~\cite{Khov}.

We therefore suggest here to study a certain signed 
(with highly non-trivial sign 
modifications) version $\Arcalg$ 
of Khovanov's original algebra, which we call 
the \textit{Blanchet-Khovanov algebra}. This is a graded $\field$-algebra 
defined diagrammatically via explicit multiplication rules on a distinguished 
set of basis vectors similar to the family of algebras from~\cite{BS1} 
or~\cite{ES2}.

The main result of the paper is then that $\Arcalg$ is an 
algebraic counterpart of $\webalg$:

\begin{theoremnonumber}
There is an equivalence of graded, $\field$-linear
$2$-categories 
\[
\boldsymbol{\Iso{}{}}\colon\webcat\stackrel{\cong}{\longrightarrow}\Modpgr{\Arcalg}
\] 
induced by an isomorphism of graded algebras
\begin{gather*}
\Iso{}{}\colon\webalg^{\circ}\to\Arcalg.
\end{gather*}
(Where $\webalg^{\circ}$ 
is a certain subalgebra of $\webalg$.)\makeqed
\end{theoremnonumber}

This provides a direct link between the topological 
and the algebraic point of view. As a consequence, computations
(which are hard to do in practice on the topological side) can be done 
on the algebraic side, whereas the associativity (a non-trivial fact 
on the algebraic side) is clear from the topological point of view.

\subsubsection*{The set-up in more details}\label{sub-intropart1}

In his pioneering work~\cite{Khov}, 
Khovanov introduced the so-called \textit{arc algebra} $H_m$. 
One of his main 
purposes was to extend his celebrated categorification of the 
Jones polynomial~\cite{Kho} to tangles. 
To a given tangle with $2m$ bottom boundary points and 
$2m^{\prime}$ top boundary points one associates a certain 
complex of graded 
$H_m$-$H_{m^{\prime}}$-bimodules. 
He showed that the chain homotopy
equivalence class of this complex is an invariant of the tangle. 
Moreover, taking the tensor product with a certain 
$H_m$-module from the left and a certain 
$H_{m^{\prime}}$-module from the right produces a complex which is still 
an invariant. On the level
of Grothendieck groups this invariant descends to the Kauffman bracket
of the tangle.

In this set-up it makes sense to ask if cobordisms between tangles 
correspond to 
natural transformations between bimodules. Or said in other words, 
whether there is a $2$-functor from the $2$-category of tangles to 
a certain $2$-category of 
$H=\bigoplus_{m\in\Z_{\geq 0}}H_m$-bimodules. 
This is often called \textit{functoriality}.

In a series of papers~\cite{BS1},~\cite{BS2},~\cite{BS3},~\cite{BS4} 
and~\cite{BS5}  
a \textit{generalization} 
$\arcalg_{\Lambda}$ of the arc algebra 
was studied revealing that 
Khovanov's arc algebra has, left aside its knot theoretical 
origin, interesting representation theoretical, 
algebraic geometrical and combinatorial 
properties. These algebras $\arcalg_{\Lambda}$ 
were defined using an algebraic 
approach via the \textit{combinatorics of arc diagrams}, i.e. 
certain diagrams consisting of embedded lines in $\R^2$ inspired by the diagrams 
for Temperley-Lieb algebras.

This series of results has led to several 
variations and generalizations of Khovanov's original formulation, 
utilized in a large body of work by several researchers 
(including the authors of this paper), e.g. 
an $\mathfrak{sl}_3$-variation  
considered in~\cite{MPT},~\cite{Rob1},~\cite{Rob2} and~\cite{Tub1}, 
and an $\mathfrak{sl}_n$-variation studied in~\cite{Mack} and~\cite{Tub}, all of them 
having relations to (cyclotomic) KL-R algebras as in~\cite{KL} or~\cite{Rou}, 
and link homologies 
in the sense of Khovanov and Rozansky~\cite{KR}.
There is also 
the $\mathfrak{gl}_{1|1}$-variation developed in~\cite{Sar} with 
relations to the Alexander polynomial 
as well as
a type 
$\boldsymbol{\mathrm{D}}$-version introduced in~\cite{ES1} and~\cite{ES2} 
with connections to the representation theory 
of Brauer's centralizer algebras and 
orthosymplectic Lie superalgebras, see e.g.~\cite{ES3}.

A fact we like to stress about the 
$\mathfrak{sl}_3$/$\mathfrak{sl}_n$-variations is 
that their \textit{graded}
$2$\textit{-cate\-gories of biprojective, 
finite-dimensional modules} are equivalent to certain \textit{graded}
$2$\textit{-categories of} 
$\mathfrak{sl}_3$/$\mathfrak{sl}_n$\textit{-foams},
the analogues of Bar-Natan's cobordism category~\cite{BN1} studied e.g. 
in~\cite{Khova},~\cite{LQR},~\cite{MSV} and~\cite{QR} 
from the viewpoint of link homologies. 
(We note hereby that such a topological 
description for the type 
$\boldsymbol{\mathrm{D}}$-version was
found in~\cite{ETW1}, providing, in some sense, 
the first ``foamy description'' outside 
of type
$\boldsymbol{\mathrm{A}}$.)

We like to stress that Khovanov's original construction 
as well as Bar-Natan's reformulation from~\cite{BN1} are 
\textit{not} functorial, but 
are \textit{functorial up to signs}, 
see~\cite{BN1},~\cite{Jac},~\cite{Ras} or~\cite{Str}.
It became clear that 
Bar-Natan cobordisms miss some subtle extra signs
(see for example~\cite{CMW} for the first fix of functoriality using 
``disoriented'' cobordisms).

A solution to this problem, that is of key interest for us, 
was provided by Blanchet in~\cite{Bla}. 
He formulated Khovanov's link homology using certain 
singular cobordisms, that we call ($\gltwo$-)\textit{foams}, which, by construction, 
include highly non-trivial 
signs fixing the functoriality of Khovanov's link homology.
Moreover, Blanchet's formulation fits neatly into the framework of 
graded $2$-representations of the categorified 
quantum group in the sense of~\cite{KL}, as it was shown in~\cite{LQR}.

Using Blanchet's construction it makes sense to define ``foamy'' versions 
$\webalg_{\vec{k}}$ of 
Khovanov's arc algebra, which we call 
($\gltwo$-)\textit{web algebras}. We set
$\webalg=\bigoplus_{\vec{k}\in\bY}\webalg_{\vec{k}}$. 
The web algebras $\webalg_{\vec{k}}$ and $\webalg$ are graded
$\field$-algebras defined using 
Blanchet's singular cobordisms and the multiplication is given by composition of 
singular cobordisms (for our conventions see Section~\ref{subsec:foamy}).
The signs within this multiplication are quite sophisticated, e.g. 
even merges (which are 
quite easy in the formulations of~\cite{BS1} and~\cite{ES2}) 
can come with a sign.

Unfortunately calculating in 
$\webalg_{\vec{k}}$ and $\webalg$ is very hard. Indeed, 
it is not even clear what a basis 
of $\webalg_{\vec{k}}$ or $\webalg$ is - left aside the question how to rewrite 
an arbitrary foam in terms of some basis.
Thus, the main purpose of this paper is to 
give \textit{algebraic counterparts} of 
$\webalg_{\vec{k}}$ and $\webalg$, denoted by $\Arcalg_{\Lambda}$ and 
$\Arcalg=\bigoplus_{\Lambda\in\bblock}\Arcalg_{\Lambda}$ 
where these questions about bases are easy. 
We call the algebraic counterparts, which are 
built up using certain 
combinatorics of arc diagrams, \textit{Blanchet-Khovanov algebras}.

The proof of our main theorem relies on 
the rather subtle Theorem~\ref{proposition:matchalgebras} 
which needs careful treatment of all involved signs. 
The whole Subsection~\ref{subsec:proofisoofalgebras} 
is devoted to its proof.
Our main theorem clarifies algebraically the deficiency in the 
original theory.
For brevity, we stop our investigation here 
although several natural questions 
remain open, e.g. 
a direct representation theoretic construction of 
Blanchet-Khovanov algebras, see 
Remark~\ref{remark:uniquecategorification}.
\\[0.2cm]
To keep the paper self-contained we 
start by a rather detailed exposition of 
the main ingredients and players adapted to the main purpose of the paper.

\subsubsection*{Conventions used throughout}

\begin{convention}\label{convention:used-throughout1}
Let $\field$ denote a field of arbitrary 
characteristic. 
(We sometimes need to work with $\field(q)$ 
for a formal parameter $q$. All notions 
below are similarly defined in this case.)
An \textit{algebra} 
always means a non-necessarily
finite-dimensional, non-necessarily unital $\field$-algebra $A$. 
We do not assume that 
such $A$'s are associative and it will be a non-trivial 
fact that all $A$'s which we consider are actually associative.
Given two algebras $A$ and $B$, then an $A$\textit{-}$B$\textit{-bimodule} is a 
$\field$-vector space $M$ with a left action of $A$ and a right action of $B$ 
in the usual sense. If $A=B$, then we also write $A$\textit{-bimodule} 
for short.
We call an $A$-$B$-bimodule $M$ \textit{biprojective}, if 
it is projective as a left $A$-module and right $B$-module 
(such finitely generated bi\-modules are called sweet in~\cite[Subsection~2.6]{Khov}).
We denote
the category of \textit{locally fi\-nite-dimensional} $A$-bimodules by $\biMod{A}$, i.e.
the category of $A$-bimodules $M$ such that $eMe^{\prime}$ is finite-dimensional for any
two primitive idempotents $e,e^{\prime}\in A$. Diagrammatic left 
(or right) actions will be given by acting on the 
bottom (or top).
\end{convention}

\begin{convention}\label{convention:used-throughout}
By a \textit{graded algebra} we mean an 
algebra $A$ which decomposes into 
graded pieces $A=\bigoplus_{i\in\Z}A_i$ such that $A_iA_j\subset A_{i+j}$ 
for all $i,j\in\Z$. 
Given two graded algebras $A$ and $B$, 
we study (and only consider) \textit{graded} 
$A$-$B$-bimodules, i.e. $A$-$B$-bimodules 
$M=\bigoplus_{i\in\Z}M_i$ such that $A_iM_jB_k\subset M_{i+j+k}$ for all $i,j,k\in\Z$.
We also set $M_i\{s\}=M_{i-s}$ for $s\in\Z$ (thus, 
positive integers shift up).

If $A$ is a graded algebra and $M$ is a graded $A$-bimodule, 
then $\overline M$ obtained from $M$ by forgetting the grading is 
in $\biMod{A}$.

Given such $A$-bimodules 
$\overline M,\overline N$, then
\begin{equation}\label{eq:degreehom}
\Hom_{\biMod{A}}(\overline M,\overline N)=\bigoplus_{s\in\Z}\Hom_{0}(M,N\{s\}).
\end{equation}
Here $\Hom_{0}$ means all degree-preserving $A$-homomorphisms, 
i.e. $f(M_i)\subset N_i$.
\end{convention}

\begin{convention}\label{convention:used-throughout0}
We consider three diagrammatic calculi in this paper: 
$\gltwo$-webs (webs for short) in the sense of~\cite{CKM} and~\cite{Kup},  
foams whose definition is motivated from~\cite{Bla} and~\cite{Khova}, and 
arc diagrams in the sense of~\cite{BS1} and~\cite{BS3}. 
Our reading convention for all of these 
is from bottom to top and from left to right. 
We often illustrate local pieces only; the diagram 
then is meant to be the identity or arbitrary 
outside of the 
displayed part (which one will be clear from the context).
\end{convention}

\begin{remark}\label{remark:colors}
We use colors in this paper. It is only necessary to distinguish 
colors for webs and foams. For the readers 
with a black-and-white version: we illustrate colored web edges using 
dashed lines, while colored foam facets appear shaded.
\end{remark}

%% file: res/2-webs-and-foams.tex
In this section we 
introduce the foam $2$-category $\foamcat$ and 
the web algebra $\webalg$ in the spirit of 
Khovanov~\cite{Khov}, but using foams {\`a} la Blanchet~\cite{Bla}.

\subsection{Webs, foams and TQFTs}\label{subsec:foams}

We start by recalling the definition of a web. For this purpose, 
we denote by $\iY$ the set of 
all vectors 
$\vec{k}=(k_i)_{i\in\Z}\in\{0,1, \invo{1},2,\invo{2}\}^{\Z}$ with $k_i=0$ 
for $|i|\gg 0$. Abusing notation, we also sometimes write 
$\vec{k}=(k_{a},\dots,k_b)$ for some fixed part of $\vec{k}$ (with $a<b\in\Z$) 
where it is to be understood that all non-displayed 
entries are zero. 
By convention, the \textit{empty vector} 
is the unique vector containing only zeros. 
We consider $\vec{k}\in\iY$ as a set of 
discrete labeled points in $\R\times\{\pm 1\}$ (or in 
$\R\times\{0\}$) by putting 
the symbols $k_i$ at position $(i,\pm 1)$ (or $(i,0)$).
We denote by $\Y\subset\iY$ the subset of all 
vectors with entries from $\{0,1,2\}$ only.

\begin{definition}\label{definition:sl2webs}
A \textit{web} is an embedded 
labeled, oriented,
trivalent graph which can be obtained by gluing 
(whenever this makes sense and the labels fit) or 
juxtaposition of finitely many 
(possibly zero) of the following pieces:
\begin{gather}\label{eq:webs}
\begin{aligned}
&\raisebox{0.09cm}{
\xy
(0,0)*{\includegraphics[scale=1]{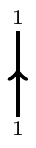}};
\endxy}
\quad,\quad
\xy
(0,0)*{\includegraphics[scale=1]{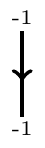}};
\endxy
\quad,\quad
\xy
(0,0)*{\includegraphics[scale=1]{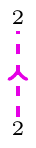}};
\endxy
\quad,\quad
\xy
(0,0)*{\includegraphics[scale=1]{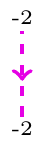}};
\endxy
\quad,\quad
\xy
(0,0)*{\includegraphics[scale=1]{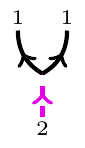}};
\endxy
\quad,\quad
\xy
(0,0)*{\includegraphics[scale=1]{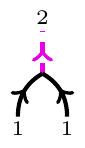}};
\endxy
\quad,\quad
\xy
(0,0)*{\includegraphics[scale=1]{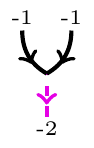}};
\endxy
\quad,\quad
\xy
(0,0)*{\includegraphics[scale=1]{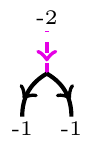}};
\endxy
\\
&\raisebox{0.09cm}{
\xy
(0,0)*{\includegraphics[scale=1]{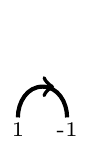}};
\endxy}
\;\;,\;\;
\xy
(0,0)*{\includegraphics[scale=1]{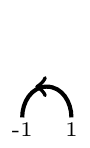}};
\endxy
\;\;,\;\;
\xy
(0,0)*{\includegraphics[scale=1]{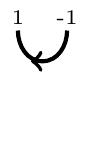}};
\endxy
\;\;,\;\;
\xy
(0,0)*{\includegraphics[scale=1]{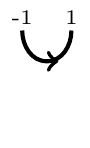}};
\endxy
\;\;,\;\;
\xy
(0,0)*{\includegraphics[scale=1]{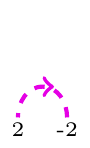}};
\endxy
\;\;,\;\;
\xy
(0,0)*{\includegraphics[scale=1]{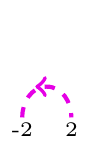}};
\endxy
\;\;,\;\;
\xy
(0,0)*{\includegraphics[scale=1]{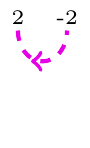}};
\endxy
\;\;,\;\;
\xy
(0,0)*{\includegraphics[scale=1]{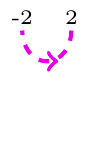}};
\endxy
\end{aligned}
\end{gather}
(Hence, this includes the \textit{empty web}.)
We assume that webs are embedded in $\R\times[-1,1]$ such 
that each edge starts/ends either in a trivalent vertex
or at the boundary of the strip at the points $(i,\pm 1)$. We assume 
that the points at $(i,\pm 1)$ are labeled $1$, $\invo{1}$, 
$2$ or $\invo{2}$. 
In particular, 
these webs have distinguished bottom $\vec{k}$ and top $\vec{l}$
boundary which we will throughout denote from left 
to right by $\vec{k}=(k_a,\dots,k_b)$ and 
$\vec{l}=(l_{a^{\prime}},\dots,l_{b^{\prime}})$ where $k_i$ is the 
label at $(i,-1)$ and $l_i$ is the 
label at $(i,1)$.

Edges come in two different version, i.e. \textit{ordinary edges} 
which are only allowed 
to have boundary points labeled $1$ or $\invo{1}$, and \textit{phantom edges} 
which are only allowed to have boundary points labeled 
$2$ or $\invo{2}$. 
As in~\eqref{eq:webs}, we draw phantom edges dashed (and colored); one should 
think of them as ``non-existing''. 

We denote the set consisting 
of all webs with bottom boundary $\vec{k}$ and top boundary $\vec{l}$ 
by $\Hom_{\overline{\F}}(\vec{k},\vec{l})$ 
(for a reason that will become clear later). 
Given $\vec{k}\in\iY$, we denote by 
$\oneinsert{\vec{k}}\in\Hom_{\overline{\F}}(\vec{k},\vec{k})$ 
the identity web on 
$\vec{k}$.
\end{definition}

\begin{remark}\label{remark:highest-weight-setup}
For our purposes it will be 
mostly sufficient to consider webs in 
a ``highest weight setup'', i.e. only upwards pointing 
webs (see also 
Subsection~\ref{subsection:qgroup}).
In particular, we do not need the labels $\invo{1}$ 
and $\invo{2}$ much in this paper. 
Still, all construction from this and the next subsection 
can be done in 
a more flexible setup using topological 
webs, which are however not 
necessary for our purposes.
\end{remark}

By a \textit{surface} we mean a marked, orientable, compact surface 
with possible finitely many boundary components and with finitely many 
connected components. 
Additionally, by a \textit{trivalent surface} we understand the same as 
in~\cite[Subsection~3.1]{Khova}, i.e. certain 
embedded, marked, singular cobordisms whose boundaries are webs.

Precisely, fix the following data denoted by $\boldsymbol{S}$:
\begin{enumerate}[label=(\Roman*)]
\item A surface $S$ with connected components 
divided into two sets $\{S^{\osymbol}_1,\dots,S^{\osymbol}_r\}$ and 
$\{S^{\psymbol}_1,\dots, S^{\psymbol}_{r^{\prime}}\}$. (The former are called 
\textit{ordinary surfaces} 
and the latter 
are called \textit{phantom surfaces}.) 
\item The boundary components of $S$ are partitioned into triples 
$(C^{\osymbol}_i,C^{\osymbol}_j,C_k^{\psymbol})$ 
such that each triple contains precisely one boundary 
component $C_k^{\psymbol}$ 
of a phantom surface.
\item The three circles $C^{\osymbol}_i,C^{\osymbol}_j$ and $C_k^{\psymbol}$ in 
each triple are identified via 
diffeomorphisms $\varphi_{ij}\colon C^{\osymbol}_i\to C^{\osymbol}_j$ 
and $\varphi_{jk}\colon C^{\osymbol}_j\to C_k^{\psymbol}$.
\item A finite (possible empty) set 
of markers per connected components.
\end{enumerate}

\begin{definition}\label{definition-trivsurface}
Let $\boldsymbol{S}$ be as above.
The \textit{closed, singular trivalent surface} $f_c=f_c^{\boldsymbol{S}}$ 
attached to $\boldsymbol{S}$ is 
the CW-complex obtained as the quotient of $S$ by the identifications 
$\varphi_{ij}$ and $\varphi_{jk}$.
We call all such $f_c$'s 
\textit{closed pre-foams} (following~\cite{Khova}) and their markers 
\textit{dots}. A 
triple $(C^{\osymbol}_i,C^{\osymbol}_j,C_k^{\psymbol})$ becomes one circle in $f_c$
which we call a \textit{singular seam}, while the interior 
of the connected components $S^{\osymbol}_1,\dots,S^{\osymbol}_r$ and 
$S^{\psymbol}_1,\dots, S^{\psymbol}_{r^{\prime}}$ are facets of $f_c$, 
called \textit{ordinary facets} respectively \textit{phantom facets}.
We embed these pre-foams into $\R^3$ 
in such a way that the three annuli glued 
to a singular seam can be oriented. We additionally 
choose orientations on the singular seams, compare to~\eqref{eq:orientation}.
\end{definition}

\begin{example}\label{example:foamsforcatharina}
Consider two spheres with two punctures respectively one puncture. 
The sphere with one puncture is assumed to be a phantom sphere 
(we color phantom facets in what follows).
\[
\xy
\xymatrix{
\xy
(0,0)*{\includegraphics[scale=1]{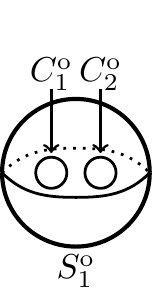}};
\endxy
\quad\quad
\xy
(0,0)*{\includegraphics[scale=1]{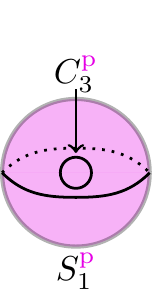}};
\endxy
\quad\quad
\ar@<-5pt>[r]^/.5cm/{\text{glue}} &
\quad\quad
\xy
(0,0)*{\includegraphics[scale=1]{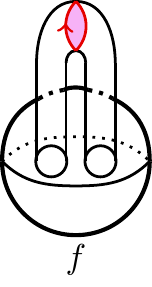}};
\endxy
}
\endxy
\]  
Then the pre-foam on the right is obtained by identifying the three boundary circles. 
(We have also chosen an orientation of the singular seam.)
Another example of a closed pre-foam 
is given below in the proof of Lemma~\ref{lemma:TQFTyetagain} 
(where we leave it to the reader to identify the precise labels).
\end{example}

The pre-foams we have constructed so far are all closed. 
We will need \textit{non-closed} pre-foams as well.
To this end, we follow~\cite[Subsection~3.3]{Khova} 
and consider a plane $P\cong\R^2\subset\R^3$. 
We say that $P$ intersects a closed pre-foam 
$f_c$ \textit{generically}, 
if $P\cap f_c$ is a non-oriented web 
(seen as a topological space).

\begin{definition}\label{definition:openfoam}
Let $P^{\pm 1}_{xy}$ be the $xy$-plane in $\R^3$ 
(embedded such that the third coordinate is $\pm 1$).
A \textit{(non-closed) pre-foam} $f$ is the intersection 
of $\R^2\times[-1,1]$ with some closed pre-foam $f_c$ such that 
$P^{\pm 1}_{xy}$ intersects $f_c$ generically, and 
the intersection is a web as in Definition~\ref{definition:sl2webs}. 
We see such pre-foam $f$ as a singular cobordism 
between $P^{-1}_{xy}\cap f_c$ (bottom, source) and 
$P^{+1}_{xy}\cap f_c$ (top, target) embedded in 
$\R^2\times[-1,1]$. Moreover, there is an evident composition $g\circ f$
via gluing and rescaling. Similarly, we construct pre-foams 
embedded in $\R\times[-1,1]\times[-1,1]$ with vertical boundary 
components. These vertical boundary components should be 
the boundary of the webs at the bottom/top times $[-1,1]$. 
We consider such pre-foams modulo isotopies in $\R\times[-1,1]\times[-1,1]$ 
which fix the horizontal boundary as well as the vertical boundary, 
and the condition that generic slices are webs.
\end{definition}

We call pre-foam parts \textit{ordinary}, 
if they do not contain 
singular seams or phantom facets, and we 
call pre-foam parts \textit{ghostly}, 
if they only contain phantom facets.

\begin{examplen}\label{example:foams}
Pre-foams can be seen as singular surfaces (with oriented, singular seams) 
in $\R\times[-1,1]\times[-1,1]$ such 
that the bottom boundary and the top boundary 
are webs with facets colored as follows:
\[
\xy
(0,0)*{\includegraphics[scale=1]{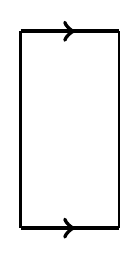}};
\endxy\colon
\xy
(0,0)*{\includegraphics[scale=1]{figs/fig01.pdf}};
\endxy
\to
\xy
(0,0)*{\includegraphics[scale=1]{figs/fig01.pdf}};
\endxy\quad\text{and}\quad
\xy
(0,0)*{\includegraphics[scale=1]{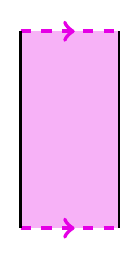}};
\endxy\colon
\xy
(0,0)*{\includegraphics[scale=1]{figs/fig03.pdf}};
\endxy
\to
\xy
(0,0)*{\includegraphics[scale=1]{figs/fig03.pdf}};
\endxy
\]
The leftmost facet is an ordinary facet.
Whereas the rightmost facet is a phantom facet, and the 
reader might think of it as ``non-existing'' (similar 
to a phantom edge) - they only 
encode signs. 
The singularities 
of $f$ are all locally of the following form 
(where the other orientations of the seams are also allowed)
\begin{equation}\label{eq:orientation}
\xy
(0,0)*{\reflectbox{\includegraphics[scale=1]{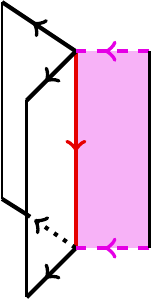}}};
\endxy\colon
\xy
(0,0)*{\includegraphics[scale=1]{figs/fig05.pdf}};
\endxy
\to
\xy
(0,0)*{\includegraphics[scale=1]{figs/fig05.pdf}};
\endxy
\quad\text{and}\quad
\xy
(0,0)*{\includegraphics[scale=1]{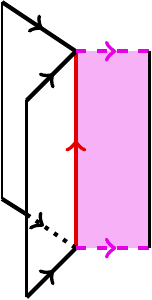}};
\endxy
\colon
\xy
(0,0)*{\includegraphics[scale=1]{figs/fig06.pdf}};
\endxy
\to
\xy
(0,0)*{\includegraphics[scale=1]{figs/fig06.pdf}};
\endxy
\end{equation}
(Note that we consider webs  
in a ``monoidal'' way. 
Thus, we do not have to relate the orientations 
of facets/seam to the orientations of webs as 
e.g. in~\cite[Section~1]{Bla}.)
Such pre-foams can carry dots 
that freely 
move around its facets:
\[
\xy
(0,0)*{\includegraphics[scale=1]{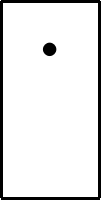}};
\endxy
\;=\;
\xy
(0,0)*{\includegraphics[scale=1]{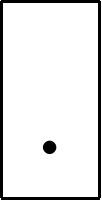}};
\endxy
\quad\text{and}\quad
\xy
(0,0)*{\includegraphics[scale=1]{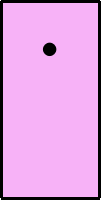}};
\endxy
\;=\;
\xy
(0,0)*{\includegraphics[scale=1]{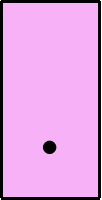}};
\endxy
\hspace{2.85cm}
\text{\raisebox{-1cm}{$\makeqedtri$}}
\hspace{-2.85cm}
\]
\end{examplen}

\begin{remark}\label{remark:isotopies}
Pre-foams are considered modulo boundary 
preserving isotopies that do preserve 
the condition that each generic slice is a web. These isotopies 
form a finite list: isotopies
coming from the two cobordism theories associated 
to the two different types of facets 
(see for example~\cite[Section~1.4]{Kock}) and 
isotopies coming from isotopies of the singular seams seen as 
tangles in $\R^2\times[-1,1]$.
\end{remark}

To work with the $2$-category of foams it will be enough (for our purposes)
to consider its image under a 
certain (singular) 
TQFT functor defined by Blanchet~\cite{Bla}. 
Recall that equivalence classes of 
TQFTs for surfaces 
(i.e. a monoidal functor $\TQFTold$ from the category of 
two-dimensional cobordisms to the category of $\field$-vector spaces) are in one-to-one 
correspondence with isomorphism classes of 
(finite-dimensional, associative) commutative Frobenius algebras. The reader 
unfamiliar with these notions might want to consult Kock's book~\cite{Kock} 
for a detailed account. 
Given a Frobenius algebra $\mathcal{A}$ corresponding 
to a TQFT $\TQFTold_{\mathcal{A}}$, then the association 
is as follows. 
To a disjoint union of $m$ circles one 
associates the $m$-fold tensor product $\mathcal{A}^{\otimes m}$. To 
a cobordism $\Sigma$ with distinguished incoming and outgoing
boundary components consisting of, let us say, $m$ and $m^{\prime}$ circles, 
we assign a $\field$-linear map from $\mathcal{A}^{\otimes m}$ to 
$\mathcal{A}^{\otimes m^{\prime}}$.
Hereby 
the usual cup/cap respectively pants cobordisms
correspond to
the unit, counit, multiplication and comultiplication maps. These 
are the basic pieces of every cobordism. Then 
the TQFT assigns to $\Sigma$ a $\field$-linear map
\[
\TQFTold_{\mathcal{A}}(\Sigma)\colon\mathcal{A}^{\otimes m}\to\mathcal{A}^{\otimes m^{\prime}}, 
\]
which is obtained by decomposing $\Sigma$ into basic pieces.

The commutative Frobenius algebras we need are 
\begin{equation}\label{eq:frobs}
\mathcal{A}_{\osymbol}=\field[X]/(X^2),\quad\quad\mathcal{A}_{\psymbol}=\field
\end{equation}
with induced multiplications, counits 
$\varepsilon_{\osymbol,\psymbol}(\cdot)$ and comultiplications 
$\Delta_{\osymbol,\psymbol}(\cdot)$
given via
\begin{gather*}
\begin{aligned}
\varepsilon_{\osymbol}(1)=0,\quad\quad\varepsilon_{\osymbol}(X)&=1,
\quad\quad\varepsilon_{\psymbol}(1)=-1,\\
\Delta_{\osymbol}(1)=1\otimes X+X\otimes 1,\quad\quad \Delta_{\osymbol}(&X)=X\otimes X,\quad\quad\Delta_{\psymbol}(1)=-1\otimes 1.
\end{aligned}
\end{gather*}
Thus, we have the traces
\begin{gather}\label{eq:trace}
\mathrm{tr}_{\osymbol}(1\otimes 1)=\mathrm{tr}_{\osymbol}(X\otimes X)=0,\;
\mathrm{tr}_{\osymbol}(1\otimes X)=\mathrm{tr}_{\osymbol}(X\otimes 1)=1,\;
\mathrm{tr}_{\psymbol}(1\otimes 1)=-1.
\end{gather}
We associate the Frobenius algebra 
$\mathcal{A}_{\osymbol}$ to the ordinary parts,
and the Frobenius algebra 
$\mathcal{A}_{\psymbol}$ to the phantom parts of a pre-foam $f$ 
using the usual notion
of a TQFT functor, but extended to surfaces marked 
with dots via multiplication by $X$ or $-1$:
\begin{equation}\label{eq:dotsinfoams}
\xy
(0,0)*{\includegraphics[scale=1]{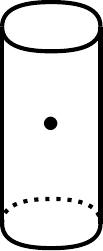}};
\endxy\;
\overset{\text{\raisebox{0.1cm}{$\TQFTold_{\mathcal{A}_{\osymbol}}$}}}{\longmapsto}\;
\cdot X\colon\mathcal{A}_{\osymbol}\to\mathcal{A}_{\osymbol},\quad\quad
\xy
(0,0)*{\includegraphics[scale=1]{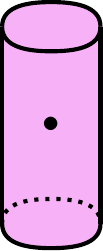}};
\endxy\;
\overset{\text{\raisebox{0.1cm}{$\TQFTold_{\mathcal{A}_{\psymbol}}$}}}{\longmapsto}\;
\cdot ({-}1)\colon\mathcal{A}_{\psymbol}\to\mathcal{A}_{\psymbol}.
\end{equation}
(Note that 
pre-foams without singular seams 
and markers are surfaces in the 
usual sense.)
By the universal construction given in~\cite{BHMV}, it is 
no problem to extend the usual construction of TQFTs to the marked 
setup using~\eqref{eq:dotsinfoams}. We leave 
the details to the reader.

\begin{example}\label{example:someexampleTQFT}
If we view a $\field$-linear 
map $\phi\colon\field\to\mathcal{A}_{\osymbol,\psymbol}^{\otimes m}$ 
as $\phi(1)\in\mathcal{A}_{\osymbol,\psymbol}^{\otimes m}$, then
\begin{gather}\label{eq:something}
\xy
(0,0)*{\includegraphics[scale=1]{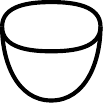}};
\endxy\;
\overset{\text{\raisebox{0.1cm}{$\TQFTold_{\mathcal{A}_{\osymbol}}$}}}{\longmapsto}\; 1\in\mathcal A_{\osymbol},\quad\quad
\xy
(0,0)*{\includegraphics[scale=1]{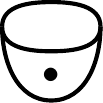}};
\endxy\;
\overset{\text{\raisebox{0.1cm}{$\TQFTold_{\mathcal{A}_{\osymbol}}$}}}{\longmapsto}\; X\in\mathcal A_{\osymbol},\quad\quad
\xy
(0,0)*{\includegraphics[scale=1]{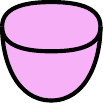}};
\endxy\;
\overset{\text{\raisebox{0.1cm}{$\TQFTold_{\mathcal{A}_{\psymbol}}$}}}{\longmapsto}\; 1\in\mathcal A_{\psymbol}.
\end{gather}
Here we have from left to right 
$\iota_{\osymbol}$, $(\cdot X)\circ\iota_{\osymbol}$ 
and $\iota_{\psymbol}$ as maps.
These are sometimes called (marked) \textit{units}.
The \textit{counits} $\varepsilon_{\osymbol,\psymbol}$ 
are obtained by flipping the pictures.
\end{example}

Note that the values of the 
non-closed surfaces can be determined by closing 
them in all possible ways using~\eqref{eq:something} and its dual.
Here and throughout, we say for short that a relation $a=b$ 
(of formal $\field$-linear combinations 
of marked surfaces) 
\textit{lies in the kernel of a TQFT functor} $\mathcal Z$, if 
$\mathcal{Z}(a)=\mathcal{Z}(b)$ as $\field$-linear maps. 
(Similarly later on for singular TQFT functors as defined below.)

\begin{lemma}\label{lemma:moreandmore}
The following \textit{ordinary sphere relations}, 
the \textit{ghostly sphere relation}
and the \textit{cyclotomic relations}
\begin{gather}
\xy
(0,0)*{\includegraphics[scale=1]{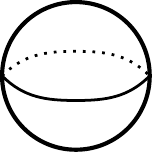}};
\endxy
=0,\quad\quad
\xy
(0,0)*{\includegraphics[scale=1]{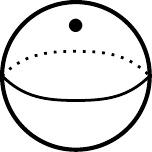}};
\endxy
=1,\quad\quad
\xy
(0,0)*{\includegraphics[scale=1]{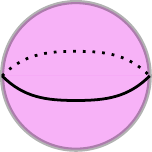}};
\endxy
=-1,\label{eq:theusualrelations1}\\
\raisebox{0.09cm}{\xy
(0,0)*{\includegraphics[scale=1]{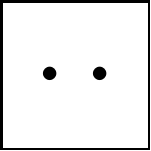}};
\endxy}=0,\quad\quad
\raisebox{0.09cm}{\xy
(0,0)*{\includegraphics[scale=1]{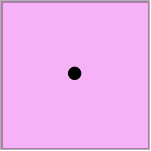}};
\endxy}=-1.\label{eq:theusualrelations1b}
\end{gather}
as well as
the following \textit{ordinary and ghostly neck cutting relations}
\begin{gather}\label{eq:theusualrelations2}
\xy
(0,0)*{\includegraphics[scale=1]{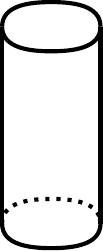}};
\endxy
\;=\;
\xy
(0,0)*{\includegraphics[scale=1]{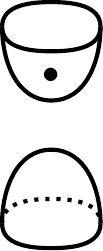}};
\endxy
\;+\;
\xy
(0,0)*{\includegraphics[scale=1]{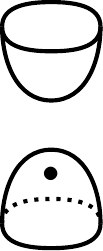}};
\endxy
\quad\text{and}\quad
\xy
(0,0)*{\includegraphics[scale=1]{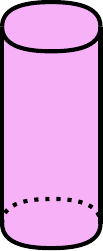}};
\endxy
\;=\;-\;
\xy
(0,0)*{\includegraphics[scale=1]{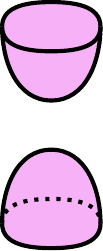}};
\endxy
\end{gather}
are in the kernel 
of $\TQFT_{\mathcal{A}_{\osymbol}}$ 
(ordinary) respectively of $\TQFT_{\mathcal{A}_{\psymbol}}$ (ghostly).\makeqed
\end{lemma}

\begin{proof}
Via direct calculation. 
For instance, the ordinary respectively ghostly 
neck cutting relations decompose the identity map as 
$\mathrm{id}=(\cdot X)\circ\iota_{\osymbol}\circ\varepsilon_{\osymbol}+\iota_{\osymbol}\circ\varepsilon_{\osymbol}\circ(\cdot X)$ 
respectively $\mathrm{id}=-\iota_{\psymbol}\circ\varepsilon_{\psymbol}$ with 
units $\iota_{\osymbol,\psymbol}$ and counits 
$\varepsilon_{\osymbol,\psymbol}$ as in~\eqref{eq:something}.
\end{proof}

The neck cutting relations~\eqref{eq:theusualrelations2} give a 
topological interpretation of a dot as a shorthand notation 
for ($\frac{1}{2}$-times) a handle, see also~\cite[(4)]{BN1}.

\subsection{Blanchet's singular TQFT construction}\label{sub:Blanchet}

The following definition follows the 
construction given by Blanchet, see~\cite[Subsection~1.5]{Bla}.

We want to construct a monoidal functor $\TQFT$ 
on 
the category whose objects 
are webs with values in finite-dimensional $\field$-vector spaces.
To this end, let $\Ff$ denote the category 
whose objects are webs and whose morphisms are pre-foams (composition 
is gluing of pre-foams). 
We view $\Ff$ as a monoidal category by juxtaposition 
of webs and pre-foams.
Moreover, we define for $a,b,c,d\in\field$ two maps:
\begin{gather}\label{eq:TQFT2}
\begin{aligned}
\alpha_{\mathcal{A}_{\osymbol}}\colon \mathcal{A}_{\osymbol}\otimes\mathcal{A}_{\osymbol}\to \mathcal{A}_{\osymbol},\; 
(a+& bX)\otimes (c+dX)\mapsto (a+bX)(c-dX),\\
\alpha_{\mathcal{A}_{\psymbol}}\colon &\mathcal{A}_{\psymbol}\to\mathcal{A}_{\osymbol},\; 1\mapsto 1.
\end{aligned}
\end{gather}

\begin{definition}\label{definition:foamycat}
Let $\TQFTold_{\mathcal{A}_{\osymbol}}$ and 
$\TQFTold_{\mathcal{A}_{\psymbol}}$ denote the TQFTs associated to 
$\mathcal{A}_{\osymbol}$ and $\mathcal{A}_{\psymbol}$ from~\eqref{eq:frobs}. Given a 
closed pre-foam $f_c$, 
let $\dot{f}_c=f_{\osymbol}\dot{\cup}f_{\psymbol}$ be the pre-foam obtained by cutting 
$f_c$ along the singular seams (of which we assume to have $m$ in total). 
Here $f_{\osymbol}$ is the surface which in $f_c$ is attached 
to the ordinary parts and $f_{\psymbol}$ is the surface which in $f_c$ is attached 
to phantom parts. Note that the boundary of $f_{\osymbol}$ splits 
into $\sigma_i^+$ and 
$\sigma_i^-$ for each $i\in\{1,\dots,m\}$. Which one is which depends on the orientation 
of the singular seam: use the 
right hand rule with the index finger pointing in the direction of the singular seam 
and the middle finger pointing in direction of the attached phantom facet, then 
the thumb points in direction of $\sigma_i^+$. In contrast, 
$f_{\psymbol}$ has only boundary components $\sigma_i$ for each $i\in\{1,\dots,m\}$. Now
\begin{gather*}
\TQFTold_{\mathcal{A}_{\osymbol}}(f_{\osymbol})\in\bigotimes_{i=1}^m (\TQFTold_{\mathcal{A}_{\osymbol}}(\sigma_i^+)
\otimes \TQFTold_{\mathcal{A}_{\osymbol}}(\sigma_i^-)),\quad\quad
\TQFTold_{\mathcal{A}_{\psymbol}}(f_{\psymbol})\in\bigotimes_{i=1}^m \TQFTold_{\mathcal{A}_{\psymbol}}(\sigma_i).
\end{gather*}
Let $\mathrm{tr}_{\osymbol}\colon\mathcal{A}_{\osymbol}\to\field$ 
be as in~\eqref{eq:trace}, 
and let $\alpha_{\mathcal{A}_{\osymbol}},\alpha_{\mathcal{A}_{\psymbol}}$ be as in~\eqref{eq:TQFT2}.
Then we set
\[
\TQFT(f_c)=(\mathrm{tr}_{\osymbol})^{\otimes 2m}
(\alpha_{\mathcal{A}_{\osymbol}}^{\otimes m}(\TQFTold_{\mathcal{A}_{\osymbol}}(f_{\osymbol}))
\otimes \alpha_{\mathcal{A}_{\psymbol}}^{\otimes m}(\TQFTold_{\mathcal{A}_{\psymbol}}(f_{\psymbol})))
\in\field^{\otimes 2m}\cong\field.
\] 
This gives a well-defined value $\TQFT(f_c)\in\field$ for all closed 
pre-foams $f_c$.
\end{definition}

A crucial insight of Blanchet is that this extends 
to pre-foams:

\begin{theorem}\label{theorem:blanchet}
The construction from Definition~\ref{definition:foamycat} 
can be extended to a monoidal functor 
$\TQFT\colon\Ff\to\Vect$. (We call such 
a functor a \textit{singular TQFT}.)\makeqed
\end{theorem}

\begin{proof}
This follows from the universal construction from~\cite{BHMV}. 
That is, the only thing one really needs to check for this is 
that the $\field$-vector spaces constructed via the 
universal construction are finite-dimensional. This is 
not a priori clear, but also not hard to show. 
First observe that the evaluation given above ensures the case 
for the $\field$-vector space associated to the empty web.
By using the relations found below, one can show an analog 
of Lemma~\ref{lemma:invertible}, which in turn provides recursively 
that the $\field$-vector space associated to any web is finite-dimensional.
\end{proof}

Note the following properties of pre-foams $f$, which follow by construction.
\begin{enumerate}[label=(\Roman*)] 
\item The \textit{topological reduction} $\hat f$ obtained 
by removing all phantom facets of $f$ is the cobordism 
theory corresponding to $\mathcal{A}_{\osymbol}$ from~\eqref{eq:frobs}.
\item The \textit{phantom} $\check f$ obtained 
by removing all $1$-labeled facets of $f$ is the cobordism 
theory corresponding to $\mathcal{A}_{\psymbol}$ from~\eqref{eq:frobs}.
\end{enumerate}

Hence, the relations from~\eqref{eq:theusualrelations1},~\eqref{eq:theusualrelations1b} 
and~\eqref{eq:theusualrelations2} 
are also in the kernel of the functor $\TQFT$ 
(for all possible orientations of the boundary webs). 

\begin{lemma}\label{lemma:orientation}
Let $\tilde f$ 
be the pre-foam obtained from a pre-foam $f$ by reversing 
the orientation of a singular seam. Then $f+\tilde f=0$ is 
in the kernel of $\TQFT$.\makeqed
\end{lemma}

\begin{proof}
Switching the orientation 
of a singular seam swaps the attached parts of 
$\sigma_1^+$ and $\sigma_1^-$. 
In particular, it swaps 
the two copies of 
$\mathcal{A}_{\osymbol}$ in the source 
of $\alpha_{\mathcal{A}_{\osymbol}}$ from~\eqref{eq:TQFT2} 
and hence, produces an extra sign
(we note that the case $b=d=0$ is killed 
by applying the trace $\varepsilon_{\osymbol}$ 
in the formula for $\TQFT(f_c)$).
\end{proof}

\begin{lemma}\label{lemma:TQFTyetagain}
Let $a,b\in\Z_{\geq 0}$. The \textit{sphere relations}, i.e. 
\begin{gather}\label{eq:sphere}
\xy
(0,0)*{\includegraphics[scale=1]{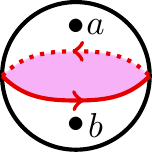}};
\endxy
=\begin{cases}\phantom{-}1,& \text{if }a=1,b=0,\\-1,& \text{if }a=0,b=1,\\\phantom{-}0,&\text{otherwise},
\end{cases}
\end{gather}
are in the kernel of $\TQFT$. 
(We call such pre-foams \textit{spheres}.)\makeqed
\end{lemma}

\begin{proof}
We prove the case $a=0,b=1$. The others 
are similar and omitted for brevity. Decompose $f_c$ into 
(t=thumb, i=index finger, m=middle finger)
\[
\xy
\xymatrix{
\xy
(0,0)*{\includegraphics[scale=1]{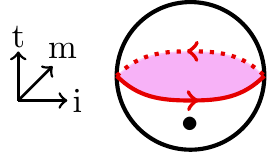}};
\endxy
\quad\ar@{~>}[r]&
\quad
\xy
(0,0)*{\includegraphics[scale=1]{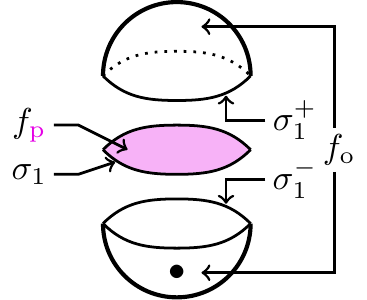}};
\endxy
}
\endxy
\]
Now, because of the assignment in~\eqref{eq:something}, 
we have $\TQFT_{\mathcal{A}_{\osymbol}}(f_{\osymbol})=1\otimes X$ and 
$\TQFT_{\mathcal{A}_{\psymbol}}(f_{\psymbol})=1$. 
Thus, $\alpha_{\mathcal{A}_{\osymbol}}(\TQFT_{\mathcal{A}_{\osymbol}}(f_{\osymbol}))=-X$ 
and $\alpha_{\mathcal{A}_{\psymbol}}(\TQFT_{\mathcal{A}_{\psymbol}}(f_{\psymbol}))=1$, 
both considered in $\mathcal{A}_{\osymbol}$. 
Applying the trace 
$\mathrm{tr}_{\osymbol}$ to $-X\otimes 1$ gives $-1$ as in~\eqref{eq:sphere}.
\end{proof}

\begin{lemma}\label{lemma:morerelations}
The \textit{bubble removals} (a ``sphere'' in a phantom plane; the top 
dots are meant to be on the front facets and the bottom dots on the back facets)
\begin{gather}
\raisebox{0.0775cm}{\xy
(0,0)*{\includegraphics[scale=1]{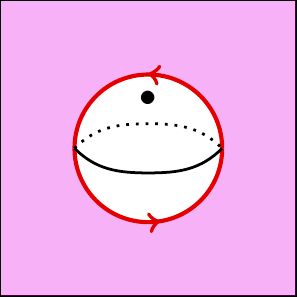}};
\endxy}
\;=\;
\xy
(0,0)*{\includegraphics[scale=1]{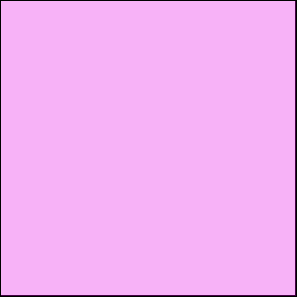}};
\endxy
\;=\;
-\;
\xy
(0,0)*{\includegraphics[scale=1]{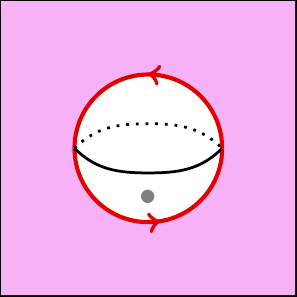}};
\endxy\label{eq:bubble1}
\\
\xy
(0,0)*{\includegraphics[scale=1]{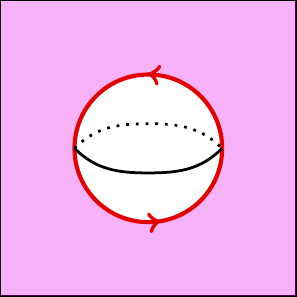}};
\endxy
\;=\;
0
\;=\;
\raisebox{-0.005cm}{\xy
(0,0)*{\includegraphics[scale=1]{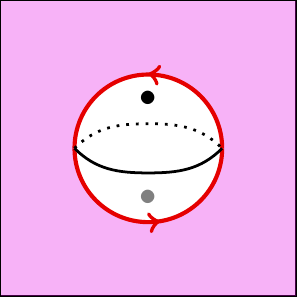}};
\endxy}\label{eq:bubble2}
\end{gather}
are in the kernel of $\TQFT$. The \textit{neck cutting relation}
\begin{gather}\label{eq:neckcut}
\xy
(0,0)*{\includegraphics[scale=1]{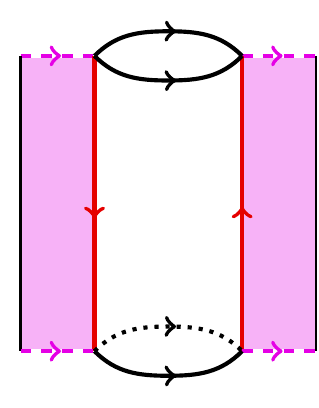}};
\endxy
\;=\;
\xy
(0,0)*{\includegraphics[scale=1]{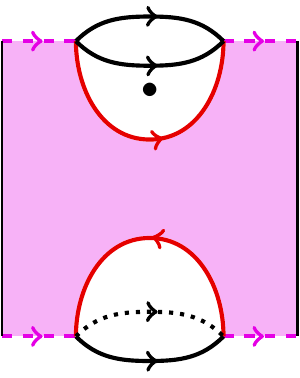}};
\endxy
\;-\;
\xy
(0,0)*{\includegraphics[scale=1]{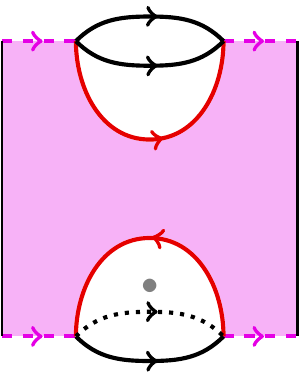}};
\endxy
\end{gather}
(with top dot on the front facet and 
bottom dot on the back facet) 
is also in the kernel of $\TQFT$. 
Furthermore, 
the (left) \textit{squeezing relation}
\begin{gather}\label{eq:squeezing}
\xy
(0,0)*{\includegraphics[scale=1]{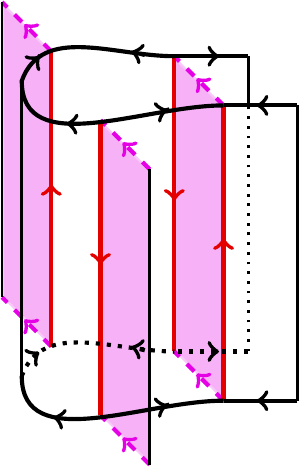}};
\endxy
=
\;-\;
\xy
(0,0)*{\includegraphics[scale=1]{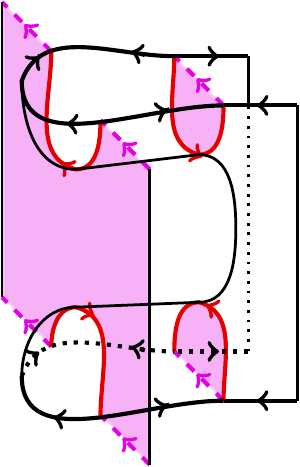}};
\endxy
\end{gather}
(there is also a similar right squeezing relation)
and the \textit{dot migrations}
\begin{gather}\label{eq:dotmigration}
\xy
(0,0)*{\reflectbox{\includegraphics[scale=1]{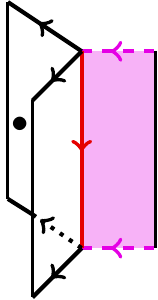}}};
\endxy
\;=\;
-\;
\xy
(0,0)*{\reflectbox{\includegraphics[scale=1]{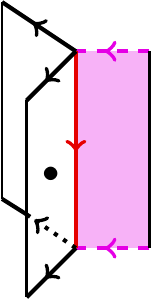}}};
\endxy
\quad\text{and}\quad
\xy
(0,0)*{\includegraphics[scale=1]{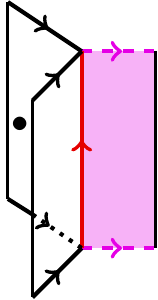}};
\endxy
\;=\;
-\;
\xy
(0,0)*{\includegraphics[scale=1]{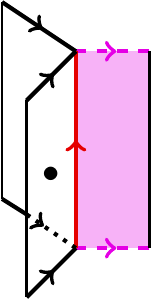}};
\endxy
\end{gather}
as well as 
the \textit{ordinary-to-phantom neck cutting relations} 
(in the leftmost picture the upper closed circle is an ordinary facet, while the lower 
closed circle is a phantom facet, and vice versa for the rightmost picture)
\begin{gather}\label{eq:neckcutphantom}
\xy
(0,0)*{\includegraphics[scale=1]{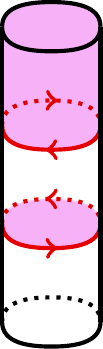}};
\endxy
\;=\;
-\;
\xy
(0,0)*{\includegraphics[scale=1]{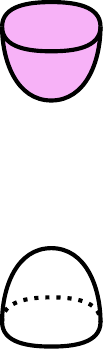}};
\endxy
\quad\quad\text{and}\quad\quad
\xy
(0,0)*{\includegraphics[scale=1]{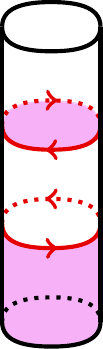}};
\endxy
\;=\;
-\;
\xy
(0,0)*{\includegraphics[scale=1]{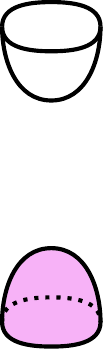}};
\endxy
\end{gather}
(we have omitted the orientations of the circle-shaped webs)
are also in the kernel of $\TQFT$. 
(We only need these in the following. But there are 
also similar relations with 
different orientations of the webs.)\makeqed
\end{lemma}

The leftmost situation in~\eqref{eq:neckcut} is called a \textit{cylinder} - 
as all local parts of pre-foams $f$ such that the corresponding part in $\hat{f}$ 
is a cylinder. Note that the squeezing relation~\eqref{eq:squeezing} 
enables us to use the neck cutting~\eqref{eq:neckcut} 
on any such cylinders.

\begin{proof}
We only prove the left equation in~\eqref{eq:neckcutphantom}. 
First note that we have to consider all possible 
ways to close the non-closed pre-foam on the left-hand and on the right-hand side 
of the equation. We consider the closing
\[
\xy
(0,0)*{\includegraphics[scale=1]{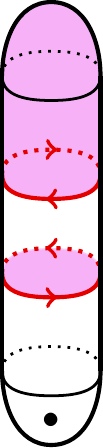}};
\endxy
\quad\quad\text{and}\quad\quad
-\;
\xy
(0,0)*{\includegraphics[scale=1]{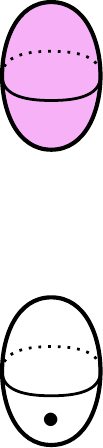}};
\endxy
\]
since all other possibilities give zero (as the reader might want to check).
By~\eqref{eq:theusualrelations1}, the right-hand closed pre-foams evaluate 
to $-(1\cdot(-1))$. 
Now:
\[
\xy
(0,0)*{\includegraphics[scale=1]{figs/BK-algebras-figure105.pdf}};
\endxy
\;\stackrel{\eqref{eq:theusualrelations2}}{=}\;
\xy
(0,0)*{\includegraphics[scale=1]{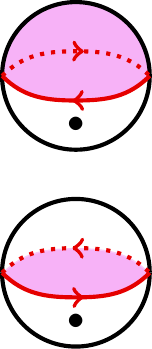}};
\endxy
\;+\;
\xy
(0,0)*{\includegraphics[scale=1]{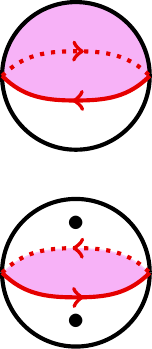}};
\endxy
\;
\stackrel{\eqref{eq:sphere}}{=}\;
\xy
(0,0)*{\includegraphics[scale=1]{figs/BK-algebras-figure108.pdf}};
\endxy
\]
The bottom sphere evaluates to $-1$ because of~\eqref{eq:sphere}. 
Moreover, performing the same steps as 
in the proof of Lemma~\ref{lemma:TQFTyetagain}, 
we see that the top sphere also evaluates 
to $-1$. Thus, the left-hand 
and the right-hand side evaluate to the 
same value. This shows that 
the first equation in~\eqref{eq:neckcutphantom} is in the kernel 
of $\TQFT$.
The other relations are verified similarly, see also~\cite[Lemma~1.3]{Bla} 
or~\cite[Subsection~3.1]{LQR}.
\end{proof}

If we define a grading on the TQFT-modules by setting 
$\mathrm{deg}(1)=-1$ 
and $\mathrm{deg}(X)=1$, 
then the TQFT $\TQFTold_{\mathcal{A}_{\osymbol}}$ 
respects the grading, where 
the degree of a cobordism $\Sigma$ is 
given by $\mathrm{deg}(\Sigma)=-\chi(\Sigma)+2\cdot\text{dots}$. 
Here $\chi(\Sigma)$ is the topological Euler characteristic of $\Sigma$, that is, the number 
of vertices minus the number of edges plus the number of faces of 
$\Sigma$ seen as a CW complex, and ``dots'' 
is the total number of dots. Additionally, we can see 
the TQFT $\TQFTold_{\mathcal{A}_{\psymbol}}$ as being trivially graded.
Motivated by this we define the following.

\begin{definition}\label{definition:foamydegree}
Given a pre-foam $f$, we define its
\textit{degree}
\[
\mathrm{deg}(f)=-\chi(\hat f)+2\cdot\text{dots}+\tfrac{1}{2}\text{vbound},
\]
where vbound is the total number of vertical boundary components. 
If $\hat f$ is the empty cobordism, 
then, by convention, $\chi(\hat f)=0$.
\end{definition}

\begin{example}\label{example:foamydegree}
For example,
\[
\mathrm{deg}
\left(\;
\xy
(0,0)*{\reflectbox{\includegraphics[scale=1]{figs/BK-algebras-figure90.pdf}}};
\endxy
\;\right)
=2,\quad\quad
\mathrm{deg}
\left(\;
\xy
(0,0)*{\includegraphics[scale=1]{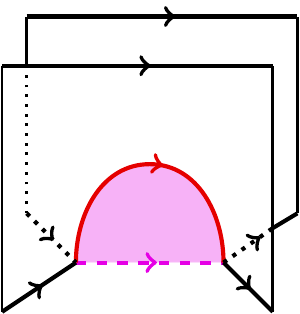}};
\endxy
\;\right)
=1.
\]
The leftmost 
pre-foam is called a \textit{dotted cup} (the name will 
become clear in Lemma~\ref{lemma:multfoams}), while
the rightmost pre-foam is called a 
\textit{saddle} (there are also saddles obtained 
by flipping the picture upside down). 
Furthermore, we have
\[
\mathrm{deg}
\left(\;
\xy
(0,0)*{\includegraphics[scale=1]{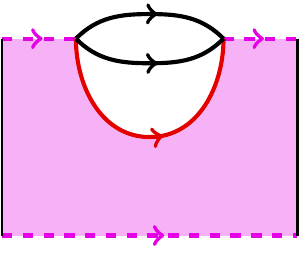}};
\endxy
\;\right)
=-1=
\mathrm{deg}
\left(\;
\xy
(0,0)*{\includegraphics[scale=1]{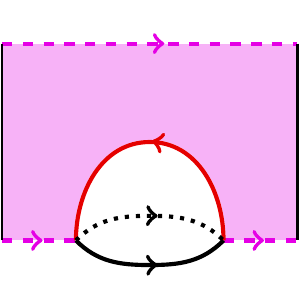}};
\endxy
\;\right)
\]
for the pre-foams called \textit{cup} respectively \textit{cap}.
\end{example}

The $2$-category we like to study is cooked up from $\TQFT$ 
as follows.

\begin{definition}\label{defn:foamcat}
Let $\overline{\F}$ be the $\field$-linear $2$-category given by:
\begin{itemize}
\item The objects are all $\vec{k}\in\iY$.
\item The morphisms spaces 
$\Hom_{\overline{\F}}(\vec{k},\vec{l})$ 
are as in Definition~\ref{definition:sl2webs}.
\item The $2$-morphisms spaces 
$\twoHom_{\overline{\F}}(u,v)$ for two webs $u,v$ 
is the $\field$-linear span of all pre-foams with bottom boundary 
$u$ and top boundary $v$.
\item Vertical 
compositions $g\circ f$ of pre-foams 
by stacking $g$ on top of $f$, horizontal 
composition $g\otimes f$ by putting $g$ 
to the right of $f$ (whenever those 
operations make sense).
\item We take everything modulo the 
relations 
from~\eqref{eq:theusualrelations1},~\eqref{eq:theusualrelations1b} 
and~\eqref{eq:theusualrelations2}, as well as the 
relations found in 
Lemmas~\ref{lemma:orientation},~\ref{lemma:TQFTyetagain} and~\ref{lemma:morerelations}.
\end{itemize}
Since the relations are degree preserving, $\overline{\F}$ is 
a graded, $\field$-linear $2$-category by taking the degree from 
Definition~\ref{definition:foamydegree}. 
Similarly, the ``highest weight'' $2$-subcategory 
$\F$ of $\overline{\F}$ is the full $2$-subcategory 
consisting of only webs with upwards pointing edges. 
(We will only consider $\F$ in the following.) 
\end{definition}

We call the $2$-morphisms in $\F$ 
(or in $\overline{\F}$) \textit{foams}. Moreover, all 
notions we had for pre-foams can be adapted to the setting of foams 
and we do so in the following. Note that 
the objects and the morphisms of $\foamcat$ 
can be seen as a $\field(q)$-linear category 
(by considering the spaces 
$\Hom_{\F}(\vec{k},\vec{l})$ as $\field(q)$-linear vector spaces 
whose basis are the webs in $\Hom_{\F}(\vec{k},\vec{l})$) 
which we denote by $\sltwowebcat$.

\subsection{An action of the quantum group \texorpdfstring{$\Udot$}{Udot(glinfty)}}\label{subsection:qgroup}

We denote by $\Udot$ the 
$\field(q)$-linear category whose objects are given by $\vec{k}\in\Y$ 
and whose morphisms are generated by pairwise orthogonal idempotents
$\onek$, and by $E^{(r)}_i\onek$ and $F^{(r)}_i\onek$ for 
$\vec{k}\in\Y$, $i\in\Z$ 
and $r\in\Z_{>0}$ (the generators $E^{(r)}_i\onek$ and $F^{(r)}_i\onek$ are called the \textit{divided powers})
modulo some relations 
which are analogs of the relations 
in the quantum group $\mathrm{U}_q(\mathfrak{gl}_{\infty})$ 
(see~\cite[Chapter~23]{Lus}). 
We note that there exists a unique $\onel$ such that $\onel E^{(r)}_i\onek\neq 0$ 
and $\onel F^{(r)}_i\onek\neq 0$.
This enable us to write $\onek$ only 
on one side of any expression.
Now, there is a $\field(q)$-linear functor
\[
\Phi^{\sltwowebcat}_{\mathrm{Howe}}\colon\Udot\to\sltwowebcat
\]
given on objects by $\Phi^{\sltwowebcat}_{\mathrm{Howe}}(\vec{k})=\vec{k}$, 
and on morphisms by $\Phi^{\sltwowebcat}_{\mathrm{Howe}}(\onek)=\oneinsert{\vec{k}}$  
and (where we use a simplified, ``rectangular'', notation for webs):
\begin{gather}\label{eq:EF}
\begin{aligned}
E_i\onek&\overset{\text{\raisebox{0.05cm}{$\Phi^{\sltwowebcat}_{\mathrm{Howe}}$}}}{\longmapsto}
\xy
(0,0)*{\includegraphics[scale=1]{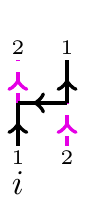}};
\endxy
\quad,\quad
\xy
(0,0)*{\includegraphics[scale=1]{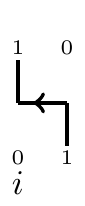}};
\endxy
\quad,\quad
\xy
(0,0)*{\includegraphics[scale=1]{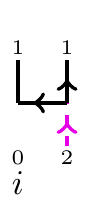}};
\endxy
\quad,\quad
\xy
(0,0)*{\includegraphics[scale=1]{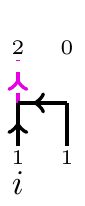}};
\endxy
\quad,\quad
E_i^{(2)}\onek\overset{\text{\raisebox{0.05cm}{$\Phi^{\sltwowebcat}_{\mathrm{Howe}}$}}}{\longmapsto}
\xy
(0,0)*{\includegraphics[scale=1]{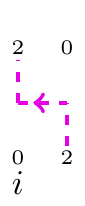}};
\endxy\\
F_i\onek&\overset{\text{\raisebox{0.05cm}{$\Phi^{\sltwowebcat}_{\mathrm{Howe}}$}}}{\longmapsto}
\xy
(0,0)*{\includegraphics[scale=1]{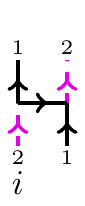}};
\endxy
\quad,\quad 
\xy
(0,0)*{\includegraphics[scale=1]{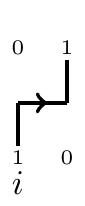}};
\endxy
\quad,\quad
\xy
(0,0)*{\includegraphics[scale=1]{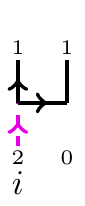}};
\endxy
\quad,\quad
\xy
(0,0)*{\includegraphics[scale=1]{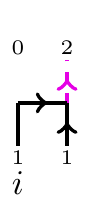}};
\endxy
\quad,\quad
F_i^{(2)}\onek\overset{\text{\raisebox{0.05cm}{$\Phi^{\sltwowebcat}_{\mathrm{Howe}}$}}}{\longmapsto}
\xy
(0,0)*{\includegraphics[scale=1]{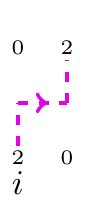}};
\endxy
\end{aligned}
\end{gather}
Here the generators $E_i\onek$, $F_i\onek$, $E^{(2)}_i\onek$ 
and $F_i^{(2)}\onek$ are sent 
to the local (between 
strand $i$ and $i+1$) pictures above (we have displayed all possibilities 
depending on $\vec{k}$ at position $k_i$ and $k_{i+1}$). 
Moreover, all higher divided powers 
$E_i^{(r)}\onek,F_i^{(r)}\onek$ for $r>2$ are sent to zero. 
We call webs that arise 
as $\Phi^{\sltwowebcat}_{\mathrm{Howe}}(X)$, for $X$ being any composition 
of the $\onek,E_i^{(r)}\onek,F_i^{(r)}\onek$ generators, 
$EF$\textit{-generated}, and, 
on the other hand, 
webs $F$\textit{-generated}, if $X$ is any composition of 
only $\onek,F_i^{(r)}\onek$ generators 
(in both cases no coefficients from $\field(q)$ 
are allowed to occur).
For details about the functor $\Phi^{\sltwowebcat}_{\mathrm{Howe}}$ 
we refer to~\cite[Section~5]{CKM}.

\begin{example}\label{example:EFwebs}
If $\vec{k}=(0,0,1,0,2,0,1,2,0,0)$, then 
$E_1E_0E_{-1}E_2F_0\onek$ is sent to
\begin{gather*}
\xy
(0,0)*{\includegraphics[scale=1]{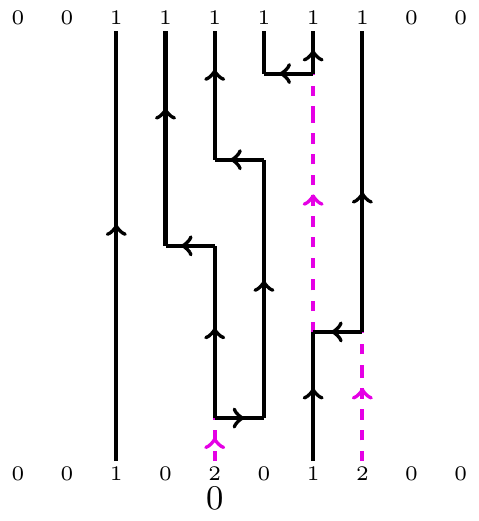}};
\endxy
\end{gather*} 
Here the first entry $2$ is assumed to be 
at position $0$.
\end{example}

\subsection{\texorpdfstring{$\gltwo$}{gl2}-web algebras}\label{subsec:webalg}
Now we define the following ``algebraic'' version of $\foamcat$. To this end, 
let $\ell\in\Z_{\geq 0}$ and let 
$\omega_{\ell}=(1,\dots,1,0,\dots,0)$ with $\ell$ numbers equal to $1$. 
Fix $\vec{k}$ and let $\CUP(\vec{k})=\Hom_{\mathfrak{F}}(2\omega_{\ell},\vec{k})$ 
and $\CAP(\vec{k})=\Hom_{\mathfrak{F}}(\vec{k},2\omega_{\ell})$. 
(Thus, $\CUP(\vec{k})=\emptyset$ and 
$\CAP(\vec{k})=\emptyset$ if $\sum_{i\in\Z}k_i\neq2\ell$.)
Elements of these are called \textit{cup webs} respectively \textit{cap webs}.
For diagrams in the multiplication 
process of $\webalg$ described below, 
we need \textit{cup-ray webs} as well, i.e. elements of
$\CUPRAY(\vec{k})=\Hom_{\F}(\omega_{\ell+\ell^{\prime}}+\omega_{\ell},\vec{k})$ 
for $\vec{k}\in\Y$ 
(and similarly defined \textit{cap-ray webs}).

\begin{definition}\label{definition:webalg}
Let $u,v\in\CUP(\vec{k}),\vec{k}\in\Y$. 
We denote by
${}_u(\webalg^{\natural}_{\vec{k}})_v$ the space $\twoHom_{\F}(u,v)$.
The 
\textit{web algebra} 
$\webalg^{\natural}_{\vec{k}}$ \textit{for} $\vec{k}\in\Y$ and the 
\textit{(full) web algebra} 
$\webalg^{\natural}$ are the graded $\field$-vector spaces
\[
\webalg^{\natural}_{\vec{k}}=\!\!\!\!\!\bigoplus_{u,v\in\CUP(\vec{k})}\!\!\!\!\!{}_u(\webalg^{\natural}_{\vec{k}})_v,\quad\;\;\webalg^{\natural}=\bigoplus_{\vec{k}\in\Y}\webalg^{\natural}_{\vec{k}},
\]
whose grading is induced by the grading in $\foamcat$. 
We consider these as graded 
algebras with multiplication given by composition in $\foamcat$.
\end{definition}

\begin{remark}\label{remark:clear}
Note that $\webalg^{\natural}_{\vec{k}}$ is defined 
via composition of foams and thus, forms a graded, associative, unital 
algebra. Similarly for (the locally unital) algebra $\webalg^{\natural}$. 
\end{remark}

\begin{remark}\label{remark:webalg}
Although new in this form, the algebras 
from Definition~\ref{definition:webalg} are 
of course inspired by Khovanov's original 
arc algebras from~\cite{Khov}. Consequently, we 
obtain that $\webalg_{\vec{k}}$ is a graded 
Frobenius algebra (by copying~\cite[Theorem~3.9]{MPT}).
\end{remark}

\begin{definition}\label{definition:bY}
Denote by $\bY\subset\Y$ the set of all $\vec{k}\in\Y$ 
which have an even number of entries $1$. We call elements of $\bY$ \textit{balanced}.
\end{definition}

\begin{remark}\label{remark:bwebs}
Clearly there are no cups respectively caps if $\vec{k}\in\Y-\bY$. 
Hence, $\webalg^{\natural}_{\vec{k}}=0$ iff $\vec{k}\in\Y-\bY$ or $\vec{k}=\emptyset$.
Consequently, we restrict ourselves to balanced $\vec{k}$ in what follows. 
We note that the full set $\Y$ would be needed if one wants to 
study \textit{generalized} web algebras 
in the sense of~\cite{BS1} or~\cite{BS3}.
\end{remark}

There is an alternative way to define the web algebras 
which is the one we will use later on. Thus, we make the 
following definition (and show below that it agrees 
with the one from Definition~\ref{definition:webalg}).
We denote by ${}^{\ast}$ 
the involution on webs which flips 
the diagrams upside down 
and reverses their orientations.

\begin{definition}\label{definition:webalg2}
Let $u,v\in\CUP(\vec{k}),\vec{k}\in\bY$. 
We denote by
${}_u(\webalg_{\vec{k}})_v$ the space 
$\twoHom_{\F}(\oneinsert{2\omega_{\ell}},uv^{\ast})\{d(\vec{k})\}$, where 
$d(\vec{k})=\ell-\sum_{i\in\Z}k_i(k_i-1)$.
The 
\textit{web algebra} 
$\webalg_{\vec{k}}$ \textit{for} $\vec{k}\in\bY$ and the 
\textit{(full) web algebra} 
$\webalg$ are the graded $\field$-vector spaces
\[
\webalg_{\vec{k}}=\!\!\!\!\!\bigoplus_{u,v\in\CUP(\vec{k})}\!\!\!\!\!{}_u(\webalg_{\vec{k}})_v,\quad\;\;\webalg=\bigoplus_{\vec{k}\in\bY}\webalg_{\vec{k}}.
\] 
We consider these as graded 
algebras with multiplication
\begin{equation}\label{eq:multfoams}
\boldsymbol{\mathrm{Mult}}\colon\webalg_{\vec{k}} \otimes \webalg_{\vec{k}} \rightarrow \webalg_{\vec{k}},\;f\otimes g\mapsto \boldsymbol{\mathrm{Mult}}(f,g)=fg
\end{equation} 
using multiplication foams as follows. 
To multiply $f\in{}_u(\webalg_{\vec{k}})_v$ 
with $g\in{}_{\tilde v}(\webalg_{\vec{k}})_w$  
stack the diagram $\tilde vw^{\ast}$ on top of $uv^{\ast}$ 
and obtain $uv^{\ast}\tilde vw^{\ast}$. Then $fg=0$ if $v\neq\tilde v$. Otherwise, 
pick the leftmost cup-cap pair
indicated in the picture to the left below and perform a ``surgery''
\begin{equation*}
\xy
\xymatrix{
\reflectbox{\xy
(0,0)*{\includegraphics[scale=1]{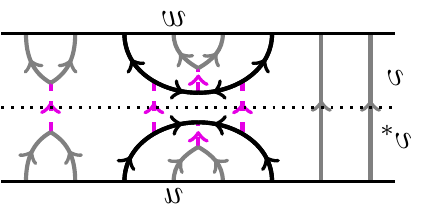}};
\endxy}
\ar[rr]^{\text{saddle foam}}
&&
\reflectbox{\xy
(0,0)*{\includegraphics[scale=1]{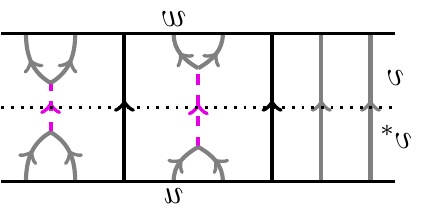}};
\endxy}
}
\endxy
\end{equation*}
where the saddle foam is locally of the 
following form (and the identity elsewhere)
\begin{equation}\label{eq:ssaddle}
\xy
(0,0)*{\includegraphics[scale=1]{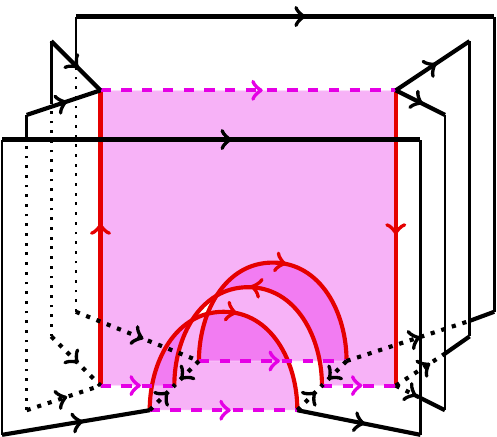}};
\endxy
\end{equation}
This foam should be read as follows: start with 
$f\in\twoHom_{\F}(\oneinsert{2\omega_{\ell}},uv^{\ast}vw)$ 
and stack on top of it a foam which is the 
identity at the bottom ($u$ part)
and top ($w$ part) of the web and the saddle 
in between.
Repeat until no cup-cap pair as
above remains. This gives inductively 
rise to a multiplication foam (after the last
surgery step we collapse the webs and foams). 
Compare also to~\cite[Definition~3.3]{MPT}.
\end{definition}

Note that each intermediate step in the multiplication 
from~\eqref{eq:multfoams} is a web 
of the form $v^{\ast}v$ with $v\in\CAPRAY(\vec{k})$ 
and the multiplication foam is zero or a foam in 
$\Hom_{\F}(uv^{\ast}vw^{\ast},uw^{\ast})$
(and thus, locally a foam in $\Hom_{\F}(v^{\ast}v,\oneinsert{\vec{k}})$).
As a convention, we consider $u\in\CUP(\vec{k})$ 
as a web in $\mathbb{R} \times [-1,0]$, $v^{\ast}\in\CAP(\vec{k})$ 
as a web in $\mathbb{R} \times [0,1]$ such that $\vec{k}\in\R\times\{0\}$ 
whenever we use this viewpoint on $\webalg$. 
Similarly for cup-ray and cap-ray webs.

\begin{lemma}\label{lemma:multfoam}
The multiplication is degree preserving.\makeqed
\end{lemma}

\begin{proof}
The saddles as in~\eqref{eq:ssaddle} 
are always of degree $1$. 
These have two 
boundary components labeled $1$. 
Thus, the degree of all saddles 
within the multiplication procedure and the shift by $d(\vec{k})$ 
sum up to zero.
\end{proof}

\begin{example}\label{example:multfoams}
An easy example illustrating the multiplication is 
\begin{gather}\label{eq:saddleconvention}
\begin{xy}
  \xymatrix{
\xy
(0,0)*{\includegraphics[scale=1]{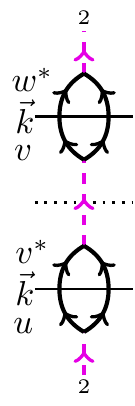}};
\endxy
\ar[rr]^{
\xy
(0,0)*{\includegraphics[scale=1]{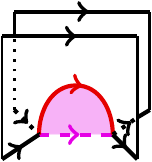}};
\endxy
}& &
\xy
(0,0)*{\includegraphics[scale=1]{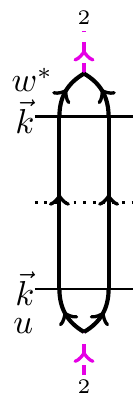}};
\endxy
\ar[rr]^{
\text{collapsing}
}& &
\xy
(0,0)*{\includegraphics[scale=1]{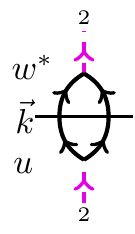}};
\endxy
}
\end{xy}
\end{gather}
where the reader should think about any foam 
$f\colon\boldsymbol{1}_{2\omega_{1}}\to uv^{\ast}vw^{\ast}$ sitting underneath. 
The saddle is of degree $1$ and thus, taking the shift $d(\vec{k})$ 
into account for $\vec{k}=(1,1)$, the multiplication 
foam is of degree zero.
\end{example}

The following shows that Definitions~\ref{definition:webalg} 
and~\ref{definition:webalg2} agree.

\begin{lemma}\label{lemma:multfoams}
We have $\webalg^{\natural}_{\vec{k}}\cong\webalg_{\vec{k}}$ 
and $\webalg^{\natural}\cong\webalg$ as graded algebras.\makeqed
\end{lemma}

\begin{proof}
Recall that the multiplication in $\webalg^{\natural}_{\vec{k}}$ is composition, 
while the multiplication in $\webalg_{\vec{k}}$ is given 
by multiplication foams. Thus, for the former we take foams 
$f\in\twoHom_{\F}(u,v)$ and $g\in\twoHom_{\F}(v,w)$ and obtain a 
foam $g\circ f\in\twoHom_{\F}(u,w)$, while for the latter 
we take foams $f\in\twoHom_{\F}(\oneinsert{2\omega_{\ell}},uv^{\ast})$ 
and $g\in\twoHom_{\F}(\oneinsert{2\omega_{\ell}},vw^{\ast})$
and obtain a foam 
$fg\in\twoHom_{\F}(\oneinsert{2\omega_{\ell}},uw^{\ast})$ 
(in case the multiplication is non-zero).
Now, the following ``clapping of pictures'' 
(as indicated by the arrows)
\[
\xy
(0,0)*{\reflectbox{\includegraphics[scale=1]{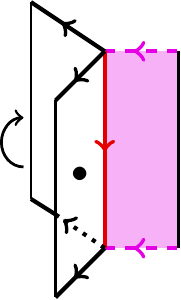}}};
\endxy
\leftrightsquigarrow
\xy
(0,0)*{\includegraphics[scale=1]{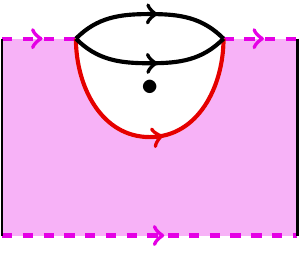}};
\endxy
\;\;\text{and}\;\;
\xy
(0,0)*{\includegraphics[scale=1]{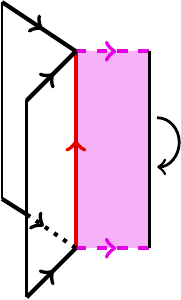}};
\endxy
\leftrightsquigarrow
\xy
(0,0)*{\includegraphics[scale=1]{figs/BK-algebras-figure112.pdf}};
\endxy
\] 
induces isomorphisms of $\field$-vector spaces
\begin{gather*}
\begin{aligned}
\twoHom_{\F}(u,w)&\cong\twoHom_{\F}(\oneinsert{2\omega_{\ell}},uw^{\ast})\{d(\vec{k})\},\\
\twoHom_{\F}(v,v)&\cong\twoHom_{\F}(v^{\ast}v,\oneinsert{2\omega_{\ell}})\{d(\vec{k})\}.
\end{aligned}
\end{gather*}
These are isomorphisms of graded $\field$-vector spaces since the shift 
by $d(\vec{k})$ encodes the vertical boundary components which are ``lost'' 
by the ``clapping''. Moreover, as indicated in the rightmost picture above, 
the multiplications in $\webalg^{\natural}_{\vec{k}}$ and 
$\webalg_{\vec{k}}$ are identified under this ``clapping procedure''. 
This shows the isomorphism of graded $\field$-algebras.
For more details the reader might also consult~\cite[Lemma~3.7]{MPT}.
\end{proof}

As a direct consequence of Remark~\ref{remark:clear} 
and Lemma~\ref{lemma:multfoams} we obtain 
in particular the 
associativity of $\webalg_{\vec{k}}$:

\begin{corollary}\label{corollary:multweb}
The map 
$\boldsymbol{\mathrm{Mult}}\colon\webalg_{\vec{k}}\otimes\webalg_{\vec{k}}\rightarrow\webalg_{\vec{k}}$ from Definition~\ref{definition:webalg2} is 
independent of the order in 
which the surgeries are performed. 
This turns $\webalg_{\vec{k}}$ into a 
graded, associative, unital algebra.
Similar for (the locally unital) algebra $\webalg$.\qedmake
\end{corollary}

\subsection{Web bimodules}\label{subsec:bimodweb}

We still consider only balanced $\vec{k},\vec{l}\in\bY$ in this subsection.

\begin{definition}\label{definition:bimoduleswebs}
Given any web $u\in\Hom_{\F}(\vec{k},\vec{l})$ 
(with boundaries $\vec{k}$ and $\vec{l}$ summing up to $2\ell$), 
we consider the $\webalg$-bimodule
\[
\M(u)=\!\!\!\!\bigoplus_{\substack{v\in\CUP(\vec{k}),\\w\in\CUP(\vec{l})}}\!\!\!\!\twoHom_{\F}(\oneinsert{2\omega_{\ell}},vuw^{\ast})
\]
with left (bottom) and 
right (top) action of $\webalg$ as in Definition~\ref{definition:webalg2}. 
We call such $\webalg$-bimodules $\M(u)$ \textit{web bimodules}.
\end{definition}

\begin{proposition}\label{proposition:webbimodules}
Let $u\in\Hom_{\F}(\vec{k},\vec{l})$ be a web. 
Then the left (bottom) action of $\webalg_{\vec{k}}$ and the 
right (top) action of 
$\webalg_{\vec{l}}$
on $\M(u)$ are well-defined and commute. 
Hence, $\M(u)$ is a 
$\webalg_{\vec{k}}$\hspace*{0.035cm}-$\webalg_{\vec{l}}$\hspace*{0.06cm}-bimodule 
(and thus, a $\webalg$-bimodule).\makeqed
\end{proposition}

\begin{proof}
Let $u\in\Hom_{\F}(\vec{k},\vec{l})$. Then,
by construction, the left (bottom) action of $\webalg_{\vec{k}}$ and 
the right (top) action of $\webalg_{\vec{l}}$ commute since they are topologically 
``far apart''. Hence, $\M(u)$ is indeed a 
$\webalg_{\vec{k}}$\hspace*{0.035cm}-$\webalg_{\vec{l}}$\hspace*{0.06cm}-bimodule 
(and thus, a $\webalg$-bimodule).
\end{proof}

Note that, given two webs $u,v\in\Hom_{\F}(\vec{k},\vec{l})$, 
then $\M(u)$ and $\M(v)$ could be isomorphic even though $u$ and 
$v$ are different, see for example~\eqref{eq:capcupwebs}.

\begin{proposition}\label{proposition:webbimodules2}
The $\webalg$-bimodules $\M(u)$ are graded biprojective
$\webalg$-bimodules, with finite-dimensional
subspaces for all pairs $v,w$.\makeqed
\end{proposition}

\begin{proof}
Clearly, they are graded, $\webalg$-bimodules, but 
finite-dimensionality is not a
priori clear. It follows from the 
existence of a cup foam basis as in Subsection~\ref{subsec:basis}. 
(More precisely, from Lemma~\ref{lemma-it-is-a-basis2}.)
They are also biprojective, because they are direct summands 
of some $\webalg_{\vec{k}}$ (of some $\webalg_{\vec{l}}$) as left (right) modules and 
for suitable $\vec{k}\in\bY$ (or $\vec{l}\in\bY$). 
See also~\cite[Proposition~5.11]{MPT}.
\end{proof}

This proposition motivates the definition of the following 
$2$-category which is one of the main objects that we 
are going to study.

\begin{definition}\label{definition:catbimodulesweb}
Given $\webalg$ as above, 
let $\webcat$ be the following $2$-category:
\begin{itemize}
\item Objects are the various $\vec{k}\in\bY$.
\item Morphisms are finite sums and tensor products 
(taken over the algebra $\webalg$) of $\webalg$-bimodules $\M(u)$.
\item The composition of $\webalg$-bimodules is given by tensoring (over $\webalg$).
\item $2$-morphisms are $\webalg$-bimodule homomorphisms.
\item The vertical composition of $\webalg$-bimodule homomorphisms is 
the usual composition and the horizontal composition 
is given by tensoring (over $\webalg$).
\end{itemize}
We consider $\webcat$ as a graded $2$-category with 
$2$-hom-spaces as in~\eqref{eq:degreehom}.  
\end{definition}

%% file: res/3-Khovanov.tex
In this section we define the \textit{Blanchet-Khovanov algebra}, following the framework of~\cite{Khov} and~\cite{BS1}, but with signs differing at a number of crucial places.

\subsection{Combinatorics of arc diagrams}\label{subsec:combinatorics}
We start with the notion of weights 
and blocks. These definitions are 
the same as in~\cite[Section~2]{BS1} 
and, apart from the exact definition of blocks, as in~\cite[Sections~2 and~3]{ES2}.

\begin{definition}\label{definition-block}
A \textit{(diagrammatical) weight} is a sequence 
$\lambda =(\lambda_i)_{i \in \mathbb{Z}}$ 
with entries $\lambda_i \in \{ \circ, \times, \down, \up \}$, 
such that $\lambda_i = \circ$ for $|i| \gg 0$. 
Two weights $\lambda$ and $\mu$ are said to be 
equivalent if one can obtain $\mu$ from $\lambda$ 
by permuting some symbols $\up$ and $\down$ in 
$\lambda$. The equivalence classes are called \textit{blocks}. 
We denote by $\block$ the set of blocks.
\end{definition}

To a block we assign a number of invariants.

\begin{definition}\label{definition-blockseq}
Let $\Lambda\in\block$ be a block. To 
$\Lambda$ we associate its (well-defined) \textit{block sequence} 
${\rm seq}(\Lambda) = ({\rm seq}(\Lambda)_i)_{i \in \mathbb{Z}}$ 
by taking any $\lambda \in \Lambda$ and replacing the symbols 
$\up,\down$ by $\dummy$. Moreover, we define 
$\mathrm{up}(\Lambda)$ respectively $\mathrm{down}(\Lambda)$ 
to be the total number of $\up$'s respectively 
$\down$'s in $\Lambda$ where we count $\times$ as 
both, $\up$ and $\down$.
\end{definition}

Important for us is the following subset of blocks.

\begin{definition}\label{definition:bblocks}
A block $\Lambda \in \block$ is called 
\textit{balanced}, if 
$\mathrm{up}(\Lambda) =\mathrm{down}(\Lambda)$. We denote by 
$\bblock\subset\block$ the set of balanced blocks.
\end{definition}

\begin{remark}\label{remark:bblocks}
The Blanchet-Khovanov algebras 
will only be defined for balanced blocks, 
while general blocks can be used to define a 
\textit{generalized} version of these algebras in the spirit of~\cite{BS1}.
\end{remark}

A \textit{cup diagram} $c$ is a finite collection 
of non-intersecting arcs inside 
$\mathbb{R} \times [-1,0]$ such that each 
arc intersects the boundary exactly in its 
endpoints, and either connecting two distinct points 
$(i,0)$ and $(j,0)$ with $i,j \in \Z$ (called a \textit{cup}), 
or connecting one point $(i,0)$ with $i \in \Z$ with a 
point on the lower boundary of $\mathbb{R} \times [-1,0]$ 
(called a \textit{ray}). Furthermore, each point in the 
boundary is endpoint of at most one arc.
Two cup diagrams are equal if the arcs contained 
in them connect the same points.
Similarly, a \textit{cap diagram} $d^{\ast}$ is defined 
inside $\mathbb{R} \times [0,1]$.
By 
construction, one can reflect a cup diagram 
$c$ along the axis $\mathbb{R} \times \{0\}$, 
denote this operation by ${}^{\ast}$,
to obtain a cap diagram 
$c^{\ast}$. Clearly, $(c^{\ast})^{\ast}=c$.

A cup diagram $c$ 
(and similarly a cap diagram $d^{\ast}$) is 
\textit{compatible} with a block $\Lambda\in\block$ if
$\{ (i,0) \mid {\rm seq}(\Lambda)_i = \dummy \} =
\left(\mathbb{R} \times \{0\}\right)\cap c$.

We will view a weight $\lambda$ as labeling integral points, called vertices, 
of the horizontal line $\mathbb{R} \times \{0\}$ 
inside $\mathbb{R} \times [-1,0]$ and 
$\mathbb{R} \times [0,1]$, putting the symbol 
$\lambda_i$ at position $(i,0)$. Together with 
a cup diagram $c$ this forms a new diagram $c\lambda$. 

\begin{definitionn}\label{definition:orientationcups}
We say that $c\lambda$ is \textit{oriented} if:
\begin{enumerate}[label=(\Roman*)]
\item An arc in $c$ only contains vertices 
labeled by $\up$ or $\down$.
\item The two vertices of a cup are labeled 
by exactly one $\up$ and one $\down$.
\item Every vertex labeled $\up$ or $\down$ 
is contained in an arc.
\item It is not possible to find $i<j$ such 
that $\lambda_i=\down$, $\lambda_j=\up$, 
and each are contained in a ray.\makeqedtri
\end{enumerate}
\end{definitionn}

In the following, when depicting a 
composite diagram like $c\lambda$, we will 
omit the line and only draw the labels obtained from $\lambda$.

\begin{example}\label{example:orientation}
Consider the following diagrams.
\begin{gather*}
\text{(i) }
\xy
(0,0)*{\includegraphics[scale=1]{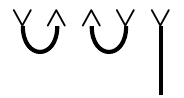}};
\endxy
\quad,\quad
\text{(ii) }
\xy
(0,0)*{\includegraphics[scale=1]{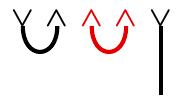}};
\endxy
\quad,\quad
\text{(iii) }
\xy
(0,0)*{\includegraphics[scale=1]{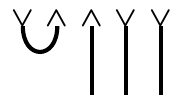}};
\endxy
\quad,\quad
\text{(iv) }
\xy
(0,0)*{\includegraphics[scale=1]{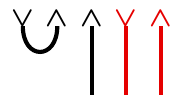}};
\endxy
\end{gather*}
The diagrams (i) and (iii) are oriented. Diagram (ii) 
is not oriented since condition (II) is 
violated, while (iv) is not oriented because condition (IV) is not fulfilled.
\end{example}

Similarly, a cap diagram $d^{\ast}$ together with 
a weight $\lambda$ forms a diagram $\lambda d^{\ast}$, 
which is called \textit{oriented} if $d\lambda$ is 
oriented. 
A cup respectively a cap in such diagrams 
is called \textit{anticlockwise} (or \textit{clockwise}), 
if its rightmost vertex is labeled $\up$ (or $\down$).

Putting a cap diagram $d^{\ast}$ on top 
of a cup diagram $c$ such that they are 
connected to the line $\mathbb{R} \times \{0\}$ 
at the same points creates a \textit{circle diagram}, 
denoted by $cd^{\ast}$. All connected component of this 
diagram that do not touch the boundary of 
$\mathbb{R} \times [-1,1]$ are called \textit{circles}, 
all others are called \textit{lines}. 
Together with $\lambda\in\Lambda$ such that $c\lambda$ and $\lambda d^{\ast}$ are 
oriented this forms an \textit{oriented circle diagram} $c \lambda d^{\ast}$.

\begin{definition}\label{definition:degreecups}
We define the \textit{degree} of 
an oriented cup diagram $c\lambda$, 
of an oriented cap diagram $\lambda d^{\ast}$ 
and of an oriented 
circle diagram $c\lambda d^{\ast}$
as follows.
\begin{gather}\label{eq:degreegenKh}
\begin{aligned}
\mathrm{deg}(c\lambda)&=\text{number of clockwise cups in } c\lambda,\\
\mathrm{deg}(\lambda d^{\ast})&=\text{number of clockwise caps in } \lambda d^{\ast},\\
\mathrm{deg}(c\lambda d^{\ast})&=\mathrm{deg}(c\lambda)+\mathrm{deg}(\lambda d^{\ast}).
\end{aligned}
\end{gather} 
Note that the degree is always non-negative.
\end{definition}

\begin{example}\label{example:degreecups}
In Example~\ref{example:orientation} above, 
diagram (i) has degree 1 (due to one clockwise cup) and diagram 
(iii) has degree 0.
\end{example}

Finally, we associate to each $\lambda\in\Lambda$ 
a unique cup diagram, denoted by $\underline{\lambda}$, 
via:
\begin{enumerate}[label=(\Roman*)]
\item Connect neighboring pairs $\down\up$ with a cup, 
ignoring symbols of the type $\circ$ and $\times$ 
as well as symbols already connected. 
Repeat this process until there are no more 
$\down$'s to the left of any $\up$.
\item Put a ray under any remaining 
symbols $\down$ or $\up$.
\end{enumerate}
It is an easy observation that $\underline{\lambda}$ 
always exists for a fixed $\lambda$. Furthermore, $\lambda$ is the (unique) orientation of $\underline{\lambda}$, such that $\underline{\lambda}\lambda$ has minimal degree. Each cup diagram $c$ is of the form $\underline{\lambda}$ for $\lambda\in\Lambda$, a block compatible with $c$.

Similarly we can define $\overline{\lambda}=\underline{\lambda}^*$, and,  
as before, in an oriented circle diagram 
$\underline{\lambda}\nu\overline{\mu}$ a circle $C$ is 
said to be oriented \textit{anticlockwise} if the 
rightmost vertex contained in the circle 
is $\up$ and \textit{clockwise} 
in case its $\down$.
Two helpful facts about the degree of oriented 
circle diagrams are summarized below. 

\begin{lemma}\label{lemma:degree}
Fix a block $\Lambda$ and $\lambda,\mu,\nu \in \Lambda$.
\begin{enumerate}[label=(\alph*)]
\item The contribution to the degree of the arcs 
contained in a given circle $C$ inside an 
oriented circle diagram 
$\underline{\lambda}\nu\overline{\mu}$ is equal to
\[
{\rm deg}(C) = (\text{number of cups in } C) \pm 1,
\]
with $+1$, if the circle $C$ is oriented clockwise and $-1$ otherwise.
\item If, in an oriented circle diagram 
$\underline{\lambda}\nu\overline{\mu}$, 
one changes the orientation such that all vertices 
contained in exactly one circle $C$ are changed, 
then the degree increases by $2$, if $C$ 
was oriented anticlockwise, and decreases by 
$2$, if $C$ was oriented clockwise.\makeqed
\end{enumerate}
\end{lemma}

\begin{proof}
(a) is a special case of~\cite[Proposition~4.9]{ES2} 
while (b) follows from (a).
\end{proof}

We conclude this part with the notions of \textit{distance} 
and \textit{saddle width}, 
which will be important for spreading signs in the 
multiplication given below.

\begin{definition}\label{definition:length}
For $i\in\Z$ and a block $\Lambda$ 
define the \textit{position of} $i$ as 
\[
\pos_\Lambda(i) = \left|\{ j \, \middle| \, j \leq i, {\rm seq}(\Lambda)_j = \dummy \}\right| + 2 \left|\{ j \, \middle| \, j \leq i, {\rm seq}(\Lambda)_j = \times \}\right|.
\]
For a cup or cap $\gamma$ in a 
diagram connecting vertices $(i,0)$ and $(j,0)$ we define its \textit{distance} $\length_\Lambda(\gamma)$ and \textit{saddle width} $s_\Lambda(\gamma)$ by
\[
\length_\Lambda(\gamma) = \left| \pos_\Lambda(i) - \pos_\Lambda(j) \right| \quad\text{ respectively }\quad
s_\Lambda(\gamma)=\tfrac{1}{2}\left( \length_\Lambda(\gamma) +1\right).
\]
For a ray $\gamma$ set $\length_\Lambda(\gamma)=0$. For a collection 
$M=\{\gamma_1,\ldots,\gamma_r\}$ of distinct arcs 
(e.g. a circle or sequence of arcs connecting two vertices) set 
\[
\length_\Lambda(M) = \sum_{1 \leq k \leq r} \length_\Lambda(\gamma_k).
\]

The saddle width will be interpreted in 
Subsection~\ref{subsec:isoofalgebras} as the 
number of phantom facets at the bottom of a saddle (e.g. $s=2$ 
for the saddle from~\eqref{eq:ssaddle}).

We omit the subscript $\Lambda$, if no confusion can arise. 
\end{definition}

\subsection{The Blanchet-Khovanov algebras as graded \texorpdfstring{$\field$}{K}-vector spaces}\label{subsec:BKalg}

Fix a block $\Lambda\in\block$, and 
consider the \textit{basis set of 
oriented circle diagrams}
\[
\mathbb{B}(\Lambda) = \left\lbrace \underline{\lambda} \nu \overline{\mu} \, \, \middle| \, \, \underline{\lambda} \nu \overline{\mu} \text{ is oriented and } \lambda,\mu,\nu \in \Lambda \right\rbrace.
\]
This set is subdivided into smaller sets 
of the form ${}_\lambda\mathbb{B}(\Lambda)_\mu$ 
which are those diagrams in $\mathbb{B}(\Lambda)$ 
which have $\underline{\lambda}$ as cup part and $\overline{\mu}$ as 
cap part.

From now on, we restrict to circle diagrams 
that only contain cups and caps. Formally this is done as follows: for a 
block $\Lambda\in\block$ denote by $\Lambda^\circ$ 
the set of weights $\lambda$ such that 
$\underline{\lambda}$ only contains cups. Note that 
$\Lambda^\circ \neq \emptyset$ iff $\Lambda$ is balanced. Define
\begin{equation}\label{eq:basis2}
\mathbb{B}^\circ(\Lambda) = \left\lbrace \underline{\lambda} \nu \overline{\mu} \, \, \middle| \, \, \underline{\lambda} \nu \overline{\mu} \text{ is oriented and } \lambda,\mu \in \Lambda^\circ,\nu \in \Lambda \right\rbrace = \bigcup_{\lambda,\mu \in \Lambda^\circ} {}_\lambda\mathbb{B}(\Lambda)_\mu.
\end{equation}

\begin{definition}\label{definition:bkalgebra}
The \textit{Blanchet-Khovanov algebra} $\Arcalg_{\Lambda}$ attached to a block 
$\Lambda\in\bblock$ and the 
(full) \textit{Blanchet-Khovanov algebra} 
$\Arcalg$ are the graded $\field$-vector space
\[
\Arcalg_\Lambda = \left\langle \mathbb{B}^\circ(\Lambda) \right\rangle_\field = \!\!\!\!\bigoplus_{(\underline{\lambda} \nu \overline{\mu}) \in \mathbb{B}^\circ(\Lambda)}\!\!\!\! \field (\underline{\lambda} \nu \overline{\mu}),
\quad\quad
\Arcalg=\bigoplus_{\Lambda \in \bblock }\Arcalg_\Lambda,
\]
with multiplication $\boldsymbol{\mathrm{mult}}$ given in Subsection~\ref{sec:multiplication}. Denote also by ${}_\lambda(\Arcalg_\Lambda)_\mu$ the span of the basis vectors inside ${}_\lambda\mathbb{B}(\Lambda)_\mu$.
\end{definition}

\begin{proposition}\label{proposition:BlanchetKhovanov}
The map 
$\boldsymbol{\rm mult}\colon\Arcalg_\Lambda \otimes \Arcalg_\Lambda \rightarrow \Arcalg_\Lambda$ given in Subsection~\ref{sec:multiplication} endows $\Arcalg_\Lambda$ 
with the structure of a graded, unital algebra 
with pairwise orthogonal, primitive idempotents 
${}_\lambda \mathbbm{1}_\lambda = \underline{\lambda}\lambda \overline{\lambda}$ 
for $\lambda \in \Lambda$ and unit 
$\mathbbm{1} = \sum_{\lambda \in \Lambda} {}_\lambda \mathbbm{1}_\lambda$.
Similar for (the locally unital) algebra $\Arcalg$.
\makeqed
\end{proposition}

\begin{proof}
The maps $\boldsymbol{\rm mult}_{D_l,D_{l+1}}$ 
are homogeneous of degree $0$ by~\cite[Proposition~5.19]{ES2}, since 
the proof is diagrammatic and independent of any signs or 
coefficients. The proof that the ${}_\lambda \mathbbm{1}_\lambda$ 
are idempotents is the same as in~\cite[Theorem~6.2]{ES2}, 
since multiplying them only involves merges of non-nested 
circles, in which case the map 
$\boldsymbol{\rm mult}_{D_l,D_{l+1}}$ agrees with the one defined 
in~\cite[Section~5]{ES2}. That they are pairwise 
orthogonal and primitive is clear by definition.
\end{proof}

\begin{remark}
Note that so far we do not know whether $\Arcalg_\Lambda$ 
is associative.
It will follow from the identification 
of $\Arcalg_\Lambda$ with 
$\webalg_{\vec{k}}$ 
that $\boldsymbol{\rm mult}$ is independent of the chosen 
order in which the surgeries are performed and that 
$\Arcalg_\Lambda$ is associative, see Corollary~\ref{corollary:matchalgebras1}. 
Alternatively, the independence from 
the chosen order and associativity can be shown in the 
same spirit as~\cite[Theorem 5.34]{ES2}.
\end{remark}

\subsection{Multiplication of the Blanchet-Khovanov algebra}\label{sec:multiplication}

The multiplication on $\Arcalg_\Lambda$ is 
based on the one defined in~\cite{Khov} and 
used in~\cite{BS1}. We will first recall the 
maps used in each step, which are the same as 
in~\cite{BS1} and afterwards go into details 
about how we modify these maps with different sign choices.

For $\lambda,\mu,\mu^{\prime},\eta \in \Lambda^\circ$ 
we define a map $\boldsymbol{\rm mult}\colon {}_\lambda(\Arcalg_\Lambda)_\mu \otimes {}_{\mu^{\prime}}(\Arcalg_\Lambda)_\eta \rightarrow {}_\lambda(\Arcalg_\Lambda)_\eta$ as follows. If $\mu \neq \mu'$ we declare the map 
to be identically zero. Thus, assume 
that $\mu = \mu^{\prime}$, and stack the 
diagram, without orientations, 
$\underline{\mu}\overline{\eta}$ on top 
of the diagram $\underline{\lambda}\overline{\mu}$, 
creating a diagram 
$D_0=\underline{\lambda}\overline{\mu}\underline{\mu}\overline{\eta}$.
In~\cite[Definition~5.1]{ES2} such a diagram is 
called a \textit{stacked circle diagram}. Given 
such a diagram $D_l$, starting with 
$l=0$, we construct below a new diagram $D_{l+1}$ 
by choosing a certain symmetric pair of a cup and 
a cap in the middle section. If $r$ is the number 
of cups in $\underline{\mu}$, then this can be done 
a total number of $r$ times. We call this procedure 
a \textit{surgery} at the corresponding cup-cap pair.
For each such step we define below a map 
$\boldsymbol{\rm mult}_{D_{l},D_{l+1}}$. Observing 
that the space of orientations of the final diagram 
$D_r$ is equal to the space of orientations of the 
diagram $\underline{\lambda}\overline{\eta}$, we define
\[
\boldsymbol{\rm mult} = \boldsymbol{\rm mult}_{D_{r-1},D_r} \circ \ldots \circ \boldsymbol{\rm mult}_{D_{0},D_1} \colon {}_\lambda(\Arcalg_\Lambda)_\mu \otimes {}_{\mu'}(\Arcalg_\Lambda)_\eta \rightarrow {}_\lambda(\Arcalg_\Lambda)_\eta.
\]
The global map 
$\boldsymbol{\mathrm{mult}}$ is defined as the 
direct sum of all the ones defined here.
In order to make $\boldsymbol{\mathrm{mult}}$ a priori well-defined, 
we always pick the \textit{leftmost available cup-cap pair}. 
Corollary~\ref{corollary:matchalgebras1} will finally ensure 
that this fixed choice is irrelevant.

\subsubsection{The surgery procedure}
To obtain $D_{l+1}$ from 
$D_l = \underline{\lambda}c^* c\overline{\eta}$ 
(for some cup diagram $c$) choose the cup-cap pair 
with the leftmost endpoint in $c^*c$ that can be 
connected without crossing any arcs (this means that 
the cup and cap are not nested inside any other arcs). 
Cut open the cup and the cap and stitch the loose ends 
together to form a pair of vertical line segments, 
call this diagram $D_{l+1}$:
\vspace*{-.25cm}
\begin{gather*}
\xy
\xymatrix{\reflectbox{
\xy
(0,0)*{\includegraphics[scale=1]{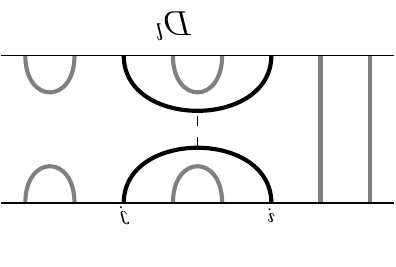}};
\endxy}\ar[r]
&
\reflectbox{\xy
(0,0)*{\includegraphics[scale=1]{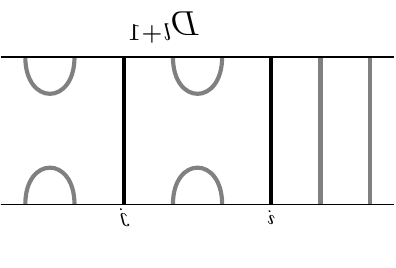}};
\endxy}
}
\endxy
\end{gather*}
\vspace*{-.75cm}
\subsubsection{The map without signs} \label{sec:multnosigns}
As in~\cite{Khov} and~\cite{BS1} the 
map $\boldsymbol{\mathrm{mult}}_{D_{l+1},D_l}$, 
without any additional signs only depends on how 
the components change when going from 
$D_l$ to $D_{l+1}$. The image of an 
orientation of $D_l$ is constructed as 
follows in these cases (where we always leave the orientations on
non-interacting arcs fixed).\newline

\noindent \textbf{Merge.} If two circles, 
say $C_1$ and $C_2$, are merged into a circle $C$ proceed as follows.

\noindent \textit{$\blacktriangleright$ If both circles 
are oriented anticlockwise}, then orient $C$ anticlockwise.

\noindent \textit{$\blacktriangleright$ If one circle 
is oriented clockwise, one is oriented anticlockwise}, then orient $C$ clockwise.

\noindent \textit{$\blacktriangleright$ If both 
circles are oriented clockwise}, then the map is zero.
\newline

\noindent \textbf{Split.} If one circle $C$ is 
replaced by two circles, say $C_1$ and $C_2$, proceed as follows.

\noindent \textit{$\blacktriangleright$ If $C$ is 
oriented anticlockwise}, then take two 
copies of $D_{l+1}$. In one copy 
orient $C_1$ clockwise and $C_2$ anticlockwise, 
in the other vice versa.

\noindent \textit{$\blacktriangleright$ If $C$ is 
oriented clockwise}, then orient both, $C_1$ and 
$C_2$, clockwise.

\begin{remark}
These are the same rules as 
in~\cite{Khov} and~\cite{BS1}. They can also be 
given by a TQFT, but the direct definition simplifies 
the introduction of signs.
\end{remark}

\subsubsection{The map with signs}\label{sec:multsigns}

In general, the formulas below include two 
types of signs. One type, which we call 
\textit{dot moving signs}, also appear in~\cite{ES2}, 
while the second, called \textit{topological signs} 
and \textit{saddle signs}, are topological in nature 
and more involved. These second types of signs will 
be given an interesting meaning in 
Subsection~\ref{subsec:proofisoofalgebras}. The main 
difference to~\cite{ES2} will be that we distinguish 
whether the two circles, that are merged together or 
split into, are nested in each other or not. Define:
\[
\op{t}(C)=\text{(a choice of) a rightmost point 
in the circle }C.
\]
Let $\gamma$ denote the cup in the 
cup-cap pair we use to perform the surgery procedure 
in this step connecting vertices $i < j$.
\newline

\noindent \textbf{Non-nested Merge.} The non-nested 
circles $C_1$ and $C_2$ are merged into $C$. 
The only case that is modified here is:

\noindent \textit{$\blacktriangleright$ One circle 
oriented clockwise, one oriented anticlockwise.} 
Let $C_k$ (for $k \in \{1,2\}$) be the clockwise 
oriented circle and let $\gamma^{\rm dot}$ be 
a sequence of arcs in $C$ connecting $\op{t}(C_k)$ 
and $\op{t}(C)$. (Neither 
$\op{t}(C_k),\op{t}(C)$ nor 
$\gamma^{\rm dot}$ are unique, but possible choices 
differ in distance by $2$, making the sign well-defined, 
see also~\cite[Lemma~5.7]{ES2}. 
Thus, the reader may choose any of these.) Proceed as in Subsection~\ref{sec:multnosigns} 
and multiply with the \textit{dot moving sign} 
\begin{equation}\label{eq:dotsign}
(-1)^{\length_\Lambda(\gamma^{\rm dot})}. 
\end{equation}

\noindent \textbf{Nested Merge.} The nested 
circles $C_1$ and $C_2$ are merged into $C$. Denote 
by $C_{\rm in}$ the inner of the two original circles. The cases are modified as follows.

\noindent \textit{$\blacktriangleright$ Both circles oriented anticlockwise.} 
Proceed as above, but multiply with
\begin{equation}\label{eq:topsign}
-(-1)^{\tfrac{1}{4}(\length_\Lambda(C_{\rm in})-2)}\cdot(-1)^{s_\Lambda(\gamma)}.
\end{equation}

\noindent \textit{$\blacktriangleright$ One circle 
oriented clockwise, one oriented anticlockwise.} Again 
perform the surgery procedure as described 
in Subsection~\ref{sec:multnosigns} and multiply with
\begin{equation*}
(-1)^{\length_\Lambda(\gamma^{\rm dot})}\cdot(-(-1)^{\tfrac{1}{4}(\length_\Lambda(C_{\rm in})-2)})\cdot(-1)^{s_\Lambda(\gamma)},
\end{equation*}
where $\gamma^{\rm dot}$ is defined as in~\eqref{eq:dotsign}.
\newline

\noindent \textbf{Non-nested Split.} The circle $C$ 
splits into the non-nested circles $C_i$, containing the 
vertices at position $i$, and $C_j$, containing the 
vertices at position $j$. For both orientations we 
introduce a \textit{dot moving sign} as well as a \textit{saddle sign} as follows.

\noindent \textit{$\blacktriangleright$ $C$ oriented anticlockwise.}
Use the map as in Subsection~\ref{sec:multnosigns}, 
but the copy where $C_i$ is oriented clockwise is multiplied with 
\[
(-1)^{\length_\Lambda(\gamma_i^{\rm ndot})}\cdot(-1)^{s_\Lambda(\gamma)},
\]
while the one where $C_j$ is oriented clockwise is multiplied with 
\[
-(-1)^{\length_\Lambda(\gamma_j^{\rm ndot})}\cdot(-1)^{s_\Lambda(\gamma)}.
\]
Here $\gamma_i^{\rm ndot}$ and $\gamma_j^{\rm ndot}$ are sequences of arcs 
connecting $(i,0)$ and ${\rm t}(C_i)$ 
inside $C_i$ respectively $(j,0)$ and ${\rm t}(C_j)$ in $C_j$ ($\mathrm{ndot}$ 
can be read as ``newly created dot'').

\noindent \textit{$\blacktriangleright$ $C$ oriented clockwise.}
In this case multiply the result with 
\[
(-1)^{\length_\Lambda(\gamma^{\rm dot})}\cdot(-1)^{\length_\Lambda(\gamma_i^{\rm ndot})}\cdot(-1)^{s_\Lambda(\gamma)}.
\]  
Here $\gamma^{\rm dot}$ is a sequence of arcs connecting 
${\rm t}(C)$ and ${\rm t}(C_j)$ in $C$ and $\gamma_i^{\rm ndot}$ is as before.
\newline

\noindent \textbf{Nested Split.} We use here the same 
notations as in the non-nested split case above and 
furthermore denote by $C_{\rm in}$ the inner 
of the two circles $C_i$ and $C_j$. The difference to the 
non-nested case is that we substitute 
the \textit{saddle sign} with a \textit{topological sign}.

\noindent \textit{$\blacktriangleright$ $C$ oriented anticlockwise.}
We use the map as defined in Subsection~\ref{sec:multnosigns}, but 
the copy where $C_i$ is oriented clockwise is multiplied with
\[
(-1)^{\length_\Lambda(\gamma_i^{\rm ndot})}\cdot(-1)^{\tfrac{1}{4}(\length_\Lambda(C_{\rm in})-2)},
\] 
while the copy where $C_j$ is oriented clockwise is multiplied with 
\[
-(-1)^{\length_\Lambda(\gamma_j^{\rm ndot})}\cdot(-1)^{\tfrac{1}{4}(\length_\Lambda(C_{\rm in})-2)}.
\]
Here $\gamma_i^{\rm ndot}$ and $\gamma_j^{\rm ndot}$ are as before.

\noindent \textit{$\blacktriangleright$ $C$ oriented clockwise.}
We multiply with
\[
(-1)^{\length_\Lambda(\gamma^{\rm dot})}\cdot(-1)^{\length_\Lambda(\gamma_i^{\rm ndot})}\cdot(-1)^{\tfrac{1}{4}(\length_\Lambda(C_{\rm in})-2)},
\]
again with the same notations as above.

\subsubsection{Examples for the multiplication}

We give below examples for some of 
the shapes that can occur 
during the surgery procedure and determine the signs. In 
all examples assume that outside of 
the shown strip all entries of the 
weights are $\circ$.

\begin{example}\label{ex:multiplication}
This example illustrates the merge situation. 
First we look at a simple merge of \textit{two 
anticlockwise, non-nested} circles. In this 
case no signs appear at all that means we have
\begin{gather}\label{eq:mult1}
\xy
\xymatrix{
\raisebox{0.09cm}{
\xy
(0,0)*{\includegraphics[scale=1]{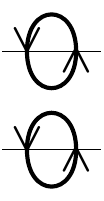}};
\endxy}
\ar[r]
&
\raisebox{0.09cm}{
\xy
(0,0)*{\includegraphics[scale=1]{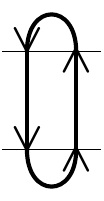}};
\endxy}
\ar[r]
&
\raisebox{0.09cm}{
\xy
(0,0)*{\includegraphics[scale=1]{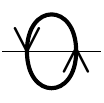}};
\endxy}
}
\endxy
\end{gather}
The rightmost step above, called \textit{collapsing}, is always performed at the end of a multiplication procedure and is omitted in what follows.

Secondly, we consider a merge of \textit{two 
anticlockwise, nested} circles. Depending on the concrete 
shape of the diagram it can produce different signs:

\begin{gather}\label{eq:mult2}
\begin{aligned}
\xy
\xymatrix@C-13pt{
\xy
(0,0)*{\includegraphics[scale=1]{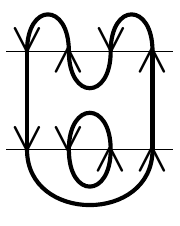}};
\endxy
\ar[r]
&
\xy
(0,0)*{\includegraphics[scale=1]{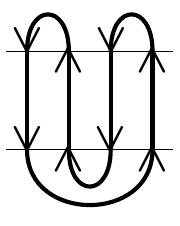}};
\endxy}
\endxy
\!\text{and}\!
\xy
\xymatrix@C-13pt{
\raisebox{-0.1cm}{
\xy
(0,0)*{\includegraphics[scale=1]{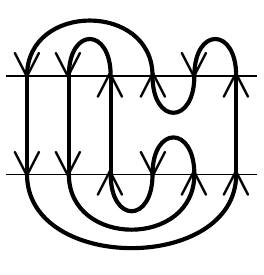}};
\endxy}
\ar[r]
&
-
\raisebox{-0.1cm}{
\xy
(0,0)*{\includegraphics[scale=1]{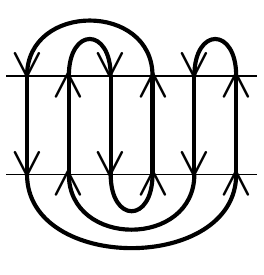}};
\endxy}
}
\endxy
\end{aligned}
\end{gather}
Note that, in contrast to~\cite{Khov},~\cite{BS1} or~\cite{ES2}, nested merges 
can come with a sign. 
\end{example}

\begin{examplen}
This example illustrated the 
two versions of a split. In 
both cases a non-nested merge is performed, followed by a 
split into \textit{two non-nested} respectively \textit{nested} circles. First, the H-shape:

\begin{gather}\label{eq:mult3}
\begin{aligned}
&\xy
\xymatrix{
\raisebox{0.09cm}{\xy
(0,0)*{\includegraphics[scale=1]{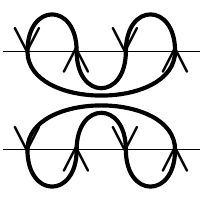}};
\endxy}
\ar[r]
&
\raisebox{0.09cm}{\xy
(0,0)*{\includegraphics[scale=1]{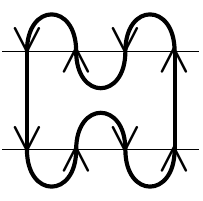}};
\endxy}
\ar[r]
&
-
\xy
(0,0)*{\includegraphics[scale=1]{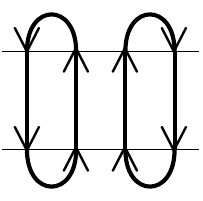}};
\endxy
-
\xy
(0,0)*{\includegraphics[scale=1]{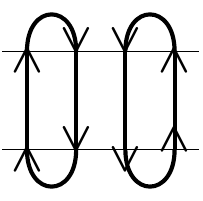}};
\endxy
}
\endxy \\
&\xy
\xymatrix{
\raisebox{0.09cm}{\xy
(0,0)*{\includegraphics[scale=1]{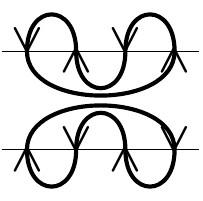}};
\endxy}
\ar[r]
&
\raisebox{0.09cm}{\xy
(0,0)*{\includegraphics[scale=1]{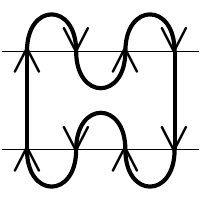}};
\endxy}
\ar[r]
&
-
\xy
(0,0)*{\includegraphics[scale=1]{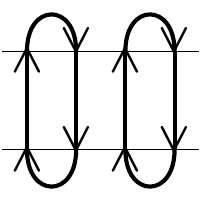}};
\endxy
}
\endxy
\end{aligned}
\end{gather}

Next, the C shape:

\begin{gather}\label{eq:mult4}
\begin{aligned}
&\xy
\xymatrix{
\raisebox{0.09cm}{\xy
(0,0)*{\includegraphics[scale=1]{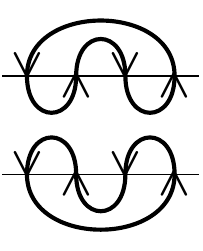}};
\endxy}
\ar[r]
&
\raisebox{0.09cm}{\xy
(0,0)*{\includegraphics[scale=1]{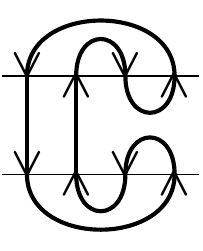}};
\endxy}
\ar[r]
&
\text{{$-$}}
\xy
(0,0)*{\includegraphics[scale=1]{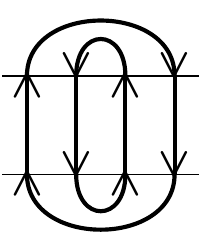}};
\endxy
+
\xy
(0,0)*{\includegraphics[scale=1]{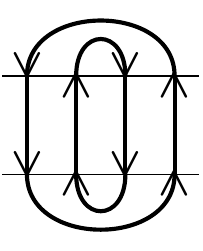}};
\endxy
}
\endxy \\
&\xy
\xymatrix{
\raisebox{0.09cm}{\xy
(0,0)*{\includegraphics[scale=1]{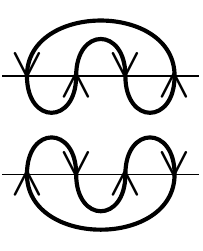}};
\endxy}
\ar[r]
&
\raisebox{0.09cm}{\xy
(0,0)*{\includegraphics[scale=1]{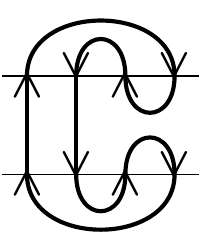}};
\endxy}
\ar[r]
&
\text{\phantom{$-$}}
\xy
(0,0)*{\includegraphics[scale=1]{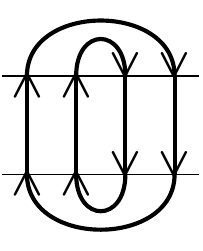}};
\endxy
}
\endxy
\hspace{2.585cm}
\text{\raisebox{-1cm}{$\makeqedtri$}}
\hspace{-2.585cm}
\end{aligned}
\end{gather}
\end{examplen}

\begin{remark}
The $\reflectbox{\text{C}}$ shape cannot appear 
as long as we impose the choice of the 
order of cup-cap pairs from left to right in the surgery 
procedure and it will not be needed in the 
proof of Theorem~\ref{proposition:matchalgebras}.
\end{remark}

\subsection{Bimodules for Blanchet-Khovanov algebras}\label{subsec:Hbimod}

To define $\Arcalg$-bimodules we need further 
diagrams moving from one block $\Lambda$ to another block $\Gamma$.

Fix two blocks $\Lambda,\Gamma\in\bblock$ 
such that ${\rm seq}(\Lambda)$ and 
${\rm seq}(\Gamma)$ coincide except at 
positions $i$ and $i+1$. Following \cite{BS2}, a
\textit{$(\Lambda,\Gamma)$-admissible 
matching (of type $\pm \alpha_i$)} is a 
diagram $t$ inside $\mathbb{R} \times [0,1]$ 
consisting of vertical lines connecting 
$(k,0)$ with $(k,1)$ if we have that ${\rm seq}(\Lambda)_k = {\rm seq}(\Gamma)_k = \dummy$ 
and, depending on the sign of $\alpha_i$, an 
arc at positions $i$ and $i+1$ of the form
\begin{gather}\label{eqn:basic_moves}
\begin{aligned}
\raisebox{0.1cm}{$\phantom{+}\alpha_i:$}&
\raisebox{0.1cm}{\xy
(0,0)*{\includegraphics[scale=1]{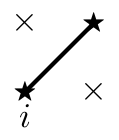}};
\endxy}
\quad,\quad
\xy
(0,0)*{\includegraphics[scale=1]{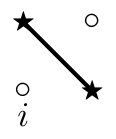}};
\endxy
\quad,\quad
\xy
(0,0)*{\includegraphics[scale=1]{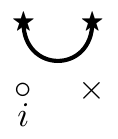}};
\endxy
\quad,\quad
\xy
(0,0)*{\includegraphics[scale=1]{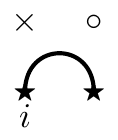}};
\endxy
\\
\raisebox{0.1cm}{$-\alpha_i:$}&
\raisebox{0.1cm}{\xy
(0,0)*{\includegraphics[scale=1]{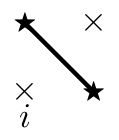}};
\endxy}
\quad,\quad
\xy
(0,0)*{\includegraphics[scale=1]{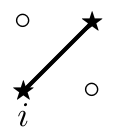}};
\endxy
\quad,\quad
\xy
(0,0)*{\includegraphics[scale=1]{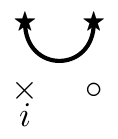}};
\endxy
\quad,\quad
\xy
(0,0)*{\includegraphics[scale=1]{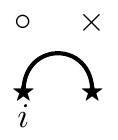}};
\endxy
\end{aligned}
\end{gather}
where we view ${\rm seq}(\Lambda)$ as 
decorating the integral points of
$\mathbb{R} \times \{ 0\}$ and 
${\rm seq}(\Gamma)$ as decorating the 
integral points of $\mathbb{R} \times \{ 1\}$. 
Again, the first two moves in each row are called 
\textit{rays}, the third ones \textit{cups} and the last ones \textit{caps}.

For $t$ a $(\Lambda,\Gamma)$-admissible 
matching, $\lambda \in \Lambda$, and 
$\mu \in \Gamma$ we say that $\lambda t \mu$ 
is \textit{oriented} if cups respectively 
caps connect one $\up$ and one $\down$ 
in $\lambda$ respectively $\mu$, 
and rays connect the same symbols in 
$\lambda$ and $\mu$.
For $\mathbf{\Lambda} = (\Lambda_0,\ldots,\Lambda_r)$ 
a sequence of blocks a 
\textit{$\mathbf{\Lambda}$-admissible composite matching} 
is a sequence of diagrams 
$\mathbf{t}=(t_1,\ldots,t_r)$ such that $t_k$ 
is a $(\Lambda_{k-1},\Lambda_k)$-admissible 
matching of some type. We view the sequence of matchings as 
being stacked on top of each other. 
A sequence of weights $\lambda_i \in \Lambda_i$ 
such that $\lambda_{k-1} t_k \lambda_k$ is 
oriented for all $k$ is an \textit{orientation} 
of the $\mathbf{\Lambda}$-admissible 
composite matching $\mathbf{t}$. 
For short, we tend to drop the word admissible, 
since the only matchings we consider are 
admissible. Moreover, if only start $\Lambda=\Lambda_0$ 
and end $\Gamma=\Lambda_r$ matter, then 
we call $\mathbf{t}$ short a 
$(\Lambda,\Gamma)$\textit{-composite matching}.

We stress that $\mathbf{\Lambda}$-composite matching can contain lines, 
in contrast to the diagrams we consider for $\Arcalg_{\Lambda}$ 
and $\Arcalg$.

\begin{example}\label{example:EFoncupdiagrams}
Below is a $(\Lambda_0,\Lambda_5)$-composite 
matching. Assume that outside 
of the indicated areas all symbols 
of the block sequences are equal to $\circ$.
\begin{gather*}
\xy
(0,0)*{\includegraphics[scale=1]{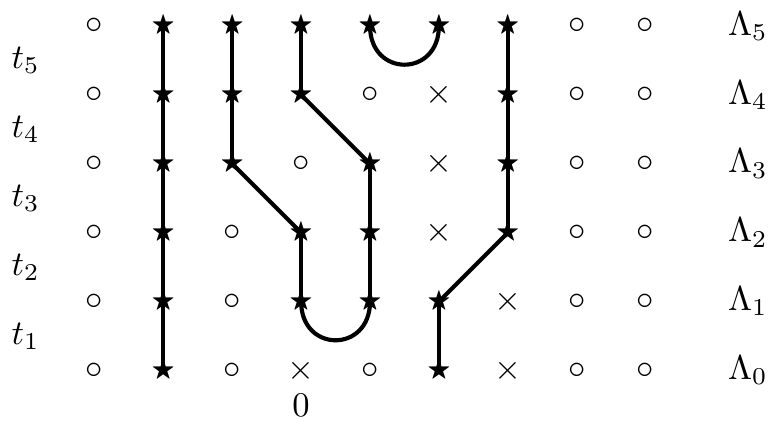}};
\endxy
\end{gather*}
The types of the matchings 
are $-\alpha_0,\alpha_2,\alpha_{-1},\alpha_0,\alpha_1$ (read from bottom to top).
\end{example}

We now want to consider bimodules 
between Blanchet-Khovanov algebras 
for different blocks, or said differently, 
bimodules for the algebra $\Arcalg$.

We start by defining a basis of the 
underlying $\field$-vector space.
To a $\mathbf{\Lambda}$-composite matching 
$\mathbf{t}$ we again associate a set of 
diagrams from which to create a $\field$-vector space 
(its elements 
are called \textit{stretched circle diagrams})
\begin{equation}\label{eq:basis-bimodule}
\mathbb{B}^\circ(\mathbf{\Lambda},\mathbf{t}) = \left\lbrace \underline{\lambda} (\mathbf{t},\boldsymbol{\nu}) \overline{\mu} \, \,
\middle| \, \, \begin{array}{l}
\lambda \in  \Lambda_0^\circ,\, \mu \in \Lambda_r^\circ,\, \boldsymbol{\nu} = (\nu_0, \ldots, \nu_r) \text{ with } \nu_i \in \Lambda_i, \\
\underline{\lambda}\nu_0 \text{ oriented, } \nu_r \overline{\mu} \text{ oriented, }\\
\nu_{i-1} t_i \nu_i \text{ oriented for all } 1 \leq i \leq r.
\end{array}
\right\rbrace
\end{equation}
As before we obtain the set 
$\mathbb{B}(\mathbf{\Lambda},\mathbf{t})$ by 
allowing $\lambda \in \Lambda_0$ and $\mu \in \Lambda_r$ 
in~\eqref{eq:basis-bimodule}.

\begin{examplen}\label{eq:bimoduleexample}
Let $\Lambda$ be the block with 
block sequence \makebox[1em]{$\dummy$} \makebox[1em]{$\dummy$} \makebox[1em]{$\circ$} \makebox[1em]{$\times$} 
at positions $0$, $1$, $2$, $3$
and $\Gamma$ the block with 
sequence \makebox[1em]{$\dummy$} \makebox[1em]{$\dummy$} \makebox[1em]{$\dummy$} \makebox[1em]{$\dummy$} at the same positions 
(both with $\circ$ everywhere else) 
and assume both blocks are balanced. Then an 
example for a $(\Lambda,\Gamma)$-matching 
of 
type $\alpha_1$ is the third diagram 
in the first row of~\eqref{eqn:basic_moves} 
denoted here by $t_1$. Taking this as our
composite matching we obtain a 
$\field$-vector space of dimension $6$ with basis consisting of
\begin{gather*}
\xy
\xymatrix{
\xy
(0,0)*{\includegraphics[scale=1]{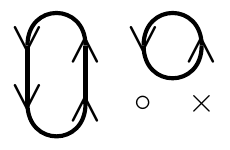}};
\endxy
&
\xy
(0,0)*{\includegraphics[scale=1]{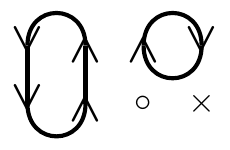}};
\endxy
&
\xy
(0,0)*{\includegraphics[scale=1]{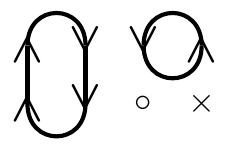}};
\endxy
&
\xy
(0,0)*{\includegraphics[scale=1]{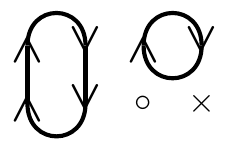}};
\endxy
}\endxy
\\
\xy
\xymatrix{
\xy
(0,0)*{\includegraphics[scale=1]{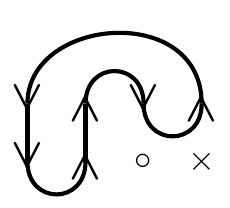}};
\endxy
&
\xy
(0,0)*{\includegraphics[scale=1]{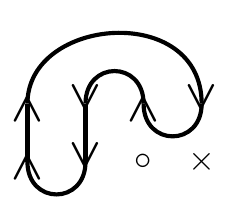}};
\endxy
}
\endxy
\hspace{3.25cm}
\text{\raisebox{-1cm}{$\makeqedtri$}}
\hspace{-3.25cm}
\end{gather*}
\end{examplen}

For a basis vector 
$\underline{\mu}(\mathbf{t},\boldsymbol{\nu})\overline{\eta}$ denote by $\overline{\eta}^\downarrow$ its \textit{downwards 
reduction}. This is the cap diagram obtained by 
stacking the diagrams $t_1,\ldots,t_r,\overline{\eta}$ 
on top of each other from left to right, removing any 
components in this stacked diagram that are 
not connected to the bottom line of $t_1$, 
and replacing all components that are connected to 
the bottom in two vertices by a cap connecting the 
vertices. It is clearly independent of 
$\underline{\mu}$. Analogously, define its \textit{upward 
reduction} $\underline{\mu}^\uparrow$, a cup 
diagram independent of $\overline{\eta}$.

\begin{definition}\label{definition:bimodules}
Let $\mathbf{t}$ be a $\mathbf{\Lambda}$-composite 
matching for $\mathbf{\Lambda}=(\Lambda_0,\ldots,\Lambda_r)$. Set
\[
\Arcalg(\mathbf{\Lambda},\mathbf{t}) = \left\langle \mathbb{B}^\circ(\mathbf{\Lambda},\mathbf{t}) \right\rangle_\field
\]
as a graded $\field$-vector space. Hereby
the \textit{degree} of a basis element 
itself is, by definition, minus the 
number of its anticlockwise circles plus the number of its
clockwise circles.

For a basis element 
$\underline{\lambda} \nu \overline{\mu} \in \Arcalg_{\Lambda_0}$ define 
a left (bottom) action via
\[
(\underline{\lambda} \nu \overline{\mu})(\underline{\mu}(\mathbf{t},\boldsymbol{\nu})\overline{\eta}) = \sum_{\boldsymbol{\nu^{\prime}}} a_{\boldsymbol{\nu^{\prime}}} \underline{\lambda}(\mathbf{t},\boldsymbol{\nu^{\prime}})\overline{\eta},
\]
with the coefficients $a_{\mathbf{\nu^{\prime}}}$ 
given by the multiplication
\[
(\underline{\lambda} \nu \overline{\mu})(\underline{\mu}\nu_0 \overline{\eta}^\downarrow) = \sum_{\nu^{\prime}} a_{\nu^{\prime}} \underline{\lambda}\nu^{\prime}\overline{\eta}^\downarrow.
\]
Then $a_{\boldsymbol{\nu^{\prime}}} = a_{\nu^{\prime}}$ 
for 
$\boldsymbol{\nu^{\prime}}$ the unique orientation 
with $\nu_0^{\prime}=\nu^{\prime}$ and all 
components not connected to the bottom line of $t_1$ 
have 
the same orientation as in 
$\underline{\mu}(\mathbf{t},\boldsymbol{\nu})\overline{\eta}$. 
Analogously define a right (top) action of $\Arcalg_{\Lambda_r}$ 
using the upwards reduction $\underline{\mu}^\uparrow$.
\end{definition}

\begin{example}\label{example:degree-module}
The basis elements (read from left to right) 
from Example~\ref{eq:bimoduleexample} 
have degrees $-2,0,0,2$ (top row) and $-1,1$ (bottom row).
\end{example}


It is not clear that 
the above actions are well-defined and commute and we 
need the translation between $\webalg$ and $\Arcalg$ 
from Section~\ref{sec:fromKhtofoams} to prove it.

\begin{proposition} \label{prop:BKbimodule}
Let $\mathbf{t}$ be a $\mathbf{\Lambda}$-composite 
matching for $\mathbf{\Lambda}=(\Lambda_0,\ldots,\Lambda_r)$. 
Then the left action of $\Arcalg_{\Lambda_0}$ and the right action of 
$\Arcalg_{\Lambda_r}$
on $\Arcalg(\mathbf{\Lambda},\mathbf{t})$ are well-defined and commute. 
Hence, $\Arcalg(\mathbf{\Lambda},\mathbf{t})$ 
is a $\Arcalg_{\Lambda_0}$-$\Arcalg_{\Lambda_r}$-bimodule 
(and thus, a $\Arcalg$-bimodule).\makeqed
\end{proposition}

\begin{proof}
Using Theorem~\ref{proposition:matchalgebras}, 
we 
obtain an isomorphism of graded algebras of 
$\Arcalg_\Lambda$ with $\webalg^{\circ}_{\vec{k}}$. An 
isomorphism of graded $\field$-vector spaces 
of $\Arcalg(\mathbf{\Lambda},\mathbf{t})$ 
with $\M(\web(\mathbf{\Lambda},\mathbf{t}))$ is obtained 
by using Lemma~\ref{lemma:matchvs2}. This 
isomorphism intertwines 
the actions of the two algebras on the bimodules by construction 
and hence, proves the claim.
\end{proof}

We introduce a slight 
generalization of the notion of an 
admissible matching, the so-called 
\textit{empty 
moves} (of which the reader should 
think of switching neighboring $\circ$ and $\times$). 
This means that in the list of local 
moves~\eqref{eqn:basic_moves} we also allow the following:

\begin{gather}\label{eqn:basic_moves2}
\begin{aligned}
\raisebox{0.1cm}{$2\alpha_i:$}
\xy
(0,0)*{\includegraphics[scale=1]{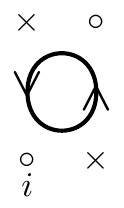}};
\endxy
\quad \quad \quad \quad
\raisebox{0.1cm}{$-2\alpha_i:$}
\xy
(0,0)*{\includegraphics[scale=1]{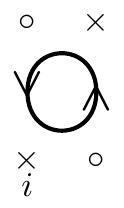}};
\endxy
\end{aligned}
\end{gather}
If a composite matching $\mathbf{t}$
contains empty moves, $\Arcalg(\boldsymbol{\Lambda},\mathbf{t})$ 
is constructed as follows. Take the 
composite matching $\mathbf{t}^{\prime}$ 
that is obtained by substituting each empty 
move by the composition of a cup and cap local 
move such that its fits with the two blocks. Then take the submodule spanned by 
those basis elements such that the internal 
circles are all oriented anticlockwise. 
It is evident that this will be a submodule of the full bimodule where 
all orientations are allowed. Finally, 
shift the module up by the total number of empty moves in $\mathbf{t}$.

\begin{example}\label{example:emptymoves}
The pictures below are wildcards 
for the $\Arcalg$-bimodules defined via 
the illustrated matchings.
\begin{gather*}
\begin{aligned}
\raisebox{0.21cm}{\xy
(0,0)*{\includegraphics[scale=1]{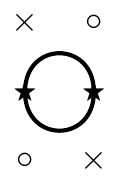}};
\endxy}
\cong
\xy
(0,0)*{\includegraphics[scale=1]{figs/BK-algebras-figure362.pdf}};
\endxy
\! \{+1\} \oplus
\xy
(0,0)*{\includegraphics[scale=1]{figs/BK-algebras-figure362.pdf}};
\endxy
\! \{-1\}
\;,\;
\raisebox{0.21cm}{\xy
(0,0)*{\includegraphics[scale=1]{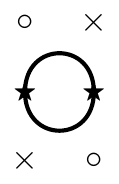}};
\endxy}
\cong
\xy
(0,0)*{\includegraphics[scale=1]{figs/BK-algebras-figure363.pdf}};
\endxy
\! \{+1\} \oplus
\xy
(0,0)*{\includegraphics[scale=1]{figs/BK-algebras-figure363.pdf}};
\endxy
\! \{-1\}
\end{aligned}
\end{gather*}
The isomorphisms of $\Arcalg$-bimodules are evident by the construction above.
\end{example}

\begin{proposition}\label{proposition-webbimodulesagain}
The $\Arcalg$-bimodules
$\Arcalg(\boldsymbol{\Lambda},\boldsymbol{\mathrm{t}})$ are finite-dimensional, 
graded biprojective $\Arcalg$-bimodules.\makeqed
\end{proposition}

\begin{proof}
Clearly, $\Arcalg(\boldsymbol{\Lambda},\boldsymbol{\mathrm{t}})$ are
finite-dimensional $\Arcalg$-bimodules. 
That they are graded as $\Arcalg$-bimodules 
follows by the identification 
with web bimodules from 
Lemma~\ref{lemma:matchvs2}.
We only prove here projectivity for the 
left action, the right action is done 
analogously. Denote by $\Lambda$ 
and $\Gamma$ the first 
and last block in 
$\boldsymbol{\Lambda}$. For 
any $\mu \in \Gamma$ denote 
by ${}_{\mu^\downarrow}\mathbbm{1}_{\mu^\downarrow}$ the 
idempotent in $\Arcalg_\Lambda$ corresponding 
to the downwards reduction of $\mu$. Then, 
as an $\Arcalg_\Lambda$-module, $\Arcalg(\boldsymbol{\Lambda},\boldsymbol{\mathrm{t}})$ decomposes as a direct sum of modules 
of the form $\Arcalg_\Lambda\cdot({}_{\mu^\downarrow}\mathbbm{1}_{\mu^\downarrow})$, 
which are projective $\Arcalg_\Lambda$-modules, 
proving the claim.
\end{proof}

This proposition again motivates the 
definition of the following 
$2$-category which is, as before, one 
of the main objects that we 
are going to study.

\begin{definition}\label{definition:catbimodules}
Given $\Arcalg$ as above, 
let $\Modpgr{\Arcalg}$ be the following $2$-category:
\begin{itemize}
\item Objects are the various $\Lambda\in\bblock$.
\item Morphisms are finite sums and 
tensor products (taken over $\Arcalg$) 
of the $\Arcalg$-bimodules 
$\Arcalg(\boldsymbol{\Lambda},\boldsymbol{\mathrm{t}})$.
\item The composition 
of $\Arcalg$-bimodules is given by tensoring (over $\Arcalg$).
\item $2$-morphisms are $\Arcalg$-bimodule 
homomorphisms.
\item The vertical composition of $\Arcalg$-bimodule 
homomorphisms is 
the usual composition and the horizontal composition is 
given by tensoring (over $\Arcalg$).
\end{itemize}
We consider $\Modpgr{\Arcalg}$ as a graded $2$-category with 
$2$-hom-spaces as in~\eqref{eq:degreehom}.
\end{definition}

In analogy 
to $\sltwowebcat$ from 
the end of Subsection~\ref{sub:Blanchet}, 
the objects 
and morphisms 
(when seen as composite matchings) 
in $\Modpgr{\Arcalg}$ can be seen as a $\field(q)$-linear category, 
which we denote by $\cupcat$.

%% file: res/4-fromKhtofoams.tex
In this section we 
assume 
that all appearing $\vec{k}$'s and $\Lambda$'s are balanced.
Our goal now is to construct an isomorphism 
of graded algebras
$\Iso{}{}\colon\webalg^{\circ}\to\Arcalg$ 
(where $\webalg^{\circ}$ 
is a certain subalgebra of $\webalg$ 
defined in~\eqref{eq:newwebalgs}). 
This isomorphism $\Iso{}{}$ induces the following:

\begin{theorem}\label{theorem:matchalgebras}
There is an equivalence of graded, $\field$-linear
$2$-categories 
\[
\boldsymbol{\Iso{}{}}\colon\webcat\stackrel{\cong}{\longrightarrow}\Modpgr{\Arcalg}
\] 
under which the $\webalg$-bimodules 
$\M(\web(\mathbf{\Lambda},\mathbf{t}))$ 
(with 
$\web(\mathbf{\Lambda},\mathbf{t})$ 
defined in Definition~\ref{definition-yes-another-one} below) are 
identified with the 
$\Arcalg$-bimodules $\Arcalg(\mathbf{\Lambda},\mathbf{t})$.\makeqed
\end{theorem}

\subsection{Some useful lemmas}\label{subsec:lemmas}

In order to prove Theorem~\ref{theorem:matchalgebras}, 
we need some lemmas.

\begin{lemma}\label{lemma:homonedim}
Let $\ell\in\Z_{\geq 0}$.
Then $\dim_{\field}(\twoEnd_{\F}(\oneinsert{2\omega_{\ell}}))=1$.\makeqed
\end{lemma}

\begin{proof}
Note that the identity foam on $\oneinsert{2\omega_{\ell}}$ 
(i.e. $\ell$ parallel phantom facets)
is an element of $\twoEnd_{\F}(\oneinsert{2\omega_{\ell}})$ 
which shows that $\dim_{\field}(\twoEnd_{\F}(\oneinsert{2\omega_{\ell}}))\geq 1$. 
Now, given any $f\in\twoEnd_{\F}(\oneinsert{2\omega_{\ell}})$, 
denote by $g\in\twoEnd_{\overline{\F}}(\emptyset)$ 
(where $\emptyset$ denotes the empty web) the closed foam obtained from 
$f$ by closings the $\ell$ bottom and top 
phantom facets to 
phantom spheres. Note that the question 
whether there are relations 
to reduce $f$ to a scalar multiple of the 
identity foam 
on $\oneinsert{2\omega_{\ell}}$ is the 
same 
as the question whether $g$ can be evaluated 
to a scalar. But the latter 
can be achieved using the 
relations in $\overline{\F}$ (which can be shown 
similar as 
in~\cite[Proposition~5]{Khova}). Thus, every 
such $f$ is a scalar multiple 
of the identity foam on $\oneinsert{2\omega_{\ell}}$ 
which shows that $\dim_{\field}(\twoEnd_{\F}(\oneinsert{2\omega_{\ell}}))\leq 1$.
\end{proof}

\begin{lemma}\label{lemma:invertible}
The following foams are locally 
one-sided invertible in $\foamcat$ 
(the ones in the first column 
have a right inverse obtained by gluing from the bottom, the ones in the 
second column a left inverse obtained by gluing from the top):
\begin{gather}\label{eq:capcupfoams}
\begin{aligned} f_1=
\xy
(0,0)*{\includegraphics[scale=1]{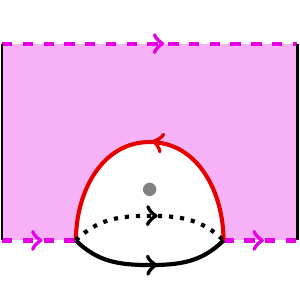}};
\endxy
\quad&,\quad f_3=
\xy
(0,0)*{\includegraphics[scale=1]{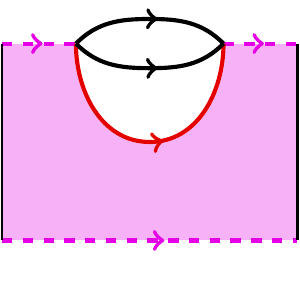}};
\endxy
\\
f_2=
\xy
(0,0)*{\includegraphics[scale=1]{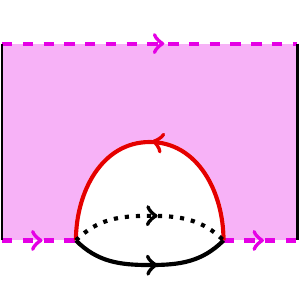}};
\endxy
\quad&,\quad f_4=
\xy
(0,0)*{\includegraphics[scale=1]{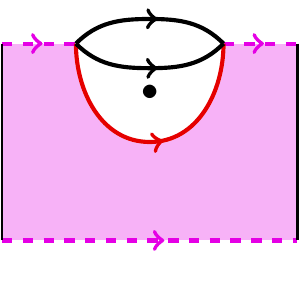}};
\endxy
\end{aligned}
\end{gather}
Similarly with the dots moved to the opposite facets. 
This implies locally that
\begin{gather}\label{eq:capcupwebs}
\xy
\xymatrix{
\raisebox{0.09cm}{\xy
(0,0)*{\includegraphics[scale=1]{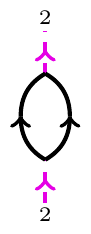}};
\endxy}
\quad
\ar[rr]^-{\left(\begin{array}{c}
f_1 \\ f_2
\end{array}\right)} & & \quad
\xy
(0,0)*{\includegraphics[scale=1]{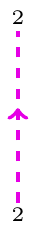}};
\endxy
\;\{+1\}
\oplus
\xy
(0,0)*{\includegraphics[scale=1]{figs/fig57.pdf}};
\endxy
\;\{-1\} \quad \ar[rr]^-{\left(\begin{array}{cc}
-f_3 & f_4
\end{array}\right)} & & \quad
\xy
(0,0)*{\includegraphics[scale=1]{figs/fig56.pdf}};
\endxy
}
\endxy
\end{gather}
induce isomorphisms in $\webcat$ between web bimodules.\makeqed
\end{lemma}

\begin{proof}
The statement of one-sided invertibility 
follows from the top bubble removals~\eqref{eq:bubble1} 
(by stacking the foams in the first 
column on top of the foams in the second column). 
The invertibility of the foams 
with a different dot placement follows from the above 
and the dot migrations~\eqref{eq:dotmigration}.
The claim that the given morphisms are isomorphisms follows 
from composing them and using both bubble removals~\eqref{eq:bubble1} 
and~\eqref{eq:bubble2} as well as the neck cutting relation~\eqref{eq:neckcut}.
\end{proof}

The next lemmas say that isotopic webs give isomorphic web modules.

\begin{lemma}\label{lemma:isoforwebs}
We have locally
\begin{gather}\label{eq:squareswitch}
\begin{aligned}
\raisebox{0.09cm}{\xy
(0,0)*{\includegraphics[scale=1]{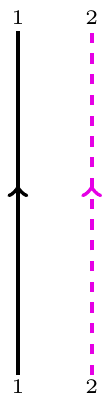}};
\endxy}
\;\cong\;
\xy
(0,0)*{\includegraphics[scale=1]{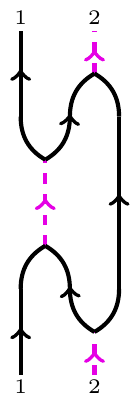}};
\endxy
\quad\text{and}\quad
\xy
(0,0)*{\reflectbox{\includegraphics[scale=1]{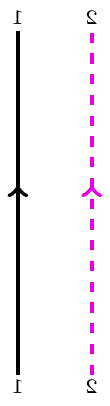}}};
\endxy
\;\cong\;
\xy
(0,0)*{\reflectbox{\includegraphics[scale=1]{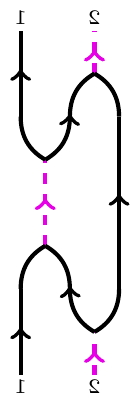}}};
\endxy
\end{aligned}
\end{gather}
which are isomorphisms in $\webcat$ between web bimodules.\makeqed
\end{lemma}

\begin{proof}
The proof is similar to the one of Lemma~\ref{lemma:invertible}: we 
can cap the bulge part 
respectively cup the two straight lines 
using the evident (undotted) foams. 
Then the ordinary-to-phantom neck cutting 
relations~\eqref{eq:neckcutphantom} provide the isomorphisms.
\end{proof}

\begin{lemma}\label{lemma:isoforwebs1}
We have locally (where we use ``rectangular'' diagrams as in~\eqref{eq:EF})
\begin{gather*}
\begin{aligned}
\raisebox{0.075cm}{\xy
(0,0)*{\includegraphics[scale=1]{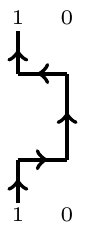}};
\endxy}
\cong
\xy
(0,0)*{\includegraphics[scale=1]{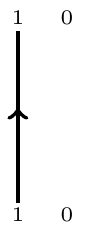}};
\endxy
\quad\text{and}\quad
\xy
(0,0)*{\reflectbox{\includegraphics[scale=1]{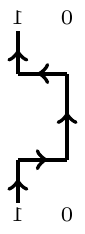}}};
\endxy
\cong
\xy
(0,0)*{\reflectbox{\includegraphics[scale=1]{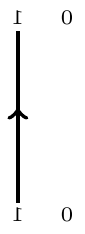}}};
\endxy
\end{aligned}
\end{gather*}
\begin{equation*}
\xy
(0,0)*{\includegraphics[scale=1]{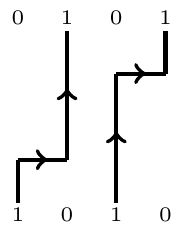}};
\endxy
\cong
\xy
(0,0)*{\includegraphics[scale=1]{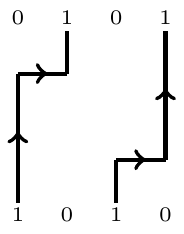}};
\endxy
\quad\text{and}\quad
\xy
(0,0)*{\includegraphics[scale=1]{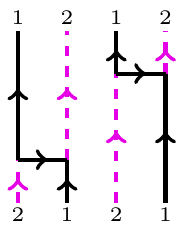}};
\endxy
\cong
\xy
(0,0)*{\includegraphics[scale=1]{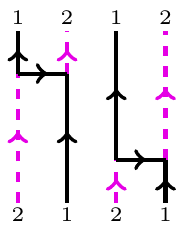}};
\endxy
\quad\text{etc.}
\end{equation*}
which are isomorphisms in $\webcat$ between web bimodules.
Analogously for other isotopies of webs.\makeqed
\end{lemma}

\begin{proof}
This is evident.
\end{proof}

Given a web $u$, then we denote by $\hat u$ the \textit{topological web} obtained 
from $u$ by forgetting orientations, labels and phantom edges, i.e. we have locally
\[
\xy
(0,0)*{\includegraphics[scale=1]{figs/fig01.pdf}};
\endxy
\mapsto
\xy
(0,0)*{\includegraphics[scale=1]{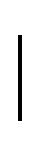}};
\endxy
\quad,\quad
\xy
(0,0)*{\includegraphics[scale=1]{figs/fig03.pdf}};
\endxy
\mapsto
\emptyset
\quad,\quad
\xy
(0,0)*{\includegraphics[scale=1]{figs/fig05.pdf}};
\endxy
\mapsto
\xy
(0,0)*{\includegraphics[scale=1]{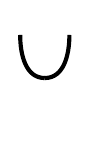}};
\endxy
\quad,\quad
\xy
(0,0)*{\includegraphics[scale=1]{figs/fig06.pdf}};
\endxy
\mapsto
\xy
(0,0)*{\includegraphics[scale=1]{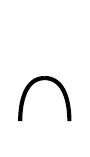}};
\endxy
\]
The topological webs are just non-crossing arcs and (closed) circles.
We call any web $u$ such that $\hat u$ is 
topologically a circle also a \textit{circle}. 
Similarly for webs $u$ such that $\hat u$ is 
topologically a cup, cap or a line.
Moreover, we assume that the leftmost non-zero entries 
of $\Lambda\in\bblock$ and $\vec{k}\in\bY$ are at 
the same position in what follows.

\begin{lemma}\label{lemma:matchcups}
Let $\lambda\in\Lambda$ with $j$ entries equal to $\times$. Consider 
the cup diagram $\underline{\lambda}$. 
Assume that $\underline{\lambda}$ does have $\ell^{\prime}$ rays 
and in total $\ell+\ell^{\prime}-j$ components. Then there exists 
a cup-ray web $u\in\CUPRAY(\vec{k})$
such that the entries of $\vec{k}$ sum up to
$2\ell+\ell^{\prime}$, and such that 
the topological web 
$\hat u$ is $\underline{\lambda}$.\makeqed
\end{lemma}

The sequences $\Lambda\in\bblock$ and $\vec{k}\in\bY$ are related 
by the bijection from~\eqref{eq:identification} below.

\begin{proof}
Given a cup diagram $\underline{\lambda}$ with 
$\ell+\ell^{\prime}-j$ 
components nested in any order, we can 
generate it using $E_i^{(r)}$'s 
and $F_i^{(r)}$'s acting on the sequence $\omega_{\ell+\ell^{\prime}}+\omega_{\ell}$. 
It is clear from~\eqref{eq:EF} that 
we can open one cup for each entry $2$ 
of $\omega_{\ell+\ell^{\prime}}+\omega_{\ell}$. 
By using
\begin{gather}\label{eq:Fgeneratedwebs}
\xy
(0,0)*{\includegraphics[scale=1]{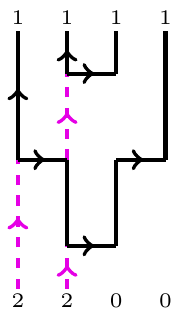}};
\endxy
\rightsquigarrow
\xy
(0,0)*{\includegraphics[scale=1]{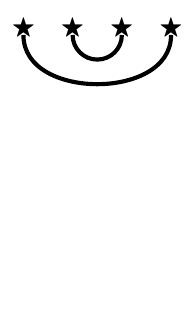}};
\endxy
\quad\text{ and }\quad
\xy
(0,0)*{\includegraphics[scale=1]{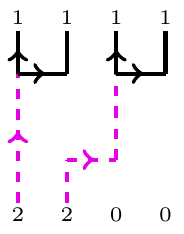}};
\endxy
\rightsquigarrow
\xy
(0,0)*{\includegraphics[scale=1]{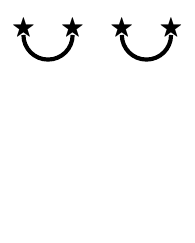}};
\endxy
\end{gather}
we can nest them and place them next 
to each other 
in any order we like (by locally 
shifting everything in place as above). Last, by using the 
shifts from~\eqref{eq:EF}, we can move the remaining 
entries $1$ to form the rays, and move the entries $2$ to the 
positions of the $\times$ symbols while building 
the cup parts of $\underline{\lambda}$.
\end{proof}

\begin{lemma}\label{lemma:matchcups2}
For each $\mathbf{\Lambda}$-composite matching $\mathbf{t}$, there exists 
a web $u\in\Hom_{\F}(\vec{k},\vec{l})$ such that 
the topological web $\hat u$ is $\mathbf{t}$.\makeqed
\end{lemma}

\begin{proof}
Mutatis mutandis as in the proof of Lemma~\ref{lemma:matchcups}.
\end{proof}

Even if we require a web $u$ to be 
$F$-generated, there are still several ways 
to built $\underline{\lambda}$ or $\mathbf{t}$. In order to 
fix one, we say an $F$-generated web $u$ 
\textit{prefers right to left} if, inductively, the component 
with rightmost right boundary point of $\underline{\lambda}$ or $\mathbf{t}$ 
is created 
from the rightmost available $2$, if its a cup, or the rightmost available $1$, 
if its a ray (using a minimal number of possible moves). We 
create the right boundary point of a cup before the left. 
In the whole procedure we avoid creating circles.

\begin{lemma}\label{lemma:matchcups3}
In the set-up of Lemma~\ref{lemma:matchcups} 
or of Lemma~\ref{lemma:matchcups2} we can make $u$ unique 
by requiring it to be $F$-generated and preferring right to left.\makeqed
\end{lemma}

\begin{proof}
An easy observation shows that we do not need $E_i$'s and $E_i^{(2)}$'s 
in order to built $\underline{\lambda}$ or $\mathbf{t}$ 
(this holds more generally, see~\cite[Lemma~4.9]{Tub}). 
Moreover, we can always avoid creating circles.
Thus, we obtain a set $F_\mathrm{gen}$ of $F$-generated webs such that $\hat{u}$ 
gives $\underline{\lambda}$ or $\mathbf{t}$. All of these 
differ by distant commutation relations as in Lemma~\ref{lemma:isoforwebs1} (bottom 
moves) or 
Serre relations $\Phi^{\sltwowebcat}_{\mathrm{Howe}}((F_{i+1}F_i^{(2)}-F_iF_{i+1}F_i+F_i^{(2)}F_{i+1})\oneinsert{\vec{k}})=0$. 
One checks that, for fixed $\vec{k}\in\bY$, 
all Serre relations have only two non-zero terms and that the corresponding 
pictures are as in Lemmas~\ref{lemma:isoforwebs} 
and~\ref{lemma:isoforwebs1}. Hence, there is a unique web 
in $F_\mathrm{gen}$ that prefers right to left which can be shown 
by induction on the number $m$ of components of $\underline{\lambda}$ or $\mathbf{t}$. 
This is clear if $m=0$ or $m=1$. For $m>1$ remove the leftmost connected component 
from $\underline{\lambda}$ or $\mathbf{t}$ and obtain 
$\underline{\lambda}^{\prime}$ or $\mathbf{t}^{\prime}$ 
with one connected component less. We can then apply the induction 
hypothesis and we get a unique web $u^{\prime}$ such that 
$\hat{u}^{\prime}$ is $\underline{\lambda}^{\prime}$ or $\mathbf{t}^{\prime}$. 
Since we removed the leftmost component of $\underline{\lambda}$ or of $\mathbf{t}$, 
we can now just construct $u$ from $u^{\prime}$ (the 
result is unique due to our conventions for such webs).
\end{proof}

Hence, the following definition makes sense.

\begin{definition}\label{definition-yes-another-one}
Let $\Lambda\in\bblock$.
Given $\lambda\in\Lambda$, we denote by $\web(\lambda)$ the unique 
web as in Lemma~\ref{lemma:matchcups3}, and given a
$\mathbf{\Lambda}$-composite matching $\mathbf{t}$, 
we denote the unique web 
as in Lemma~\ref{lemma:matchcups3} 
by $\web(\mathbf{\Lambda},\mathbf{t})$.
\end{definition}
Examples of such webs are given in~\eqref{eq:Fgeneratedwebs}.
Moreover, the following assignment
\begin{equation}\label{eq:identification}
\vec{k}\in\bY\mapsto\Lambda\in\bblock,\quad\text{via  }\;0\mapsto\circ,
\quad 1\mapsto\dummy,\quad 2\mapsto\times,
\end{equation}
clearly defines a bijection. Here $\circ,\dummy,\times$ are entries of ${\rm seq}(\Lambda)$ 
and $\Lambda$ is determined demanding that $\Lambda$ is balanced.

The next lemma is important for the calculation of the signs that 
turn up in the multiplication procedure. For this purpose,
fix a circle $C$ in a web $\web(\lambda)\web(\mu)^{\ast}$ with 
corresponding circle $C^{\prime}$ in $\underline{\lambda}\overline{\mu}$. 
For such a circle let $\mathrm{nest}(C)$ be the total 
number of circles $C_i^{\mathrm{in}}$ nested in $C$. We 
denote by $\mathrm{ipe}(C)$ the total number of \textit{internal 
phantom edges} of 
$C$ (all such edges 
that lie in the interior of $C$, but not 
in the interior of any circle nested in $C$) and
more generally by 
\[
\mathrm{ipe}\left(C-{\textstyle\bigcup_{i=1}^{\mathrm{nest}(C)}}\; C_i^{\mathrm{in}}\right)
\]
the total number of 
internal 
phantom edges of 
$C$ 
after removing $C_i^{\mathrm{in}}$ (by 
using simplifications as isotopies,~\eqref{eq:squareswitch} and~\eqref{eq:capcupwebs}).
Recalling $\length(\cdot)$ 
from Definition~\ref{definition:length}, 
we have the following lemma:

\begin{lemma}\label{lemma-internalfacets}
Given a circle diagram 
$\underline{\lambda}\overline{\mu}$ 
(for suitable $\lambda,\mu\in\Lambda$) and its associated 
web $\web(\lambda)\web(\mu)^*$. Fix any circle 
$C^{\prime}$ in $\underline{\lambda}\overline{\mu}$ 
and denote the associated circle in $\web(\lambda)\web(\mu)^*$ 
by $C$.
Then 
\begin{gather*}
\begin{aligned}
\mathrm{ipe}(C)&=
\tfrac{1}{4}\left(\length(C^{\prime})+{\textstyle\sum_{i=1}^{\mathrm{nest}(C)}}\length(C^i_{\mathrm{in}})-2+2\cdot\mathrm{nest}(C)\right),\\
\mathrm{ipe}\left(C-{\textstyle\bigcup_{i=1}^{\mathrm{nest}(C)}}\; C_i^{\mathrm{in}}\right)&=
\tfrac{1}{4}\left(\length(C^{\prime})-2\right),
\end{aligned}
\end{gather*}
where $C^i_{\mathrm{in}}$ denotes the counterparts in 
$\underline{\lambda}\overline{\mu}$ of the 
circles $C^i_{\mathrm{in}}$ of $\web(\lambda)\web(\mu)^*$.
\end{lemma}

\begin{proof}
We prove this by induction on the 
total length of all components in question. Circles $C$ in their 
easiest form as on 
the left in~\eqref{eq:usual} below have zero internal 
phantom edges and $\length(C)=2$, which is the start of our induction. The main 
observation now is that the move in~\eqref{eq:squareswitch} 
increases the number of internal phantom edges by one and the length by four 
(no matter which side of the diagram is the internal part of the circle). 
This shows the formulas in case that $\mathrm{ipe}(C)$ has no 
nested circles. The general formulas follow similarly, where we 
note that each $\times$ within $C^{\prime}$ increases 
its length by $2$ and creates a internal phantom edge in $C$.
\end{proof}

\begin{example}\label{example-internalfacets}
The circles $C^{\prime}$ and $C^1_{\mathrm{in}}$ as 
on the right in~\eqref{eq:usual} have $\length(C^{\prime})=6$ 
and $\length(C^1_{\mathrm{in}})=2$, and 
its counterpart $C$ has 
one internal phantom edge after removing the nested circle, 
while is has two internal phantom edges in total. 
Circles $C^{\prime}$ of H or C shape as in~\eqref{eq:CH} have 
$\length(C^{\prime})=6$ respectively $\length(C^{\prime})=10$ 
and their counterparts have two respectively one internal phantom edge.
\end{example}

\subsection{An action of the quantum group \texorpdfstring{$\Udot$}{Udot(glinfty)} and arc diagrams}\label{subsection:qgrouponBK}

We obtain now an action of the quantum group on $\cupcat$ 
(with $\cupcat$ being as at the end of Subsection~\ref{subsec:Hbimod}) via 
the functor
\[
\Phi^{\cupcat}_{\mathrm{Howe}}\colon\Udot\to \cupcat,
\]
given on objects by $\Phi^{\cupcat}_{\mathrm{Howe}}(\vec{k})=\Lambda$ 
(with associated $\Lambda\in\bblock$ as in~\eqref{eq:identification}) 
and on identity morphisms by 
$\Phi^{\cupcat}_{\mathrm{Howe}}(\onek)=\oneinsert{\Lambda}$. For 
the generating morphism $E_i\onek$, $\Phi^{\cupcat}_{\mathrm{Howe}}(E_i\onek)$ 
is the bimodule corresponding to the unique local 
picture in the first row of~\eqref{eqn:basic_moves} that fits 
with $\Lambda$ at position $i$ and $i+1$, or zero if none of 
them fits, $\Phi^{\cupcat}_{\mathrm{Howe}}(F_i\onek)$ is defined 
analogously using the local moves in the second row 
of~\eqref{eqn:basic_moves}. The divided powers $E^{(2)}_i\onek$ 
respectively $F_i^{(2)}\onek$ are sent to the bimodules 
for the empty moves in~\eqref{eqn:basic_moves2} for $2\alpha_i$ 
respectively $-2\alpha_i$, again only if the block fits at 
position $i$ and $i+1$, and to zero otherwise. As in 
Subsection~\ref{subsection:qgroup}, all 
higher divided powers $E_i^{(r)}\onek,F_i^{(r)}\onek$ 
for $r>2$ are sent to zero. 
By our discussion in the previous subsection, we see that $\Phi^{\cupcat}_{\mathrm{Howe}}$ 
is a functor of $\field$-linear categories.

\subsection{The cup basis}\label{subsec:basis}

The goal of this subsection is to define a basis 
of the morphism spaces on the side 
of foams which we call the \textit{cup basis}.

\begin{definition}\label{definition-cupbasis}
Let 
${}_u(\webalg_{\vec{k}})_v$ be as in 
Definition~\ref{definition:webalg2}. Perform the following steps.
\begin{enumerate}[label=(\Roman*)]
\item Label each circle in $uv^{\ast}$ by 
either ``no dot'' or ``dot''. Consider all possibilities of 
labeling the circles in such a way.
\item For each such possibility 
we construct a foam $f\colon\oneinsert{2\omega_{\ell}}\to uv^{\ast}$ 
via:
\begin{itemize}
\item If a circle is in its \textit{easiest form} (i.e. 
only one incoming phantom edge and only one outgoing phantom edge), then 
we locally use 
the foam at the top right in~\eqref{eq:capcupfoams} 
for circles with label ``no dot'' and the foam at the  
bottom right in~\eqref{eq:capcupfoams} 
for circles with label ``dot''. 
\item If a circle is more complicated, then we first 
apply the isomorphisms described in Lemma~\ref{lemma:isoforwebs} 
to reduce the circle to its easiest form, then we proceed 
as before, and finally we rebuild the circle using 
the inverses as in Lemma~\ref{lemma:isoforwebs}. 
Then we move the dot to the rightmost
facet using~\eqref{eq:dotmigration} (the meticulous 
reader will note that this is ill-defined since there could be more than one rightmost facet, 
but one can choose any of them: dots passing from one rightmost facet to another always 
have to move an even number of times across phantom facets, see~\eqref{eq:squareswitch}).
\end{itemize} 
\end{enumerate}
From this we obtain a set of foams 
$\CUP(uv^{\ast})$ which we call 
the \textit{cup basis} of ${}_u(\webalg_{\vec{k}})_v$. 
Basis elements of this basis are called \textit{basis cup foams}.
This name is justified: by construction, 
the foams in ${}_u(\webalg_{\vec{k}})_v$ 
are topological obtained by ``cupping off circles in the evident way''.
Similarly, given $u\in\Hom_{\F}(\vec{k},\vec{l})$ we can define 
a \textit{cup basis} $\CUP(u)$ \textit{of} $\M(u)$ as above with 
the extra condition that we first close the web $u$ in any possible way.
\end{definition}

\begin{lemma}\label{lemma-it-is-a-basis}
Let $u,v\in\CUP(\vec{k})$.
The set $\CUP(uv^{\ast})$ is a homogeneous, $\field$-linear 
basis of the space ${}_u(\webalg_{\vec{k}})_v$.\makeqed
\end{lemma}

\begin{proof}
That $\CUP(uv^{\ast})$ is homogeneous is evident. To show that we get a basis,
we use induction on the number $m$ 
of circles in $uv^{\ast}$. If $m=0$, then 
the statement is clear by Lemma~\ref{lemma:homonedim}. So let 
$m>0$. Choose 
now any maximally nested circle $C$ (a circle that does not have 
any nested components) 
in $uv^{\ast}$ 
(we have to choose such a maximally nested circle in order 
use locally the foams from~\eqref{eq:capcupfoams}). Remove it from $uv^{\ast}$ 
and obtain 
a new web $\tilde u\tilde v^{\ast}$ with fewer number of circles. 
By induction, $\CUP(\tilde u\tilde v^{\ast})$ 
is a basis of ${}_{\tilde u}(\webalg_{\vec{k}})_{\tilde v}$. 
Now apply the isomorphism from~\eqref{eq:capcupwebs} 
and create a new circle by taking two copies of the basis elements 
from $\CUP(\tilde u\tilde v^{\ast})$ (shifted up by one respectively down by one). 
The claim of the lemma follows if $C$ was in its easiest form.
Otherwise, we rebuild the chosen circle $C$ from
the newly created circle using isomorphisms as
\[
\xy
(0,0)*{\reflectbox{\includegraphics[scale=1]{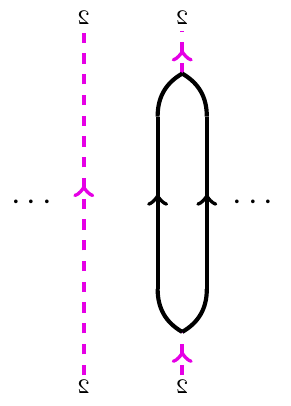}}};
\endxy
\cong
\xy
(0,0)*{\includegraphics[scale=1]{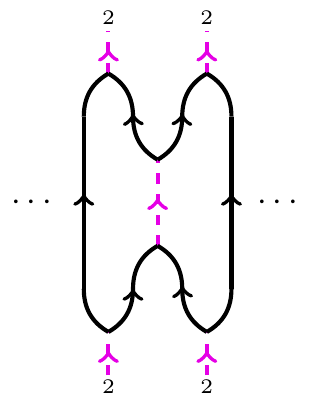}};
\endxy
\cong\;\text{etc.}
\]
from Lemma~\ref{lemma:isoforwebs}. Finally move dots
to the 
rightmost facets using the dot migrations~\eqref{eq:dotmigration}.
Thus, the lemma also follows in this case.
\end{proof}

\begin{lemma}\label{lemma-it-is-a-basis2}
Let $u\in\Hom_{\F}(\vec{k},\vec{l})$. 
The set $\CUP(u)$ is a homogeneous, $\field$-linear 
basis of the web bimodule $\M(u)$.\makeqed
\end{lemma}

\begin{proof}
Similar to the proof of Lemma~\ref{lemma-it-is-a-basis}: 
the induction start is the same and uses Lemma~\ref{lemma:homonedim}. 
Note that $\M(u)$ is obtained from $u$ (which has possibly only cups, caps 
and rays and no circles) by closing the bottom and top 
in any possible way. Thus, for each such ``closure'' of 
$\M(u)$ we can use the argument from above.   
\end{proof}

We now match the cup basis $\CUP(uv^{\ast})$ with the 
basis ${}_{\lambda}\mathbb{B}^{\circ}(\Lambda)_{\mu}$ 
defined in~\eqref{eq:basis2}.

\begin{lemma}\label{lemma:matchvs}
Let $u,v$ be webs such that $u=\web(\lambda)$ and $v=\web(\mu)$.
There is an isomorphism of graded $\field$-vector spaces
\begin{equation}\label{eq:isoalgebras1}
\Iso{uv}{\lambda\mu}\colon{}_u(\webalg_{\vec{k}})_v\to {}_{\lambda}(\Arcalg_{\Lambda})_{\mu}
\end{equation}
which sends $\CUP(uv^{\ast})$ to ${}_{\lambda}\mathbb{B}^{\circ}(\Lambda)_{\mu}$ 
by identifying 
the cup foams without dots with anticlockwise circles and 
the foams with dots with clockwise circles.\makeqed
\end{lemma}

\begin{proof}
The sets $\CUP(uv^{\ast})$ and 
${}_{\lambda}\mathbb{B}^{\circ}(\Lambda)_{\mu}$ are clearly in bijective correspondence. 
Moreover, 
recalling Lemma~\ref{lemma:degree} and the shift 
as in Definition~\ref{definition:webalg2}, 
we obtain that $\Iso{uv}{\lambda\mu}$ is 
homogeneous, which proves the lemma.
\end{proof}

Similarly, 
we match the cup basis $\CUP(u)$ with the 
basis $\mathbb{B}^\circ(\mathbf{\Lambda},\mathbf{t})$ 
from~\eqref{eq:basis-bimodule}:

\begin{lemma}\label{lemma:matchvs2}
Let $u$ be a web such that $u=\web(\mathbf{\Lambda},\mathbf{t})$.
There is a surjection of graded $\field$-vector spaces
\begin{equation*}
\Isoo{uv}{\lambda\mu}\colon \M(u)\to \Arcalg(\mathbf{\Lambda},\mathbf{t})
\end{equation*}
which sends $\CUP(u)$ to $\mathbb{B}^\circ(\mathbf{\Lambda},\mathbf{t})$ 
by identifying 
the basis cup foams without dots with anticlockwise circles and 
the basis cup foams with dots with clockwise circles.\makeqed
\end{lemma}

The statement of Lemma~\ref{lemma:matchvs2} can be easily
strengthen using Lemma~\ref{lemma:basisforwebs}.
(Morally, the web bimodules are infinite-dimensional
in a ``stupid way''.)

\begin{proof}
To show that~\eqref{eq:isoalgebras2} is 
indeed a homogeneous, $\field$-linear surjection 
we can proceed as in the proof of Lemma~\ref{lemma:matchvs} since 
both bases, $\CUP(u)$ and $\mathbb{B}^\circ(\mathbf{\Lambda},\mathbf{t})$, 
are in the end defined by closing $u$ respectively $\mathbf{t}$ 
in all possible ways.
\end{proof}

Let still $\Lambda\in\bblock$. Given $\lambda,\mu\in\Lambda$,
let us denote for $u=\web(\lambda)$ and $v=\web(\mu)$ (and only these)
\begin{equation}\label{eq:newwebalgs}
\webalg^{\circ}_{\vec{k}}=\!\!\!\!\!\bigoplus_{u,v\in\CUP(\vec{k})}\!\!\!\!\!{}_u(\webalg_{\vec{k}})_v\quad
\text{and}
\quad\webalg^{\circ}=\bigoplus_{\vec{k}\in\Y}\webalg_{\vec{k}}.
\end{equation}
The following is a direct consequence of Lemma~\ref{lemma:matchvs}.

\begin{corollary}\label{corollary:matchvs}
For any $\lambda,\mu\in\Lambda$ and $u=\web(\lambda), v=\web(\mu)$: the isomorphisms
$\Iso{uv}{\lambda\mu}$ 
from~\eqref{eq:isoalgebras1} 
extend to isomorphisms of graded, 
$\field$-vector spaces
\begin{equation}\label{eq:isoalgebras2}
\Iso{\vec{k}}{\Lambda}\colon\webalg^{\circ}_{\vec{k}}\to\Arcalg_{\Lambda},\quad\quad
\Iso{}{}\colon\webalg^{\circ}\to\Arcalg.
\end{equation}
where we identify $\vec{k}$ and $\Lambda$ as in~\eqref{eq:identification}.\qedmake
\end{corollary}

\subsection{Proof of the main result}\label{subsec:isoofalgebras}

We deduce now Theorem~\ref{theorem:matchalgebras} 
from the following.

\begin{theorem}\label{proposition:matchalgebras}
The maps from~\eqref{eq:isoalgebras2} are 
isomorphisms of graded algebras.\makeqed
\end{theorem}

Before we prove Theorem~\ref{proposition:matchalgebras} in 
Subsection~\ref{subsec:proofisoofalgebras}, we deduce some 
consequences, e.g. our main result.
Moreover, the identification from Theorem~\ref{proposition:matchalgebras} 
allows us to use topological arguments (i.e. foams) to deduce 
algebraic properties. In particular, we obtain the 
associativity of the Blanchet-Khovanov algebras.

\begin{corollary}\label{corollary:matchalgebras1}
The multiplication rule from Subsection~\ref{sec:multiplication} 
is 
independent of the order in 
which the surgeries are performed. This turns $\Arcalg_{\Lambda}$ into a 
graded, associative, unital algebra.
Similar for (the locally unital) algebra $\Arcalg$.\makeqed
\end{corollary}

\begin{proof}
This follows 
directly from Theorem~\ref{proposition:matchalgebras}
(note that the stated properties are clear if we work with the 
web algebra, see Corollary~\ref{corollary:multweb}).
\end{proof}

\begin{remark}\label{remark:uniquecategorification}
Theorem~\ref{proposition:matchalgebras} leaves the 
question how the Blanchet-Khovanov algebra $\Arcalg_{\Lambda}$ and 
Khovanov's original arc algebra $H_m$ are related (and thus, how the Blanchet foam 
construction 
relates to the Khovanov~\cite{Kho} 
and Bar-Natan~\cite{BN1} theory using cobordisms).
To answer this question, note that the action of $\Udott$ from 
Subsection~\ref{subsection:qgrouponBK} 
extends to a $2$-representation of Khovanov-Lauda's categorification 
of $\Udott$, see~\cite[Proposition~3.3]{LQR}. 
The same holds on the side of Khovanov's arc algebra, 
see~\cite[Remark~5.7]{BS3}. Hence, 
it follows, for suitable choices of $\Lambda$ and $m$, 
that $\Arcalg_{\Lambda}$ and $H_m$ 
are (graded) Morita equivalent. This can be deduced 
from Rouquier's universality theorem, 
see~\cite[Proposition~5.6 and Corollary~5.7]{Rou}, 
compare also to \cite[Proposition~5.18]{MPT}. 
Since $\Arcalg_{\Lambda}$ and $H_m$ are basic algebras, it follows 
by abstract nonsense that $\Arcalg_{\Lambda}$ and $H_m$ are isomorphic algebras
(in fact, as graded algebras). This approach 
however does not provide an explicit isomorphism. 
Such an isomorphism is constructed 
in~\cite[Section~4]{EST2} using a slightly more general 
framework. 
\end{remark}

\begin{remark}\label{remark:calculation}
Theorem~\ref{proposition:matchalgebras} gives 
a way to compute the functorial chain complex 
(which is a link invariant)
defined by Blanchet, see~\cite{Bla}, and all its involved maps. 
Indeed, in our framework $\Arcalg$ 
comes equipped with 
an easy to handle basis and all appearing $\Arcalg$-module homomorphisms 
can explicitly be computed in this basis. This is in contrast to 
the local action used in~\cite[Proposition~3.3]{LQR} 
to define Blanchet's 
link homologies, see~\cite[Subsection~4.1]{LQR}, 
because it is not a priori clear in their formulation how to do explicit
computations (since globally a significant number of non-trivial signs 
come into play).
\end{remark}

Note that we consider webs in $u\in\CUP(\vec{k})$ without 
imposing any relations. On the ``uncategorified level'' in the sense of 
Kuperberg~\cite{Kup}, this has to be modified: denote 
by 
$\CUP(\vec{k})^{\field(q)}_{\mathrm{rel}}=
\langle\Hom(\oneinsert{2\omega},\vec{k})\rangle_{\field(q)}$ 
the $\field(q)$-linear vector space obtained from $\CUP(\vec{k})$ 
by linearization and modding out by the circle removal relation
\[
\xy
(0,0)*{\includegraphics[scale=1]{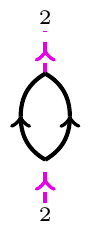}};
\endxy=
(q+q^{-1})\cdot\;
\xy
(0,0)*{\includegraphics[scale=1]{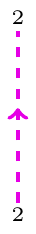}};
\endxy
\]
and isotopy relations as in the Lemmas~\ref{lemma:isoforwebs} 
and~\ref{lemma:isoforwebs1}. Note that, 
on the ``categorified level'' in which we work in the rest of the paper, 
we do not need 
to impose these relations since we ``lift'' them to 
the isomorphisms from Lemmas~\ref{lemma:invertible},~\ref{lemma:isoforwebs} 
and~\ref{lemma:isoforwebs1}.

\begin{lemma}\label{lemma:basisforwebs}
Let $\Lambda\in\bblock$ and $\vec{k}$ be its associated 
element in $\bY$ (see~\eqref{eq:identification}).
Then
\[
\{u\in\CUP(\vec{k})\mid u=\web(\lambda), \lambda\in\Lambda\}
\]
is a $\field(q)$-linear basis of $\CUP(\vec{k})^{\field(q)}_{\mathrm{rel}}$.
\end{lemma}

\begin{proof}
This is clear by the relations imposed on $\CUP(\vec{k})^{\field(q)}_{\mathrm{rel}}$.
\end{proof}

We are finally able to prove our main result.

\begin{proof}[Proof of Theorem~\ref{theorem:matchalgebras}]
Instead of taking all webs $u\in\CUP(\vec{k})$, it suffices to 
take a basis of $\CUP(\vec{k})^{\field(q)}_{\mathrm{rel}}$ and the 
webs $\web(\lambda)$ form such a basis by Lemma~\ref{lemma:basisforwebs}. 
Concretely, 
the algebras $\webalg_{\vec{k}}$ (all webs) and 
$\webalg^{\circ}_{\vec{k}}$ 
(only basis webs) are graded Morita equivalent (this 
can be seen similar to~\cite[Lemma~7.5]{Mack}) 
and the statement follows from Theorem~\ref{proposition:matchalgebras}. 
The identification of the bimodules as graded $\field$-vector spaces is clear 
by Lemma~\ref{lemma:matchvs2}, while the actions 
agree by Theorem~\ref{proposition:matchalgebras} and construction of the actions.
\end{proof}

\subsection{The proof of the graded isomorphism}\label{subsec:proofisoofalgebras}

We now prove Theorem~\ref{proposition:matchalgebras}.

\begin{proof}[Proof of Theorem~\ref{proposition:matchalgebras}]
By Lemma~\ref{lemma:matchvs}, it suffices to show that $\Iso{\vec{k}}{\Lambda}$ is 
a homomorphism of algebras (since then so is $\Iso{}{}$ as well).
For this purpose, 
fix $\lambda,\mu,\nu\in\Lambda$ and set $u=\web(\lambda)$, $v=\web(\mu)$, and $w=\web(\nu)$.
We show that any product 
of two basis cup foams 
$f\in\CUP(uv^{\ast})$ and $g\in\CUP(vw^{\ast})$ 
satisfies 
\begin{equation*}
\Iso{uw}{\lambda\nu}(fg)=\Iso{uv}{\lambda\mu}(f)\Iso{v^{\prime}w}{\mu^{\prime}\nu}(g),\quad v=v^{\prime},\mu=\mu^{\prime}
\end{equation*} 
(This is enough since all 
non-zero multiplications on the side of $\Arcalg_{\Lambda}$ 
as well as on the side of $\webalg_{\vec{k}}$ 
satisfy $v=v^{\prime}$ and $\mu=\mu^{\prime}$ by 
definition.)

In order to do so, we show that 
each step in the multiplication 
procedure from Definition~\ref{definition:webalg2} locally agrees 
with the one from Subsection~\ref{sec:multiplication}.
There are four
topological different situations to check (compare with the cases in Subsection~\ref{sec:multiplication} introducing the different signs):

\begin{enumerate}[label=(\roman*)]
\item \textbf{Non-nested merge.} Two non-nested circles are replaced by one circle.
\item \textbf{Nested merge.} Two nested circles are replaced by one circle.
\item \textbf{Non-nested split.} One circle is replaced by two non-nested circles.
\item \textbf{Nested split.} One circle is replaced by two nested circles.
\end{enumerate}
We will consider these four cases step-by-step and compare the corresponding multiplications rules. In addition to the four 
cases we will further distinguish the following shapes 
of the involved underlying webs:

\begin{enumerate}[label=(\Alph*)]
\item \textbf{Basic shape.} The involved 
components are as small as possible with 
the minimal number of phantom facets.
\item \textbf{Minimal saddle.} While the 
components themselves are allowed to be 
of any shape, the involved saddle 
only includes a single phantom facet.
\item \textbf{General case.} Both, the 
shape as well as the saddle, are arbitrary.
\end{enumerate}

We start by comparing the two multiplication rules for the basic shapes first. This will only involve simplified 
versions of the \textit{dot moving signs}.
The basic shapes for (i), (ii), (iii) and (iv) are the following 
(with the multiplication step taking place 
in the marked region)
\begin{gather}\label{eq:usual}
\xy
(0,0)*{\includegraphics[scale=1]{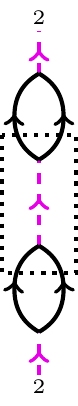}};
\endxy \; \;
\leftrightsquigarrow \; \;
\xy
(0,0)*{\includegraphics[scale=1]{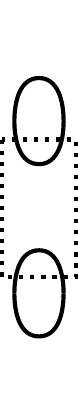}};
\endxy
\quad\text{and}\quad
\xy
(0,0)*{\includegraphics[scale=1]{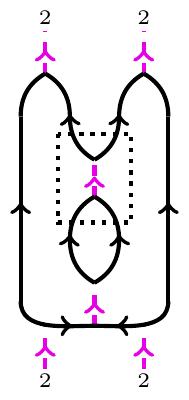}};
\endxy 
\leftrightsquigarrow \; \;
\xy
(0,0)*{\includegraphics[scale=1]{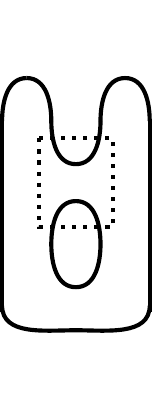}};
\endxy
\end{gather}
in cases (i) and (ii), and the following
H-shape and C-shape
\begin{gather}\label{eq:CH}
\xy
(0,0)*{\includegraphics[scale=1]{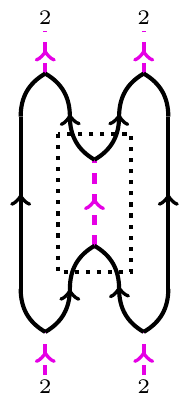}};
\endxy
\leftrightsquigarrow\;\;
\xy
(0,0)*{\includegraphics[scale=1]{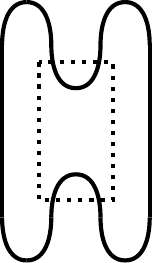}};
\endxy
\quad\text{and}\quad
\xy
(0,0)*{\includegraphics[scale=1]{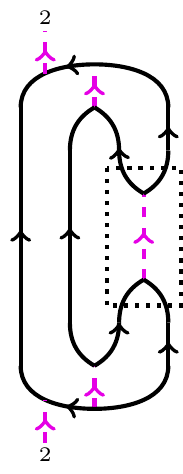}};
\endxy \; \;
\leftrightsquigarrow\;\;
\xy
(0,0)*{\includegraphics[scale=1]{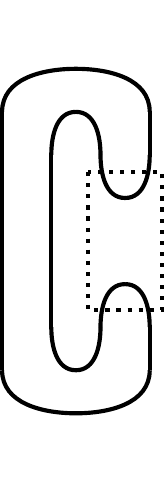}};
\endxy
\end{gather}
in cases (iii) and (iv). Here we always display both, 
the web as well as its corresponding arc diagram.\\[0.2cm]
\noindent
$\blacktriangleright$ \textbf{Non-nested merge - basic shape.} 
In this case, we merge two simple cup webs on the web side and two oriented circles on the arc diagram side. The following table gives the multiplication results in the four possible orientation combinations.
\begin{center}
\begin{tabular}{|| c || c | c || c@{\hspace{4pt}} || c || c | c ||}
\hhline{|t:===:t|~|t:===:t|}
 $\webalg^{\circ}_{\vec{k}}$  & $\begin{tikzpicture}[anchorbase,scale=.25]
\fill [nonecolor, opacity=0.3] (1,1) rectangle (7,-3);
\draw [very thick, white, fill=white] (2.5,1) to [out=315, in=180] (4,.5) to [out=0, in=225] (5.5,1) to (2.5,1) to [out=270, in=180] (4,-1) to [out=0, in=270] (5.5,1);
\draw [very thick, directed=0.55, dashed, nonecolor] (1,1) to (2.5,1);
\draw [very thick, directed=0.55, dashed, nonecolor] (5.5,1) to (7,1);
\draw [very thick, directed=0.55, dashed, nonecolor] (1,-3) to (7,-3);
\draw [very thick, directed=0.55, mycolor] (2.5,1) to [out=270, in=180] (4,-1) to [out=0, in=270] (5.5,1);
\draw [very thick, directed=0.55] (2.5,1) to [out=45, in=180] (4,1.5) to [out=0, in=135] (5.5,1);
\draw [very thick, directed=0.55] (2.5,1) to [out=315, in=180] (4,.5) to [out=0, in=225] (5.5,1);
\draw [thick] (1,1) to (1,-3);
\draw [thick] (7,1) to (7,-3);
\node at (4,1.5) {$\phantom{a}$};
\node at (4,-3.5) {$\phantom{a}$};
\end{tikzpicture}$ & $\begin{tikzpicture}[anchorbase,scale=.25]
	\fill [nonecolor, opacity=0.3] (1,1) rectangle (7,-3);
    \draw [very thick, white, fill=white] (2.5,1) to [out=315, in=180] (4,.5) to [out=0, in=225] (5.5,1) to (2.5,1) to [out=270, in=180] (4,-1) to [out=0, in=270] (5.5,1);
	\draw [very thick, directed=0.55, dashed, nonecolor] (1,1) to (2.5,1);
	\draw [very thick, directed=0.55, dashed, nonecolor] (5.5,1) to (7,1);
	\draw [very thick, directed=0.55, dashed, nonecolor] (1,-3) to (7,-3);
	\draw [very thick, directed=0.55, mycolor] (2.5,1) to [out=270, in=180] (4,-1) to [out=0, in=270] (5.5,1);
	\draw [very thick, directed=0.55] (2.5,1) to [out=45, in=180] (4,1.5) to [out=0, in=135] (5.5,1);
	\draw [very thick, directed=0.55] (2.5,1) to [out=315, in=180] (4,.5) to [out=0, in=225] (5.5,1);
	\draw [thick] (1,1) to (1,-3);
	\draw [thick] (7,1) to (7,-3);
	\node at (4,0) {\tiny $\bullet$};
	\node at (4,1.5) {$\phantom{a}$};
	\node at (4,-3.5) {$\phantom{a}$};
\end{tikzpicture}$  & &
 $\Arcalg_{\Lambda}$  & $\begin{tikzpicture}[anchorbase, very thick]
\draw (0,0) .. controls +(0,.75) and +(0,.75) .. +(1,0);
\node at (0,0) {\huge $\downb$};
\node at (1,0) {\huge $\upb$};
\draw (0,0) .. controls +(0,-.75) and +(0,-.75) .. +(1,0);
\node at (0,.6085) {$\phantom{a}$};
\node at (0,-.6085) {$\phantom{a}$};
\end{tikzpicture}$ & $\begin{tikzpicture}[anchorbase, very thick]
\draw (0,0) .. controls +(0,.75) and +(0,.75) .. +(1,0);
\node at (0,0) {\huge $\upb$};
\node at (1,0) {\huge $\downb$};
\draw (0,0) .. controls +(0,-.75) and +(0,-.75) .. +(1,0);
\node at (0,.6085) {$\phantom{a}$};
\node at (0,-.6085) {$\phantom{a}$};
\end{tikzpicture}$\\
\hhline{||=#==||~||=#==||}
  $\begin{tikzpicture}[anchorbase,scale=.25]
\fill [nonecolor, opacity=0.3] (1,1) rectangle (7,-3);
\draw [very thick, white, fill=white] (2.5,1) to [out=315, in=180] (4,.5) to [out=0, in=225] (5.5,1) to (2.5,1) to [out=270, in=180] (4,-1) to [out=0, in=270] (5.5,1);
\draw [very thick, directed=0.55, dashed, nonecolor] (1,1) to (2.5,1);
\draw [very thick, directed=0.55, dashed, nonecolor] (5.5,1) to (7,1);
\draw [very thick, directed=0.55, dashed, nonecolor] (1,-3) to (7,-3);
\draw [very thick, directed=0.55, mycolor] (2.5,1) to [out=270, in=180] (4,-1) to [out=0, in=270] (5.5,1);
\draw [very thick, directed=0.55] (2.5,1) to [out=45, in=180] (4,1.5) to [out=0, in=135] (5.5,1);
\draw [very thick, directed=0.55] (2.5,1) to [out=315, in=180] (4,.5) to [out=0, in=225] (5.5,1);
\draw [thick] (1,1) to (1,-3);
\draw [thick] (7,1) to (7,-3);
\node at (4,1.5) {$\phantom{a}$};
\node at (4,-3.5) {$\phantom{a}$};
\end{tikzpicture}$ & $\begin{tikzpicture}[anchorbase,scale=.25]
\fill [nonecolor, opacity=0.3] (1,1) rectangle (7,-3);
\draw [very thick, white, fill=white] (2.5,1) to [out=315, in=180] (4,.5) to [out=0, in=225] (5.5,1) to (2.5,1) to [out=270, in=180] (4,-1) to [out=0, in=270] (5.5,1);
\draw [very thick, directed=0.55, dashed, nonecolor] (1,1) to (2.5,1);
\draw [very thick, directed=0.55, dashed, nonecolor] (5.5,1) to (7,1);
\draw [very thick, directed=0.55, dashed, nonecolor] (1,-3) to (7,-3);
\draw [very thick, directed=0.55, mycolor] (2.5,1) to [out=270, in=180] (4,-1) to [out=0, in=270] (5.5,1);
\draw [very thick, directed=0.55] (2.5,1) to [out=45, in=180] (4,1.5) to [out=0, in=135] (5.5,1);
\draw [very thick, directed=0.55] (2.5,1) to [out=315, in=180] (4,.5) to [out=0, in=225] (5.5,1);
\draw [thick] (1,1) to (1,-3);
\draw [thick] (7,1) to (7,-3);
\node at (4,1.5) {$\phantom{a}$};
\node at (4,-3.5) {$\phantom{a}$};
\end{tikzpicture}$ & $\begin{tikzpicture}[anchorbase,scale=.25]
	\fill [nonecolor, opacity=0.3] (1,1) rectangle (7,-3);
    \draw [very thick, white, fill=white] (2.5,1) to [out=315, in=180] (4,.5) to [out=0, in=225] (5.5,1) to (2.5,1) to [out=270, in=180] (4,-1) to [out=0, in=270] (5.5,1);
	\draw [very thick, directed=0.55, dashed, nonecolor] (1,1) to (2.5,1);
	\draw [very thick, directed=0.55, dashed, nonecolor] (5.5,1) to (7,1);
	\draw [very thick, directed=0.55, dashed, nonecolor] (1,-3) to (7,-3);
	\draw [very thick, directed=0.55, mycolor] (2.5,1) to [out=270, in=180] (4,-1) to [out=0, in=270] (5.5,1);
	\draw [very thick, directed=0.55] (2.5,1) to [out=45, in=180] (4,1.5) to [out=0, in=135] (5.5,1);
	\draw [very thick, directed=0.55] (2.5,1) to [out=315, in=180] (4,.5) to [out=0, in=225] (5.5,1);
	\draw [thick] (1,1) to (1,-3);
	\draw [thick] (7,1) to (7,-3);
	\node at (4,0) {\tiny $\bullet$};
	\node at (4,1.5) {$\phantom{a}$};
	\node at (4,-3.5) {$\phantom{a}$};
\end{tikzpicture}$ & &
  $\begin{tikzpicture}[anchorbase, very thick]
\draw (0,0) .. controls +(0,.75) and +(0,.75) .. +(1,0);
\node at (0,0) {\huge $\downb$};
\node at (1,0) {\huge $\upb$};
\draw (0,0) .. controls +(0,-.75) and +(0,-.75) .. +(1,0);
\node at (0,.6085) {$\phantom{a}$};
\node at (0,-.6085) {$\phantom{a}$};
\end{tikzpicture}$ & $\begin{tikzpicture}[anchorbase, very thick]
\draw (0,0) .. controls +(0,.75) and +(0,.75) .. +(1,0);
\node at (0,0) {\huge $\downb$};
\node at (1,0) {\huge $\upb$};
\draw (0,0) .. controls +(0,-.75) and +(0,-.75) .. +(1,0);
\node at (0,.6085) {$\phantom{a}$};
\node at (0,-.6085) {$\phantom{a}$};
\end{tikzpicture}$ & $\begin{tikzpicture}[anchorbase, very thick]
\draw (0,0) .. controls +(0,.75) and +(0,.75) .. +(1,0);
\node at (0,0) {\huge $\upb$};
\node at (1,0) {\huge $\downb$};
\draw (0,0) .. controls +(0,-.75) and +(0,-.75) .. +(1,0);
\node at (0,.6085) {$\phantom{a}$};
\node at (0,-.6085) {$\phantom{a}$};
\end{tikzpicture}$\\
\hhline{||---||~||---||}
  $\begin{tikzpicture}[anchorbase,scale=.25]
	\fill [nonecolor, opacity=0.3] (1,1) rectangle (7,-3);
    \draw [very thick, white, fill=white] (2.5,1) to [out=315, in=180] (4,.5) to [out=0, in=225] (5.5,1) to (2.5,1) to [out=270, in=180] (4,-1) to [out=0, in=270] (5.5,1);
	\draw [very thick, directed=0.55, dashed, nonecolor] (1,1) to (2.5,1);
	\draw [very thick, directed=0.55, dashed, nonecolor] (5.5,1) to (7,1);
	\draw [very thick, directed=0.55, dashed, nonecolor] (1,-3) to (7,-3);
	\draw [very thick, directed=0.55, mycolor] (2.5,1) to [out=270, in=180] (4,-1) to [out=0, in=270] (5.5,1);
	\draw [very thick, directed=0.55] (2.5,1) to [out=45, in=180] (4,1.5) to [out=0, in=135] (5.5,1);
	\draw [very thick, directed=0.55] (2.5,1) to [out=315, in=180] (4,.5) to [out=0, in=225] (5.5,1);
	\draw [thick] (1,1) to (1,-3);
	\draw [thick] (7,1) to (7,-3);
	\node at (4,0) {\tiny $\bullet$};
	\node at (4,1.5) {$\phantom{a}$};
	\node at (4,-3.5) {$\phantom{a}$};
\end{tikzpicture}$ & $\begin{tikzpicture}[anchorbase,scale=.25]
	\fill [nonecolor, opacity=0.3] (1,1) rectangle (7,-3);
    \draw [very thick, white, fill=white] (2.5,1) to [out=315, in=180] (4,.5) to [out=0, in=225] (5.5,1) to (2.5,1) to [out=270, in=180] (4,-1) to [out=0, in=270] (5.5,1);
	\draw [very thick, directed=0.55, dashed, nonecolor] (1,1) to (2.5,1);
	\draw [very thick, directed=0.55, dashed, nonecolor] (5.5,1) to (7,1);
	\draw [very thick, directed=0.55, dashed, nonecolor] (1,-3) to (7,-3);
	\draw [very thick, directed=0.55, mycolor] (2.5,1) to [out=270, in=180] (4,-1) to [out=0, in=270] (5.5,1);
	\draw [very thick, directed=0.55] (2.5,1) to [out=45, in=180] (4,1.5) to [out=0, in=135] (5.5,1);
	\draw [very thick, directed=0.55] (2.5,1) to [out=315, in=180] (4,.5) to [out=0, in=225] (5.5,1);
	\draw [thick] (1,1) to (1,-3);
	\draw [thick] (7,1) to (7,-3);
	\node at (4,0) {\tiny $\bullet$};
	\node at (4,1.5) {$\phantom{a}$};
	\node at (4,-3.5) {$\phantom{a}$};
\end{tikzpicture}$ & 0 & &
  $\begin{tikzpicture}[anchorbase, very thick]
\draw (0,0) .. controls +(0,.75) and +(0,.75) .. +(1,0);
\node at (0,0) {\huge $\upb$};
\node at (1,0) {\huge $\downb$};
\draw (0,0) .. controls +(0,-.75) and +(0,-.75) .. +(1,0);
\node at (0,.6085) {$\phantom{a}$};
\node at (0,-.6085) {$\phantom{a}$};
\end{tikzpicture}$ & $\begin{tikzpicture}[anchorbase, very thick]
\draw (0,0) .. controls +(0,.75) and +(0,.75) .. +(1,0);
\node at (0,0) {\huge $\upb$};
\node at (1,0) {\huge $\downb$};
\draw (0,0) .. controls +(0,-.75) and +(0,-.75) .. +(1,0);
\node at (0,.6085) {$\phantom{a}$};
\node at (0,-.6085) {$\phantom{a}$};
\end{tikzpicture}$ & 0\\
\hhline{|b:===:b|~|b:===:b|}
\end{tabular}
\end{center}
To obtain this table one argues as follows: merging 
two basis cup foams via a saddle creates a 
new basis cup foam. Thus, only the position 
of the dot matters if we rewrite the 
result in the cup basis. If 
there is no dot or there are two dots on the new 
cup foam, then we are done (the latter follows from~\eqref{eq:theusualrelations1}). 
If there is only one dot note that it 
is automatically on the rightmost facet and 
we are done as well (no signs). 
This is precisely as in~\eqref{eq:mult1}.\qedmake\\[0.2cm]
\noindent
$\blacktriangleright$ \textbf{Nested merge - basic shape.}
This step in case of $\Arcalg_{\Lambda}$ was 
calculated in~\eqref{eq:mult2}.
In case of $\webalg^{\circ}_{\vec{k}}$ this is given by stacking the saddle displayed below on top of a given foam (the second foam displayed below
is shown to illustrate the cylinder we want to cut).
\[
\xy
\xymatrix@C+=1.3cm@L+=6pt{
\xy
(0,0)*{\includegraphics[scale=1]{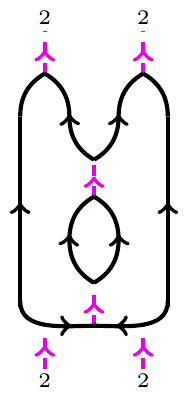}};
\endxy \; \ar[rr]^{
\xy
(0,0)*{\includegraphics[scale=1]{figs/fig26.pdf}};
\endxy
} & & \; 
\xy
(0,0)*{\includegraphics[scale=1]{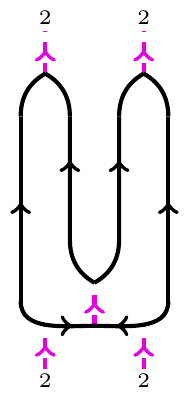}};
\endxy \; \ar[rr]^{
\xy
(0,0)*{\includegraphics[scale=1]{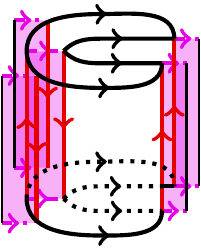}};
\endxy
}
& & \;
\xy
(0,0)*{\includegraphics[scale=1]{figs/fig66.pdf}};
\endxy
}
\endxy
\]
Note now that the difference to the non-nested merge above
is that, if a basis cup foam 
is sitting underneath the leftmost picture, then the end result is 
topological not a basis cup foam. In order to turn the result into a basis cup foam, we apply~\eqref{eq:neckcut} 
to the cylinder 
illustrated above. Here we have to use~\eqref{eq:squeezing} 
first, which gives an overall sign. After neck cutting the 
cylinder 
we create a ``bubble'' (recalling that a basis cup foam is sitting underneath) 
with two internal phantom 
facets in the bottom part of the picture. 
By~\eqref{eq:neckcutphantom}, we can remove the 
phantom facets with the cost of a sign and create an ``honest'' bubble instead. Thus, 
by~\eqref{eq:bubble2}, only the term 
in~\eqref{eq:neckcut} with the dot on the bottom 
survives (with a sign). By~\eqref{eq:bubble1} 
the remaining bubble evaluates to $-1$. Hence, we get in total four overall signs 
which is the same as no extra sign. 
The dots behave as in the table 
above, since before the neck cut we can move any of them to the top and thus they do not interfere with the above procedure. Thus,
using~\eqref{eq:dotmigration}, 
we get the same result as in~\eqref{eq:mult2}.\qedmake\newpage
\noindent
$\blacktriangleright$ \textbf{Non-nested split - basic shape.} 
This step in case of $\Arcalg_{\Lambda}$ was calculated in~\eqref{eq:mult3}. 
For $\webalg^{\circ}_{\vec{k}}$ the multiplication is 
given by the composition of the following foams.
\[
\xy
\xymatrix@C+=1.3cm@L+=6pt{
\xy
(0,0)*{\includegraphics[scale=1]{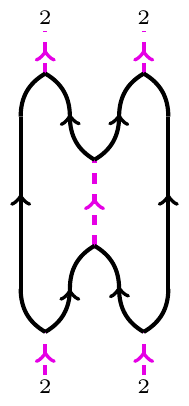}};
\endxy \;
\ar[rr]^{\xy
(0,0)*{\includegraphics[scale=1]{figs/fig26.pdf}};
\endxy
}
& & \;
\xy
(0,0)*{\includegraphics[scale=1]{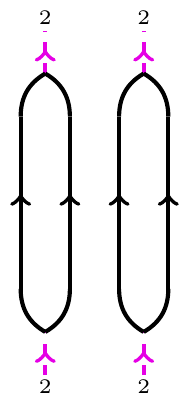}};
\endxy \;
\ar[rr]^{
\xy
(0,0)*{\includegraphics[scale=1]{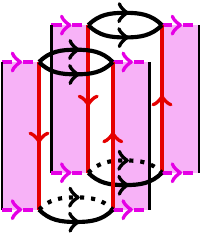}};
\endxy
}
& & \;
\xy
(0,0)*{\includegraphics[scale=1]{figs/fig70.pdf}};
\endxy
}
\endxy
\] 
As before, we can use the neck cutting~\eqref{eq:neckcut} on the 
cylinder that corresponds to the right circle (we could also choose 
the other one and slightly change the steps below). The 
result is precisely as in~\eqref{eq:mult3}: 
\begin{itemize}
\item If the original basis cup foams sitting underneath has no dots, then 
the one with the dot on the rightmost circle gets no sign (the dot is 
automatically on the rightmost facet), and the other does not as well (the dot 
from the neck cutting needs to pass
one phantom facets to move to the right).
\item If the original basis cup foams sitting underneath has already a dot, then 
only the positive term in~\eqref{eq:neckcut} survives and the 
dots are already in the rightmost positions.
\item In both cases, the resulting foam is 
topological not a basis cup foam, but using~\eqref{eq:neckcutphantom} once 
reduces it to 
a basis cup foam, giving an overall sign.\qedmake
\end{itemize}
\noindent
$\blacktriangleright$ \textbf{Nested split - basic shape.} 
This step in case of $\Arcalg_{\Lambda}$ 
was calculated in~\eqref{eq:mult4}.
In case of $\webalg^{\circ}_{\vec{k}}$ we again give the composite of foams that we stack on top of each other for the multiplication.
\[
\begin{xy}
\xymatrix@C+=1.3cm@L+=6pt{
\xy
(0,0)*{\includegraphics[scale=1]{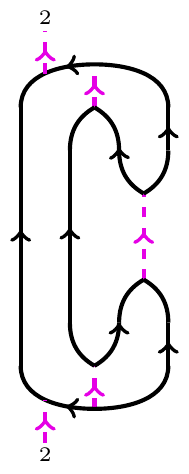}};
\endxy \; \ar[rr]^{
\xy
(0,0)*{\includegraphics[scale=1]{figs/fig26.pdf}};
\endxy
}& & \;
\xy
(0,0)*{\includegraphics[scale=1]{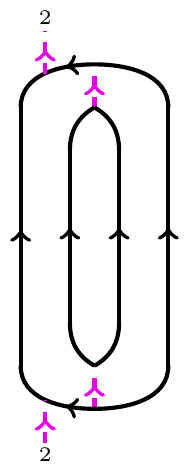}};
\endxy \;
\ar[rr]^{
\xy
(0,0)*{\includegraphics[scale=1]{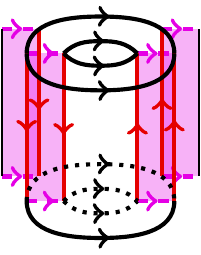}};
\endxy}& & \;
\xy
(0,0)*{\includegraphics[scale=1]{figs/fig73.pdf}};
\endxy
}
\end{xy}
\]
Again we can apply neck cutting. This time to the internal cylinder in the second foam between the middle web and the rightmost 
web connecting the two nested circles that we 
can cut using~\eqref{eq:neckcut}. The 
result is precisely as in~\eqref{eq:mult4}: 
\begin{itemize}
\item If the original basis cup foams sitting underneath has no dots, then 
the one with the dot on the non-nested circle gets a sign (the dot is 
automatically on the rightmost facet), while the other does not (the dot 
created in the neck cutting needs to pass
two phantom facets to get to the right).
\item If the original basis cup foams sitting underneath has already a dot, then 
only the positive term in~\eqref{eq:neckcut} survives and the 
dots are already in the rightmost positions.
\item In both cases, the resulting foam is already 
topological a cup basis foam 
and nothing needs to be done anymore.\qedmake
\end{itemize}

All other situations (i.e. general shapes) are similar 
to the ones discussed, but with the main difference that 
dots need to be shifted to the rightmost facets, some 
(intermediate) results might not be in the topological form of 
a basis cup foam and one needs to 
take care of saddles with a possible number of 
additional internal phantom facets as in~\eqref{eq:ssaddle}. These three 
facts together explain the signs turning up in the multiplication 
described in Subsection~\ref{sec:multiplication} with the first one 
corresponding to the dot moving sign as in~\eqref{eq:dotsign}, the 
second corresponding to the topological signs as in~\eqref{eq:topsign} 
and the latter to the saddle sign as in~\eqref{eq:topsign}.

Thus, we next deal with the minimal saddle situation as 
for example in~\eqref{eq:saddleconvention}. 
In those cases $s(\gamma)=1$ on the 
side of $\Arcalg_{\Lambda}$. This will 
introduce the \textit{topological signs} and 
simplified versions (for $s(\gamma)=1$) of the \textit{saddle signs}.\\[0.2cm]
\noindent
$\blacktriangleright$ \textbf{Non-nested merge - minimal saddle.} 
This is topologically the same as in 
\textit{non-nested merge - basic shape}, 
since the resulting foam will already be a 
basis cup foam. Thus, the only sign comes 
from moving 
dots to the right which matches 
the dot moving sign  
turning up in the multiplication on the 
side of $\Arcalg_{\Lambda}$ in \eqref{eq:dotsign}.\qedmake\\[0.2cm]
\noindent
$\blacktriangleright$ \textbf{Nested merge - minimal saddle.}
Assume now that we merge a circle $C^{\mathrm{out}}$ 
with some nested circle $C^{\mathrm{in}}$ inside of it.

The dot moving is as above in \textit{non-nested merge - minimal saddle} and 
gives the same sign as for $\Arcalg_{\Lambda}$. The difference to the nested merge in the basic shape is that we have 
to bring the resulting foam in 
the topological form of a basis cup foam. To this end, 
we can proceed as in \textit{nested merge - basic shape} 
and cut the same cylinder as there. We first note that neither the dots which are 
already on the foam sitting underneath 
nor the internal circles in $C^{\mathrm{out}}$ 
which are different from $C^{\mathrm{in}}$ matter: we can 
topologically move them ``away 
from the local picture''.

Thus, to simplify a little bit, assume that $C^{\mathrm{in}}$ is the only 
circle nested in $C^{\mathrm{out}}$ and there are no dots.
Following the procedure given as in \textit{nested merge - basic shape} above, 
we see that the only things of importance are signs that come from cutting 
the cylinder with possible internal phantom facets and evaluating 
the ``bubble'' with possible more 
than one internal phantom facets. Indeed, what matters is 
the number of times we need to 
use~\eqref{eq:squeezing} in the cutting 
procedure of the cylinder and the number of times we need to 
use~\eqref{eq:neckcutphantom} in the 
bursting of the ``bubble'' (the rest stays the same as before in \textit{nested merge - basic shape}). Now, the number of times we need to 
apply~\eqref{eq:squeezing} is $\mathrm{ipe}(C^{\mathrm{out}})-1$ 
(the $-1$ comes in 
because we apply a saddle which removes one of the internal phantom 
edges of the starting picture) 
and the number of times we need to 
apply~\eqref{eq:neckcutphantom} is $\mathrm{ipe}(C^{\mathrm{out}}-C^{\mathrm{in}})$. 
By Lemma~\ref{lemma-internalfacets}, we obtain that
\begin{gather}\label{eq:formula}
\begin{aligned}
(-1)^{\mathrm{ipe}(C^{\mathrm{out}})-1+\mathrm{ipe}(C^{\mathrm{out}}-C^{\mathrm{in}})}
&=
(-1)^{\tfrac{1}{4}(2(\length(C_{\mathrm{out}})-2)+\length(C_{\mathrm{in}})-2)}\\
&=-(-1)^{\tfrac{1}{4}(\length(C_{\mathrm{in}})-2)}\cdot(-1)^1,
\end{aligned}
\end{gather}
where $C_{\mathrm{out}}$ and $C_{\mathrm{in}}$ are the cup diagram counterparts 
of $C^{\mathrm{out}}$ and $C^{\mathrm{in}}$. 
This is precisely the same sign turning up on the side of $\Arcalg_{\Lambda}$ (compare to \eqref{eq:topsign} with $s_\Lambda(\gamma)=1$).

The case where $C^{\mathrm{out}}$ has several 
nested components, is similar since all nested components of 
$C^{\mathrm{out}}$, which are not $C^{\mathrm{in}}$, increase the number of 
times we need to use~\eqref{eq:squeezing} in the same way 
as the number of 
times we need to use~\eqref{eq:neckcutphantom} (hence, no change modulo $2$). 
Again, this matches the side of $\Arcalg_{\Lambda}$ in \eqref{eq:topsign}.\qedmake\\[0.2cm]
\noindent
$\blacktriangleright$ \textbf{Non-nested split - minimal saddle.}
The dot moving is as 
before.
Observe now that we do not have an extra sign turning up 
although the resulting foam is not in the topological shape 
of a basis cup foam. To see this, 
we first simplify by assuming that the circle $C$ which 
is split does not contain any nested components. 
Using the same cutting as in 
\textit{non-nested split - basic shape}, we have signs coming 
from squeezing cylinders and simplifying ``bubbles'' 
(similar as above in \textit{nested merge - minimal saddle}). But the 
number of times we need to apply~\eqref{eq:squeezing} 
in this case is now the same as the number of times 
one has to apply~\eqref{eq:neckcutphantom}, namely $\mathrm{ipe}(C)-1$. 
Thus, again no change modulo $2$. The case with nested components in $C$ is 
now analogously as above in 
\textit{nested merge - minimal saddle} since we can move dots and 
nested circles ``away''. As before, this
increases the number of 
times we need to use~\eqref{eq:squeezing} in the same way 
as the number of 
times we need to use~\eqref{eq:neckcutphantom} (hence, no change mod $2$). 
Thus, we do not get an extra overall sign as 
in case of $\Arcalg_{\Lambda}$ (see Subsection~\ref{sec:multsigns} non-nested split case).\qedmake\\[0.2cm]
\noindent
$\blacktriangleright$ \textbf{Nested split - minimal saddle.}
Moving dots is again as before. 
Furthermore, again, as in \textit{nested merge - minimal saddle}, 
we need to topologically 
manipulate the resulting foam until it is in basis cup foam shape. 
We can proceed as before in \textit{nested split - basic shape} and, 
similar as above in \textit{nested merge - minimal saddle}, we pick up signs coming 
from cylinder cuts and bubble removals. In fact, 
the total sign 
can be calculated analogously as in 
\textit{nested merge - minimal saddle} (and is the same as there). 
Again, this matches the side of $\Arcalg_{\Lambda}$ (see Subsection~\ref{sec:multsigns} nested split case).\qedmake \\[0.2cm]
\noindent
$\blacktriangleright$ \textbf{Non-nested merge - general case.} 
In fact, nothing changes 
compared to the discussion 
in \textit{non-nested merge - minimal saddle}, since 
dots passing a saddle always pass an odd number of phantom facets 
(compare to~\eqref{eq:ssaddle}) and the resulting foams are 
topological already basis cup foams.\qedmake \\[0.2cm]
\noindent
$\blacktriangleright$ \textbf{Nested merge - general case.}
The dot moving stays as before. 
The only thing that changes in 
contrast to \textit{nested merge - minimal saddle}
is that we obtain
\begin{gather*}
\begin{aligned}
(-1)^{\mathrm{ipe}(C^{\mathrm{out}})-s+\mathrm{ipe}(C^{\mathrm{out}}-C^{\mathrm{in}})}
&=
(-1)^{\tfrac{1}{4}(2(\length(C_{\mathrm{out}})-2)+\length(C_{\mathrm{in}})-2)}\cdot(-1)^{s(\gamma)-1}\\
&=-(-1)^{\tfrac{1}{4}(\length(C_{\mathrm{in}})-2)}\cdot(-1)^{s(\gamma)}.
\end{aligned}
\end{gather*}
within the topological re-writing procedure 
instead of the formula from~\eqref{eq:formula}. Thus, this matches 
the side of $\Arcalg_{\Lambda}$ (see \eqref{eq:topsign}).\qedmake \\[0.2cm]
\noindent
$\blacktriangleright$ \textbf{Non-nested split - general case.} Again, the dot moving stays as before. 
The difference to \textit{non-nested split - minimal saddle} 
is that we have to apply~\eqref{eq:neckcutphantom} $s$-times 
instead of once, which gives the 
sign turning up for $\Arcalg_{\Lambda}$ (see Subsection~\ref{sec:multsigns} non-nested split case).\qedmake \\[0.2cm]
\noindent
$\blacktriangleright$ \textbf{Nested split - general case.}
There is no difference 
to the arguments given in \textit{nested split - minimal saddle}.
Again, this matches the side of $\Arcalg_{\Lambda}$ (see Subsection~\ref{sec:multsigns} nested split case).\qedmake  \\[0.2cm]
\noindent
Hence, in each case the topological multiplication agrees with the algebraically 
defined one.
This concludes the proof.
\end{proof}